\documentclass[a4paper,11pt,twoside]{article}
\usepackage[left=2.5cm,right=2cm,top=2cm,bottom=2cm]{geometry}

\usepackage{amsmath,amsfonts,amssymb,amsthm,array}
\usepackage{amsmath}
\usepackage{algorithm}
\usepackage{caption}
\usepackage{subcaption}
\usepackage{algpseudocode,algorithm,algorithmicx}
\usepackage{bm}
\usepackage{color}
\usepackage{graphicx}

\newcommand{\Exp}{\mathbb{E}}

\newcommand{\E}[1]{{\mathbb{E}\left[#1\right] }}    

\newcommand{\Prob}{\mathbb{P}}
\newcommand{\R}{\mathbb{R}}

\algrenewcommand\algorithmicrequire{\textbf{Input:}}
\algrenewcommand\algorithmicensure{\textbf{Initialize:}}

\usepackage{multirow}

\newcommand{\bA}{\mathbf{A}}
\newcommand{\bB}{\mathbf{B}}

\newcommand{\bI}{\mathbf{I}}

\newcommand{\bM}{\mathbf{M}}

\newcommand{\bP}{\mathbf{P}}

\newcommand{\bS}{\mathbf{S}}

\newcommand{\bW}{\mathbf{W}}
\newcommand{\bZ}{\mathbf{Z}}

\newcommand{\eqdef}{:=}

\newcommand{\cD}{{\cal D}}

\newcommand{\cL}{{\cal L}}

\newcommand{\cQ}{{\cal Q}}

\newcommand{\cX}{{\cal X}}

\newcommand{\mA}{{\bf A}}
\newcommand{\mB}{{\bf B}}

\newcommand{\mH}{{\bf H}}
\newcommand{\mI}{{\bf I}}

\newcommand{\mM}{{\bf M}}

\newcommand{\mS}{{\bf S}}

\newcommand{\mW}{{\bf W}}

\newcommand{\mZ}{{\bf Z}}

\theoremstyle{plain}
\newtheorem{thm}{Theorem}[]

\newtheorem{lem}[thm]{Lemma}

\newtheorem{rem}{Remark}[]
\newtheorem{cor}{Corollary}[]
\theoremstyle{remark}

\newtheorem*{assumption*}{\assumptionnumber}
\providecommand{\assumptionnumber}{}
\makeatletter
\newenvironment{assumption}[2]
 {%
  \renewcommand{\assumptionnumber}{\textbf{Assumption #1${#2}$}}%
  \begin{assumption*}%
  \protected@edef\@currentlabel{#1${#2}$}%
 }
 {%
  \end{assumption*}
 }
\makeatother

\providecommand{\kernel}[1]{{\rm Null}\left( #1\right)}
\providecommand{\range}[1]{{\rm Range}\left( #1\right)}

\usepackage[colorinlistoftodos,bordercolor=orange,backgroundcolor=orange!20,linecolor=orange,textsize=scriptsize]{todonotes}
\usepackage[colorlinks=true,linkcolor=blue]{hyperref}

\begin{document}
\title{Convergence Analysis of Inexact Randomized Iterative Methods}
\author{Nicolas Loizou \thanks{University of Edinburgh}
 \and Peter Richt\'{a}rik \thanks{King Abdullah University of Science and Technology (KAUST); University of Edinburgh, MIPT}}
\maketitle
\providecommand{\keywords}[1]{\textbf{Keywords} #1} 
\providecommand{\ams}[1]{\textbf{Mathematical Subject Classifications } #1}

\begin{abstract}
In this paper we present a convergence rate analysis of inexact variants of several randomized iterative methods. Among the methods studied are: stochastic gradient descent, stochastic Newton, stochastic proximal point and stochastic subspace ascent. A common feature of these methods is that in their update rule a certain sub-problem needs to be solved exactly. We relax this requirement by allowing for the sub-problem to be solved inexactly. In particular, we propose and analyze inexact randomized iterative methods for solving three closely related problems: a convex stochastic quadratic optimization problem, a best approximation problem and its dual, a concave quadratic maximization problem. 
We provide iteration complexity results under several assumptions on the inexactness error. Inexact variants of many popular and some more exotic methods, including  randomized block Kaczmarz, randomized Gaussian Kaczmarz and randomized block coordinate descent, can be cast as special cases. Numerical experiments demonstrate the  benefits of allowing inexactness.
\end{abstract}

\noindent \keywords{ Inexact methods $\cdot$ Iteration complexity $\cdot$ Linear systems $\cdot$ Randomized block coordinate descent $\cdot$ Randomized block Kaczmarz $\cdot$  Stochastic gradient descent $\cdot$ Stochastic Newton method$\cdot$ Quadratic optimization $\cdot$ Convex optimization}

\noindent \ams{68Q25 $\cdot$ 68W20 $\cdot$ 68W40 $\cdot$ 65Y20 $\cdot$ 90C15 $\cdot$ 90C20 $\cdot$ 90C25 $\cdot$ 15A06 $\cdot$ 15B52 $\cdot$ 65F10 }

\section{Introduction}
In the era of big data where data sets become continuously larger, randomized iterative methods become very popular and they are now playing major role in areas like numerical linear algebra, scientific computing and optimization. They are preferred mainly because of their cheap per iteration cost which leads to the improvement in terms of complexity upon classical results by orders of magnitude and to the fact that they can easily scale to extreme dimensions.
However, a common feature of these methods is that in their update rule a particular subproblem needs to be solved exactly. In the case that the size of this problem is large, this step can be computationally very expensive. The purpose of this work is to reduce the cost of this step by incorporating inexact updates in the stochastic methods under study.

\subsection{The Setting}
In this paper we are interested to solve three closely related problems:
\begin{itemize}
\item Stochastic Quadratic Optimization Problem 
\item Best Approximation Problem
\item Concave Quadratic Maximization Problem
\end{itemize}
We start by presenting the main connections and key relationships between these problems as well as popular randomized iterative methods (with exact updates) for solving each one of them.

\paragraph{Stochastic Optimization Problem:} 

We study the stochastic quadratic optimization problem
\begin{equation}
\label{eq:stoch_reform}
\min_{x\in \R^n} f(x) \eqdef \Exp_{\mS\sim \cD}[f_\mS(x)],
\end{equation}
first proposed in \cite{ASDA} for reformulating \emph{consistent} linear systems
\begin{equation}
\label{linear_system_intro}
\bA x= b.
\end{equation}
In particular, problem \eqref{eq:stoch_reform} is defined by setting:
\begin{equation}
\label{eq:f_s}
f_{\mS}(x) \eqdef \frac{1}{2}\|\mA x - b\|_{\mH}^2 = \frac{1}{2}(\mA x - b)^\top \mH (\mA x - b),
\end{equation}
where $\mH$ is a random symmetric positive semi-definite matrix $\mH \eqdef  \mS (\mS^\top \mA \mB^{-1} \mA^\top \mS)^\dagger \mS^\top$ that depends on three different matrices: the data matrix $\mA\in \R^{m\times n}$ of the linear system \eqref{linear_system_intro}, a random matrix $\bS\in \R^{m\times q}\sim \cD$ and on an $n\times n$ positive definite matrix $\mB$ which defines the geometry of the space.  Throughout the paper, $\mB$ is used to define a $\bB-$inner product in $\R^n$ via $\langle x,z \rangle_\mB \eqdef \langle \mB x, z\rangle$ and an induced $\bB-$norm,
$\|x\|_\mB\eqdef (x^\top \mB x)^{1/2}$. By $\dagger$ we denote the Moore-Penrose pseudoinverse. 

The expectation in \eqref{eq:stoch_reform} is over random matrices $\mS$ with $m$ rows (and arbitrary number of columns $q$, e.g., $q=1$) drawn from an arbitrary (user defined) distribution $\cD$. The authors of \cite{ASDA} give necessary and sufficient conditions that distribution $D$ needs to be satisfied for the set of solutions of \eqref{eq:stoch_reform} to be equal to the set of solutions of the linear system \eqref{linear_system_intro}; a property for which the term exactness was coined in (see Section~\ref{gental assumption} for more details on exactness).

In \cite{ASDA}, problem \eqref{eq:stoch_reform} was solved via Stochastic Gradient Descent (SGD)\footnote{The gradient is computed with respect to the inner product $\langle \mB x, y\rangle $.}:
\begin{equation}
\label{eq:SGD}
x_{k+1} = x_k - \omega \nabla f_{\mS_k}(x_k),
\end{equation}
 and a linear rate of convergence was proved despite the fact that $f$ is not necessarily strongly convex,  \eqref{eq:stoch_reform} is not a finite-sum problem and a fixed stepsize $\omega > 0$ is used.  
 
The stochastic optimization problem \eqref{eq:stoch_reform} has many unique characteristics mainly because it has constructed in a particular way in order to capture all the information of the linear system \eqref{linear_system_intro}. For example it holds that $f_{\mS}(x) = \tfrac{1}{2}\|\nabla f_{\mS}(x)\|_\mB^2,$  and it can be proved that all eigenvalues of its Hessian matrix $\nabla^2 f(x)$ are upper bounded by 1. Due to these specific characteristics, the update rules of seemingly different randomized iterative methods are identical. In particular the following methods for solving \eqref{eq:stoch_reform} have exactly the same behavior with SGD \cite{ASDA}:
\begin{equation}
\label{alg:SNM}
\text{{\em Stochastic Newton Method (SNM)}} \footnote{In this method we take the $\mB$-pseudoinverse of the Hessian of $f_{\mS_k}$ instead of the classical inverse, as the inverse does not exist. When $\mB=\mI$, the $\mB$ pseudoinverse specializes to the standard Moore-Penrose pseudoinverse.}: \,
x_{k+1} = x_k - \omega (\nabla^2 f_{\mS_k}(x_k))^{\dagger_\mB} \nabla f_{\mS_k}(x_k),
\end{equation}
\begin{equation}
\label{alg:SPPM}
\text{{\em Stochastic Proximal Point Method (SPPM)}}\footnote{In this case, the equivalence only works for $0<\omega\leq 1$.}: \,
x_{k+1} = \arg\min_{x\in \R^n} \left\{ f_{\mS_k}(x) + \frac{1-\omega}{2\omega}\|x-x_k\|_{\mB}^2\right\}.
\end{equation}

In all methods $\omega>0$ is a fixed stepsize and  $\mS_k$ is sampled afresh in each iteration from distribution $\cD$. 
See \cite{ASDA} for more insights into the reformulation \eqref{eq:stoch_reform}, its properties and other equivalent reformulations (e.g., stochastic fixed point problem, probabilistic intersection problem, and stochastic linear system).

\paragraph{Best Approximation Problem and Sketch and Project Method:}
In \cite{ASDA, loizou2017momentum}, it has been shown that for the case of consistent linear systems with multiple solutions, SGD (and as a result SNM \eqref{alg:SNM} and SPPM \eqref{alg:SPPM}) converges linearly to one particular minimizer of function $f$, the projection of the initial iterate $x_0$ onto the solution set of the linear system \eqref{linear_system_intro}. This naturally leads to the {\em best approximation problem}:
\begin{equation}
\label{best approximation}
\min_{x\in \R^n} P(x) \eqdef \tfrac{1}{2}\|x-x_0\|_\mB^2 \quad \text{subject to} \quad \mA x = b.
\end{equation}
Unlike, the linear system \eqref{linear_system_intro} which is allowed to have multiple solutions, the best approximation problem has always (from its construction) a unique solution. For solving problem \eqref{best approximation}, the \emph{Sketch and Project Method (SPM)}: 
\begin{equation}
\label{SPM}
x_{k+1} =  \omega \Pi_{\cL_{\mS_k}, \bB}(x_k) + (1-\omega) x_k,
\end{equation}
was analyzed in \cite{gower2015randomized, ASDA}. Here, $\Pi_{\cL_{\mS_k}, \bB}(x_k)$ denotes the projection of point $x_k$ onto $\cL_{\mS_k}= \{x \in \R^n\;:\; \mS_k^\top \mA x = \mS_k^\top b\}$ in the $\mB$-norm. In the special case of unit stepsize ($\omega=1$) algorithm \eqref{SPM} simplifies to 
\begin{equation}
\label{spstepsize1}
x_{k+1} =  \Pi_{\cL_{\mS}, \bB}(x_k),
\end{equation}
first proposed in \cite{gower2015randomized}. The name \emph{Sketch and Project method} is justified by the iteration structure which follows two steps: (i) Choose the {\em sketched} system $\cL_{\mS_k} \eqdef \{x\;:\; \mS^\top \mA x =  \mS^\top b\}$, (ii) {\em Project} the last iterate $x_k$ onto $\cL_{\mS_k}$. 
The Sketch and Project viewpoint will be useful later in explaining the natural interpretation of the proposed inexact update rules. (see Section~\ref{SFPinterpretation}).

\paragraph{Dual Problem and SDSA:}
The Fenchel dual of \eqref{best approximation} is the (bounded) unconstrained concave quadratic maximization problem
\begin{equation}
\label{dual}
\max_{y\in \R^m} D(y) \eqdef (b-\bA x_0)^\top y - \tfrac{1}{2}\|\bA^\top y\|^2_{\bB^{-1}}.
\end{equation}
Boundedness follows from consistency. It turns out that by varying $\mA, \mB$ and $b$ (but keeping consistency of the linear system), the dual problem in fact captures {\em all} bounded unconstrained concave quadratic maximization problems \cite{loizou2017momentum}.

A direct dual method for solving problem \eqref{dual} was first proposed in \cite{gower2015stochastic}. The dual method---{\em Stochastic Dual Subspace Ascent (SDSA)}--- updates the dual vectors $y_k$ as follows:
\begin{equation}
\label{SDSAalg}
y_{k+1} = y_k + \omega  \mS_k \lambda_k,
\end{equation}
where the random matrix $\mS_k$ is sampled afresh in each iteration from distribution $\cD$, and $\lambda_k$ is chosen in such a way to maximize the dual objective $D$: $\lambda_k \in \arg\max_\lambda D(y_k + \mS_k \lambda)$.  More specifically, SDSA is defined by picking the $\lambda_k$ with the smallest (standard Euclidean) norm. This leads to the formula:
\begin{equation}
\label{lambdak}
\lambda_k =  \left(\mS_k^\top \bA \bB^{-1}\bA^\top \mS_k \right)^\dagger\bS_k^\top \left(b-\bA(x_0 + \bB^{-1}\bA^\top y_k) \right).
\end{equation}

It can be proved, \cite{gower2015stochastic,loizou2017momentum},  that the iterates $\{x_k\}_{k\geq0}$ of the sketch and project method \eqref{SPM} arise as affine images of the iterates $\{y_k\}_{k\geq0}$ of the dual method \eqref{SDSAalg} as follows: 
\begin{equation}
\label{eq:dual-corresp}
x_{k} = x(y_k) =  x_0 + \mB^{-1} \mA^\top y_k.
\end{equation}

In \cite{gower2015stochastic} the dual method was analyzed for the case of unit stepsize ($\omega=1$). Later in \cite{loizou2017momentum} the analysis extended to capture the cases of $\omega \in (0,2)$. Momentum variants of the dual method that provide further speed up have been also studied in \cite{loizou2017momentum}.  

An interesting property that holds between the suboptimalities of the Sketch and Project method and SDSA is that the dual suboptimality of $y$ in terms of the dual function values is equal to the primal suboptimality of $x(y)$ in terms of distance \cite{gower2015stochastic, loizou2017momentum}. That is,
\begin{equation}
\label{identity}
D(y_*)-D(y)=\frac{1}{2}\|x(y_*)-x(y)\|^2_{\bB}.
\end{equation}
This simple to derive result (by combining the expression of the dual function $D(y)$ \eqref{dual} and the equation \eqref{eq:dual-corresp}) gives for free the convergence analysis of SDSA, in terms of dual function suboptimality once the analysis of Sketch and Project is available (see Section~\ref{InexactDualMethods}). 

\subsection{Contributions}
In this work we propose and analyze {\em inexact} variants of all previously mentioned randomized iterative algorithms for solving the stochastic optimization problem, the best approximation problem and the dual problem. In all of these methods, a certain potentially expensive calculation/operation needs to be performed in each step; it is this operation that we propose to be performed inexactly.  For instance, in the case of SGD, it is the computation of the stochastic gradient $\nabla f_{\mS_k}(x_k)$, in the case of  SPM is the computation of the projection $\Pi_{\cL_{\mS}, \bB}(x_k)$, and in the case of SDSA it is the computation of the dual update $\mS_k \lambda_k$.

We perform an iteration complexity analysis under an abstract notion of inexactness and also under a more structured form of inexactness appearing in practical scenarios. An inexact solution of these subproblems can be obtained much more quickly than the exact solution. Since in practical applications the savings thus obtained are larger than the increase in the number of iterations needed for convergence, our inexact methods can be dramatically faster.

Let us now briefly outline the rest of the paper:

In Section~\ref{SecondSection} we describe the subproblems and introduce two notions of inexactness (abstract and structured) that will be used in the rest of the paper. The Inexact Basic Method (iBasic) is also presented. iBasic is a method that simultaneously captures inexact variants of the algorithms \eqref{eq:SGD}, \eqref{alg:SNM}, \eqref{alg:SPPM} for solving the stochastic optimization problem \eqref{eq:stoch_reform} and algorithm \eqref{SPM} for solving the best approximation problem \eqref{best approximation}. It is an inexact variant of the \emph{Basic Method}, first presented in \cite{ASDA}, where the inexactness is introduced by the addition of an inexactness error $\epsilon_k$ in the original update rule. We illustrate the generality of  iBasic  by presenting popular algorithms that can be cast as special cases. 

In Section~\ref{gental assumption} we establish convergence results of iBasic under general assumptions on the inexactness error $\epsilon_k$ of its update rule (see Algorithm~\ref{inexact_basic}). In this part we do not focus on any specific mechanisms which lead to inexactness; we treat the problem abstractly. However, such errors appear often in practical scenarios and can be associated with inaccurate numerical solvers, quantization,  sparsification and compression mechanisms. In particular, we introduce several abstract assumptions on the inexactness level  and describe our generic convergence results. For all assumptions we establish linear rate of decay of the quantity $\Exp[\|x_{k}-x_*\|_{\mB}^2]$ (i.e. L2 convergence)\footnote{As we explain later, a convergence of the expected function values of problem \ref{eq:stoch_reform} can be easily obtained as a corollary of L2 convergence.}. 

Subsequently, in Section~\ref{InexactSolvers} we apply our general convergence results to a more structured notion of inexactness error and  propose a concrete mechanisms leading to such errors.  We provide theoretical guarantees for this method  in situations when  a linearly convergent iterative method (e.g., Conjugate Gradient)  is used to solve the subproblem inexactly. We also highlight the importance of the dual viewpoint through a sketch-and-project interpretation.

In Section~\ref{InexactDualMethods} we study an inexact variant of SDSA, which we called iSDSA, for directly solving the dual problem \eqref{dual}.  We provide a correspondence between  iBasic and iSDSA and we show that the random iterates of iBasic arise as affine images of iSDSA. We consider both abstract and structured inexactness errors and provide linearly convergent rates in terms of the dual function suboptimality  $\E{D(y_*) - D(y_0)}$.

Finally, in Section~\ref{experiments} we evaluate the performance of the proposed inexact methods through numerical experiments and show the benefits of our approach on both synthetic and real datasets. Concluding remarks are given in Section~\ref{conlcusion}. 

A summary of the convergence results of iBasic under several assumptions on the inexactness error  with pointers to the relevant theorems is available in Table~\ref{OurResults}. We highlight that similar convergence results can be also obtained for  iSDSA in terms of the dual function suboptimality  $\E{D(y_*) - D(y_0)}$ (check Section~\ref{InexactDualMethods} for more details on iSDSA).

\begin{table}[t!]
\begin{center}
\scalebox{0.85}{
\begin{tabular}{ |c|c|c|c| }
 \hline
 \begin{tabular}{c} Assumption on \\ the Inexactness error $\epsilon_k$  \end{tabular}  & $\omega$& \begin{tabular}{c}Upper Bounds\end{tabular} & Theorem \\
 \hline
 \hline
Assumption \ref{Assumption1} &  $(0,2)$ & $\rho^{k/2} \|x_{0}-x_*\|_{\mB} +  \sum_{i=0}^{k-1} \rho^{\frac{k-1-i}{2}}\sigma_i$ &\ref{InexactSGDConstant}\\
 \hline
Assumption \ref{Assumption3} &  $(0,2)$ & $\left(\sqrt{\rho}+q\right)^{2k} \|x_0-x_*\|_{\mB}^2$ &\ref{ISGDwithq}\\
 \hline
Assumptions \ref{Assumption2},\ref{Assumption5} & $(0,2)$  &  $\rho^{k} \|x_{0}-x_*\|^2_{\mB} +  \sum_{i=0}^{k-1} \rho^{k-1-i}{\bar{\sigma}}^2_i$ & \ref{InexactSGDrandom}(i) \\
 \hline 
 Assumptions \ref{Assumption3},\ref{Assumption5} & $(0,2)$  &  $\left(\rho + q^2 \right)^k \|x_0-x_*\|_{\mB}^2$ & \ref{InexactSGDrandom}(ii) \\
 \hline 
Assumptions \ref{Assumption4},\ref{Assumption5} & $(0,2)$  &  $\left(\rho+q^2 \lambda_{\min}^+ \right)^k \|x_0-x_*\|_{\mB}^2$ & \ref{InexactSGDrandom}(iii) \\
 \hline 

\end{tabular}}
\end{center}
\caption{\small Summary of the iteration complexity results obtained in this paper. $\omega$ denotes the stepsize (relaxation parameter) of the method. In all cases, $x_*=\Pi_{\cL, \bB}(x_0)$ and $\rho=1- \omega (2-\omega)\lambda_{\min}^+ \in (0,1)$ are the quantities appear in the convergence results (here $\lambda_{\min}^+$ denotes the minimum non zero eigenvalue of matrix $\bW$, see equation \eqref{matrixW}). Inexactness parameter $q$ is chosen always in such a way to obtain linear convergence and it can be seen as the quantity that controls the inexactness. In all theorems the quantity of convergence is $\Exp[\|x_k-x_*\|^2_{\bB}]$ (except in Theorem~\ref{InexactSGDConstant} where we analyze $\Exp[\|x_k-x_*\|_{\bB}])$. As we show in Section~\ref{InexactDualMethods}, under similar assumptions, iSDSA has exactly the same convergence with iBasic but the upper bounds of the third column are related to the dual function values $\E{D(y_*) - D(y_0)}$. }
\label{OurResults}
\end{table}

\subsection{Notation}
For convenience, a table of the most frequently used notation is included in the Appendix \ref{some Tables}.
In particular, with boldface upper-case letters we denote matrices and $\bI$ is the identity matrix. By $\cL$ we denote the solution set of the linear system $\bA x=b$. By $\cL_{\bS}$, where $\mS$ is a random matrix, we denote the solution set of the {\em sketched} linear system $\bS^\top \bA x= \bS^\top b$.  In general, we use $\cdot^{*}$ to express the exact solution of a sub-problem and $\cdot^{\approx}$ to indicate its inexact variant. 
Unless stated otherwise, throughout the paper, $x_*$ is the projection of $x_0$ onto $\cL$ in the $\mB$-norm: $x_*=\Pi_{\cL, \bB}(x_0)$. 
An explicit formula for the projection of point $x$ onto set $\cL$ is given by
\begin{equation}
\label{projection}
\Pi_{\cL, \bB}(x)\eqdef \arg\min_{x' \in \cL} \|x'-x\|_{\bB} =x-\bB^{-1}\bA^\top (\bA\bB^{-1}\bA ^ \top )^\dagger (\bA x-b).
\end{equation}
A formula for the projection onto $\cL_{\mS}= \{x \in \R^n\;:\; \mS^\top \mA x = \mS^\top b\}$ is obtained by replacing $\mA$ and $b$ with $\mS^\top \mA$ and $\mS^\top b$  respectively into the above equation. We denote this projection by $\Pi_{\cL_{\mS},\bB}(x)$. We also write $[n]\eqdef \{1,2, \dots,n\}$.

In order to keep the expression brief throughout the paper we define\footnote{In the $k^{th}$ iterate the expression becomes $\mZ_k \eqdef \mA^\top \mS_k (\mS_k^\top \mA \mB^{-1} \mA^\top \mS_k)^{\dagger} \mS_k^\top \mA$.}:
\begin{equation}
\label{ZETA}
\mZ \eqdef \mA^\top \mH \mA={\bA}^\top \mS(\mS^\top\mA \mB^{-1} \mA^\top \mS)^{\dagger} \mS^{\top} \bA.
\end{equation}
Using this matrix we can easily express important quantities related to the problems under study. For example the stochastic functions $f_\mS$ of problem \eqref{eq:stoch_reform} can be expressed as 
\begin{equation}
\label{eq:f_s22}
f_{\mS}(x) = \frac{1}{2}(\mA x - b)^\top \mH (\mA x - b)= \frac{1}{2} (x-x_*)^\top \bZ_k (x-x_*),
\end{equation}
In addition the gradient and the Hessian of $f_\mS$ with respect to the $\mB$ inner product are equal to
\begin{equation}\label{eq:grad_f_S}\nabla f_\mS(x) \overset{\eqref{eq:f_s}}{=} \mB^{-1} \mA^\top \mH (\mA x - b)  = \mB^{-1} \mA^\top \mH \mA (x -x_*)  = \mB^{-1} \mZ (x -x_*),
\end{equation}
 and $\nabla^2 f_{\mS}(x)=\bB^{-1}\bZ$ \cite{ASDA}. Similarly the gradient and Hessian of the objective function $f$ of problem \eqref{eq:stoch_reform}  are $\nabla f(x)=\mB^{-1}\E{\mZ}(x-x_*)$ and $\nabla^2 f(x)=\mB^{-1}\E{\mZ},$ respectively. 
 
 A key matrix in our analysis is 
\begin{equation}
\label{matrixW}
\bW\eqdef\bB^{-\frac{1}{2}}\Exp[{\mZ}]\bB^{-\frac{1}{2}},
\end{equation}
which has the same spectrum with the matrix $\nabla^2 f(x)$ but at the same time is symmetric and positive semi-definite\footnote{Note that matrix $\nabla^2 f(x)$ is not symmetric but it is self-adjoint with respect to the $\bB$-inner product.}.
We denote with $\lambda_1\leq \lambda_2 \leq \cdots \leq \lambda_{n}$ the $n$ eigenvalues of $\bW$. With $\lambda_{\min}^+$ we indicate the smallest nonzero eigenvalue, and with $\lambda_{\max}  = \lambda_n$ the largest eigenvalue. It was shown in \cite{ASDA} that $0\leq \lambda_i \leq 1$ for all $i\in [n]$.

\section{Inexact update rules}
\label{SecondSection}
In this section we start by explaining the key sub-problems that need to be solved exactly in the update rules of the previously described methods. We present iBasic, a method that solves problems \eqref{eq:stoch_reform} and \eqref{best approximation} and we show how by varying the main parameters of the method we recover inexact variants of popular algorithms as special cases. Finally closely related work on inexact algorithms for solving different problems is also presented.
\subsection{Expensive Sub-problems in Update Rules}
\label{subproblems}
Let us devote this subsection on explaining how the inexactness can be introduced in the current exact update rules of SGD\footnote{Note that SGD has identical updates to the Stochastic Newton and Stochastic proximal point method. Thus the inexactness can be added to these updates in similar way.} \eqref{eq:SGD}, Sketch and Project \eqref{SPM} and SDSA \eqref{SDSAalg} for solving the stochastic optimization, best approximation  and the dual problem respectively. As we have shown these methods solve closely related problems and the key subproblems in their update rule are similar. However the introduction of inexactness in the update rule of each one of them can have different interpretation. 

For example for the case of SGD for solving the stochastic optimization problem \eqref{eq:stoch_reform} (see also Section~\ref{sectionlinesyst} and \ref{SFPinterpretation} for more details), if we define $\lambda_k^*=(\mS_{k}^\top \mA \mB^{-1} \mA^\top \mS_{k})^{\dagger} \mS_{k}^\top(b-\mA x_k) $  then the stochastic gradient of function $f$ becomes  $\nabla f_{\mS_k}(x_k)\overset{\eqref{eq:grad_f_S}}{=}-\bB^{-1}\bA ^\top \bS_k \lambda_k^*$ and the update rule of SGD takes the form: $x_{k+1} = x_k+\omega \mB^{-1}\mA^\top \mS_{k} \lambda_k^*$.  Clearly in this update the expensive part is the computation of the quantity $\lambda_k^*$ that can be equivalently computed to be the least norm solution of the smaller (in comparison to $\bA x=b$) linear system $\mS_{k}^\top \mA \mB^{-1} \mA^\top \mS_{k} \lambda =\mS_{k}^\top (b-\mA x_k)$. In our work we are suggesting to use an approximation $\lambda_k^{\approx}$ of the exact solution and with this way avoid executing the possibly expensive step of the update rule.  Thus the inexact update is taking the following form:
$$x_{k+1} =x_k+\omega \mB^{-1}\mA^\top \mS_{k} \lambda_k^{\approx}= x_k - \omega \nabla f_{\mS_k}(x_k)+\underbrace{\omega \bB^{-1}\bA^\top \bS_k (\lambda_k^{\approx}-\lambda_k^*)}_{\epsilon_k}.$$
Here $\epsilon_k$ denotes a more abstract notion of inexactness and it is not necessary to be always equivalent to the quantity $\omega \bB^{-1}\bA^\top \bS_k(\lambda_k^{\approx}-\lambda_k^*)$. It can be interpreted as an expression that acts as an perturbation of the exact update. In the case that $\epsilon_k$ has the above form we say that the notion of inexactness is structured.
In our work we are interested in both the \emph{abstract} and more \emph{structured} notions of inexactness. We first present general convergence results where we require the error $\epsilon_k$ to satisfy general assumptions (without caring how this error is generated) and later we analyze the concept of structured inexactness by presenting algorithms where $\epsilon_k= \omega \bB^{-1}\bA^\top \bS_k(\lambda_k^{\approx}-\lambda_k^*)$.

In similar way, the expensive operation of SPM \eqref{SPM} is the exact computation of the projection $\Pi_{\cL_{\mS_k},\mB}^*(x_k)$. Thus we are suggesting to replace this step with an inexact variant and compute an approximation of this projection. The inexactness here can be also interpreted using both, the abstract $\epsilon_k$ error and its more structured version $\epsilon_k=\omega \left( \Pi_{\cL_{\mS_k},\mB}^{\approx}(x_k)- \Pi_{\cL_{\mS_k},\mB}^*(x_k) \right)$. At this point, observe that, by using the expression \eqref{projection} the structure of the $\epsilon_k$ in SPM and SGD has the same form.

In the SDSA the expensive subproblem in the update rule is the computation of the $\lambda_k^*$ that satisfy $\lambda_k^* \in \arg\max_\lambda D(y_k + \mS_k \lambda)$. Using the definition of the dual function \eqref{dual} this value can be also computed by evaluating the least norm solution of the linear system $ \mS_{k}^\top \mA \mB^{-1} \mA^\top \mS_{k} \lambda  =\mS_{k}^\top \left(b-\bA(x_0 + \bB^{-1}\bA^\top y_k \right))$. Later in Section~\ref{InexactDualMethods} we analyze both notions of inexactness (abstract and more structured) for inexact variants of SDSA.

Table~\ref{KeySubproblems} presents the key sub-problem that needs to be solved in each algorithm as well as the part where the inexact error is appeared in the update rule.  

\begin{table}[t!]
\begin{center}
\scalebox{0.65}{
\begin{tabular}{ |c|c|c| }
 \hline
Exact Algorithms &\begin{tabular}{c} Key Subproblem \\ (problem that we solve inexactly) \end{tabular}  & \begin{tabular}{c} Inexact Update Rules \\ (abstract and structured inexactness error) \end{tabular} \\
 \hline
 \hline
 SGD \eqref{eq:SGD} & \begin{tabular}{c}Exact computation of $\lambda_k^*$, \\where $ \lambda_k^*=\arg\min_{ \lambda : \bM_k  \lambda =d_k} \| \lambda\|$. \\ Appears in the computation of $\nabla f_{\mS_k}(x_k)=-\bB^{-1}\bA ^\top \bS_k \lambda_k^*$  \end{tabular}&\begin{tabular}{c} \\$x_{k+1} = x_k+\omega \mB^{-1}\mA^\top \mS_{k} \lambda_k^{\approx}$\\$\quad = x_k - \omega \nabla f_{\mS_k}(x_k)+\underbrace{\omega \bB^{-1}\bA^\top \bS_k (\lambda_k^{\approx}-\lambda_k^*)}_{\epsilon_k}.$\end{tabular}\\
 \hline
SPM \eqref{SPM} & \begin{tabular}{c}Exact computation of the projection \\$ \Pi_{\cL_{\mS_k},\mB}^*(x_k)=\arg\min_{x' \in \cL_{\mS_k}} \|x'-x_k\|_{\bB} $ \\ \end{tabular}& \begin{tabular}{c}\\$x_{k+1} =  \omega \Pi_{\cL_{\mS_k},\mB}^{\approx}(x_k)+ (1-\omega) x_k$\\ $\quad=  \omega \Pi^\mB_{\cL_{\mS_k}}(x_k) + (1-\omega) x_k+ \underbrace{\omega \left(  \Pi_{\cL_{\mS_k},\mB}^{\approx}(x_k)- \Pi_{\cL_{\mS_k},\mB}^*(x_k) \right)}_{\epsilon_k} $\end{tabular}\\
 \hline
SDSA \eqref{SDSAalg} & \begin{tabular}{c}Exact computation of $\lambda_k^*$, \\ where $\lambda_k^* \in \arg\max_\lambda D(y_k + \mS_k \lambda)$. \end{tabular}& \begin{tabular}{c}\\$y_{k+1} = y_k + \omega  \mS_k \lambda_k^{\approx}=y_k + \omega  \mS_k \lambda_k^*+\underbrace{\omega \bS_k (\lambda_k^{\approx} - \lambda_k^*) }_{\epsilon_k^d}$ \end{tabular} \\
 \hline 
\end{tabular}}
\end{center}
\caption{\small The exact algorithms under study with the potentially expensive to compute key sub-problems of their update rule. The inexact update rules are presented in the last column for both notions of inexactness (abstract and more structured). We use ${\cdot}^*$ to define the important quantity that needs to be computed exactly in the update rule of each method and ${\cdot}^{\approx}$ to indicate the proposed inexact variant.}
\label{KeySubproblems}
\end{table}

\subsection{The Inexact Basic Method}
In each iteration of the all aforementioned exact methods a sketch matrix $\bS \sim {\cal D}$ is drawn from a given distribution and then a certain subproblem is solved exactly to obtain the next iterate. The sketch matrix $\bS \in \R^{m \times q}$ requires to have $m$ rows but no assumption on the number of columns is made which means that the number of columns $q$ allows to vary through the iterations and it can be very large.  The setting that we are interested in is precisely that of having such large random matrices $\bS$. In these cases we expect that having approximate solutions of the subproblems will be beneficial.

Recently randomized iterative algorithms that requires to solve large subproblems in each iteration have been extensively studied and it was shown that are really beneficial when they compared to their single coordinates variants ($\bS \in \R^{m \times 1}$) \cite{RBK,l2015randomized,richtarik2014iteration,LoizouRichtarik}. However, in theses cases the evaluation of an exact solution for the suproblem in the update rule can be computationally very expensive.
In this work we propose and analyze inexact variants by allowing to solve the  subproblem that appear in the update rules of the stochastic methods, inexactly. In particular, following the convention established in \cite{ASDA} of naming the main algorithm of the paper \emph{Basic method} we propose the \emph{inexact Basic method (iBasic)} (Algorithm~\ref{inexact_basic}). 

\begin{algorithm}[H]
  \caption{Inexact Basic Method (iBasic)
    \label{inexact_basic}}
  \begin{algorithmic}[1]
    \Require{Distribution $\cD$ from which we draw random matrices $\bS$, positive definite matrix $\bB\in\R^{n\times n}$, stepsize $\omega>0$.}
    \Ensure{$x_0\in\R^n$}
 \For{$k=1,2,\cdots$}
 \State Generate a fresh sample $\bS_k \sim {\cal D}$
 \State Set $x_{k+1}=x_k-\omega \mB^{-1}\mA^\top \mS_{k} (\mS_{k}^\top \mA \mB^{-1} \mA^\top \mS_{k})^{\dagger} \mS_{k}^\top(\mA x_k-b)+ \epsilon_k$
 \EndFor
 \end{algorithmic}
\end{algorithm}

The $\epsilon_k$ in the update rule of the method represents the abstract inexactness error described in Subsection~\ref{subproblems}. Note that, iBasic can have several equivalent interpretations. This allow as to study the methods \eqref{eq:SGD},\eqref{alg:SNM},\eqref{alg:SPPM} for solving the stochastic optimization problem and the sketch and project method \eqref{SPM} for the best approximation problem in a single algorithm only. 
In particular iBasic can be seen as inexact stochastic gradient descent (iSGD) with fixed stepsize applied to \eqref{eq:stoch_reform}. From \eqref{eq:f_s22}, $\nabla f_{\mS_k}(x_k) = \mB^{-1} \mA^\top \mH_k (\mA x_k - b) $ and as a result the update rule of iBasic can be equivalently written as: $x_{k+1}=x_k-\omega \nabla f_{\mS_k}(x_k)  + \epsilon_{k}.$ In the case of the best approximation problem \eqref{best approximation}, iBasic can be interpreted as inexact Sketch and Project method (iSPM) as follows:
\begin{eqnarray}
\label{anska}
x_{k+1} & = & x_k-\omega \mB^{-1}\mA^\top \mS_{k} (\mS_{k}^\top \mA \mB^{-1} \mA^\top \mS_{k})^{\dagger} \mS_{k}^\top(\mA x_k-b)+\epsilon_k  \notag\\
& = &  \omega \left[ x_k- \bB^{-1}(\bS_k^\top \bA)^\top (\bS_k^\top \bA B^{-1}(\bS_k^\top \bA) ^ \top )^\dagger (\bS_k^\top \bA x-\bS_k^\top b)\right] + (1-\omega) x_k   +\epsilon_k \notag\\
&\overset{\eqref{projection}} =&\omega \Pi_{\cL_{\mS_k},\mB}(x_k) + (1-\omega) x_k +\epsilon_k
\end{eqnarray}
For the dual problem \eqref{dual} we devote Section~\ref{InexactDualMethods} for presenting an inexact variant of the SDSA (iSDSA) and analyze its convergence using the rates obtained for the iBasic in Sections~\ref{gental assumption} and \ref{InexactSolvers}.

\subsection{General Framework and Further Special Cases}
\label{Special Cases}

The proposed inexact methods, iBasic  (Algorithm~\ref{inexact_basic}) and iSDSA (Section~\ref{InexactDualMethods}),  belong in the general \emph{sketch and project} framework, first proposed from Gower and Richtarik in \cite{gower2015randomized} for solving consistent linear systems and where a unified analysis of several randomized methods was studied. This interpretation of the algorithms allow us to recover a comprehensive array of well-known methods as special cases by choosing carefully the combination of the main parameters of the algorithms.

In particular, the iBasic has two main parameters (besides the stepsize $\omega>0$ of the update rule). These are the distribution $\cD$ from which we draw random matrices $\bS$ and the positive definite matrix $\bB\in\R^{n\times n}$. By choosing carefully combinations of the parameters $\cD$ and $\bB$ we can recover several existing popular algorithms as special cases of the general method.  For example, special cases of the exact Basic method are the Randomized Kaczmarz, Randomized Gaussian Kaczmarz\footnote{Special case of the iBasic, when the random matrix $\bS$ is chosen to be a Gaussian vector with mean $0 \in R^m$ and a positive definite covariance matrix $\Sigma \in \R^{m\times m}$. That is $\bS \sim N(0,\Sigma)$ \cite{gower2015randomized,loizou2017momentum}.}, Randomized Coordinate Descent and their block variants. For more details about the generality of the sketch and project framework and further algorithms that can be cast as special cases of the analysis we refer the interested reader to Section 3 of \cite{gower2015randomized} and Section 7 of \cite{loizou2017momentum}. Here we present only the inexact update rules of two special cases that we will later use in the numerical evaluation. 

\emph{Special Cases: }
Let us define with $\bI_{:C}$ the column concatenation of the $m \times m$ identity matrix indexed by a random subset $C$ of $[m]$. 
\begin{itemize}
\item \emph{Inexact Randomized Block Kaczmarz (iRBK)}:
Let $\bB= \bI$ and let pick in each iteration the random matrix $\bS=\bI_{:C} \sim \cD$. In this setup the update rule of the iBasic simplifies to 
\begin{equation}
\label{iRBK}
x_{k+1}=x_k -\omega \bA_{C:}^\top (\bA_{C:}\bA_{C:}^\top)^\dagger (\bA_{C:}x_k-b_C) + \epsilon_k.
\end{equation}
\item \emph{Inexact Randomized Block Coordinate Descent (iRBCD)}\footnote{In the setting of solving linear systems Randomized Coordinate Descent is known also as Gauss-Seidel method. Its block variant can be also interpret as randomized coordinate Newton method (see \cite{qu2015sdna}).}:
If the matrix $\bA$ of the linear system is positive definite then we can choose $\bB= \bA$. Let also pick in each iteration the random matrix $\bS=\bI_{:C} \sim \cD$. In this setup the update rule of the iBasic simplifies to 
\begin{equation}
\label{iRBCD}
x_{k+1}=x_k -\omega \bI_{:C} (\bI_{:C}^\top\bA \bI_{:C})^\dagger \bI_{:C}^\top (\bA x_k-b) + \epsilon_k. 
\end{equation}
\end{itemize}

For more papers related to Kaczmarz method (randomized, greedy, cyclic update rules) we refer the interested reader to \cite{kaczmarz1937angenaherte, loizou2017linearly, popa1995least,byrne2008applied, nutini2016convergence, popa2017convergence, CsibaPL, needell2010randomized, RBK, eldar2011acceleration, MaConvergence15, zouzias2013randomized, l2015randomized, schopfer2016linear}. 
For the coordinate descent method (a.k.a Gauss-Seidel for linear systems) and its block variant, Randomized Block Coordinate Descent we suggest \cite{leventhal2010randomized, nesterov2012efficiency, richtarik2014iteration, richtarik2016parallel, qu2016coordinate,qu2016coordinate2, qu2015quartz, SCP, lee2013efficient, fercoq2015accelerated, allen2016even, tu2017breaking}. 

\subsection{Other Related Work on Inexact Methods}
\label{OtherWork}
One of the current trends in the large scale optimization problems is the introduction of inexactness in the update rules of popular deterministic and stochastic methods. The rational behind this is that an approximate/inexact step can often computed very efficiently and can have significant computational gains compare to its exact variants. 

In the area of deterministic algorithms, the inexact variant of the full gradient descent method, $x_{k+1} = x_k - \omega_k [\nabla f(x_k)+\epsilon_k]$, has received a lot of attention \cite{schmidt2011convergence, devolder2014first,so2017non, friedlander2012hybrid,necoara2014rate}. It has been analyzed for the cases of convex and strongly convex functions under several meaningful assumptions on the inexactness error $\epsilon_k$ and its practical benefit  compared to the exact gradient descent is apparent.  For further deterministic inexact methods check \cite{dembo1982inexact} for Inexact Newton methods, \cite{solodov2001unified, salzo2012inexact} for Inexact Proximal Point methods and \cite{birken2015termination} for Inexact Fixed point methods.

In the recent years, with the explosion that happens in areas like machine learning and data science inexactness enters also the updating rules of several stochastic optimization algorithms and many new methods have been proposed and analyzed. 

In the large scale setting, stochastic optimization methods are preferred mainly because of their cheap per iteration cost (compared to their deterministic variants), their property to scale to extreme dimensions and their improved theoretical complexity bounds. In areas like machine learning and data science, where the datasets become larger rapidly, the development of faster and efficient stochastic algorithms is crucial. For this reason, inexactness has recently introduced to the update rules of several stochastic optimization algorithms and new methods have been proposed and analyzed. One of the most interesting work on inexact stochastic algorithms appears in the area of second order methods. In particular on inexact variants of the Sketch-Newton method and subsampled Newton Method for minimize convex and non-convex functions \cite{schmidt201111, berahas2017investigation, bollapragada2016exact,xu2017newton,xu2016sub, yao2018inexact}. Note that our results are related also with this literature since our algorithm can be seen as inexact stochastic Newton method (see equation \eqref{alg:SNM}). To the best or our knowledge our work is the first that provide convergence analysis of inexact stochastic proximal point methods (equation \eqref{alg:SPPM}) in any setting. From numerical linear algebra viewpoint inexact sketch and project methods for solving the best approximation problem and its dual problem where also never analyzed before. 

As we already mentioned our framework is quite general and many algorithms, like iRBK \eqref{iRBK} and iRBCD \eqref{iRBCD} can be cast as special cases. As a result, our general convergence analysis includes the analysis of inexact variants of all of these more specific algorithms as special cases. In \cite{RBK} an analysis of the exact randomized block Kacmzarz method has been proposed and in the experiments an inexact variant was used to speedup the method. However, no iteration complexity results were presented for the inexact variant and both the analysis and numerical evaluation have been made for linear systems with full rank matrices that come with natural partition of the rows (this is a much more restricted case than the one analyzed in our setting). For inexact variants of the randomized block coordinate descent algorithm in different settings than ours we suggest  \cite{tappenden2016inexact, fountoulakis2018flexible, cassioli2013convergence, dvurechensky2017randomized}.

Finally an analysis of approximate stochastic gradient descent for solving the empirical risk minimization problem using quadratic constraints and sequential semi-definite programs has been presented in \cite{hu2017analysis}. 

\section{Convergence Results Under General Assumptions}
\label{gental assumption}
In this section we consider scenarios in which the inexactness error $\epsilon_k$  can be controlled, by specifying a per iteration bound $\sigma_k$ on the norm of the error.  In particular, by making different assumptions on the bound $\sigma_k$ we derive general convergence rate results. Our focus is on the abstract notion of inexactness described in Section~\ref{subproblems} and we make no assumptions on how this error is generated. 

An important assumption that needs to be hold in all of our results is exactness. A formal presentation is presented below. We state it here and we highlight that is a requirement for all of our convergence results (exactness is also required in the analysis of the exact variants \cite{ASDA}).
\paragraph{Exactness.} 
Note that $f_\mS$ is a convex quadratic, and that $f_\mS(x) = 0$ whenever $x\in \cL\eqdef \{x\;:\; \mA x = b\}$. However, $f_\mS$ can be zero also for points $x$ outside of $\cL$. Clearly, $f(x)$ is nonnegative, and $f(x)=0$ for $x\in \cL$. However, without further assumptions, the set of minimizers of $f$ can be larger than $\cL$. The exactness assumption ensures that this does not happen. For necessary and sufficient conditions for exactness, we refer the reader to \cite{ASDA}. Here it suffices to remark that  a sufficient condition for exactness is to require $\E{\mH}$ to be positive definite. This is easy to see by observing that
$f(x) = \E{f_\mS(x)} =\tfrac{1}{2}\|\mA x - b\|^2_{\E{\mH}}$.
In other words,  if $\cX=\text{argmin} f(x)$ is the solution set of the stochastic optimization problem \eqref{eq:stoch_reform} and $\cL=\{x: \bA x =b\}$ the solution set of the linear system \eqref{linear_system_intro} then the notion of exactness is captured by: $\cX=\cL$
\subsection{Assumptions on Inexactness Error}
\label{asssad}
In the convergence analysis of iBasic the following assumptions on the inexactness error are used.  We note that Assumptions \ref{Assumption1}, \ref{Assumption3} and \ref{Assumption4} are special cases of Assumption \ref{Assumption2}. Moreover Assumption \ref{Assumption5} is algorithmic dependent and can hold in addition of any of the other four assumptions. In our analysis, depending on the result we aim at, we will require either one of the first four Assumptions to hold by itself, or to hold together with Assumption \ref{Assumption5}. We will always assume exactness.
  
In all assumptions the expectation on the norm of error ($\|\epsilon_k\|^2$) is conditioned on the value of the current iterate $x_k$ and the random matrix $\bS_k$. Moreover it is worth to mention that for the convergence analysis we never assume that the inexactness error has zero mean, that is $\Exp[\epsilon_k]=0$. 

\begin{assumption}{1}{}
\label{Assumption2}
 \begin{equation}
 \label{Assumption2serial}
  \Exp[\|\epsilon_k\|^2_{\bB}\;|\;x_k,\bS_k]\leq\sigma_k^2,
  \end{equation}
where the upper bound $\sigma_k$ is a sequence of random variables (that can possibly depends on both the value of the current iterate $x_k$ and the choice of the random $\bS_k$ at the $k^{th}$ iteration).
\end{assumption}

The following three assumptions on the sequence of upper bounds are more restricted however as we will later see allow us to obtain stronger and more controlled results.

 \begin{assumption}{1}{a}
 \label{Assumption1}
 \begin{eqnarray}
\label{Assumption1serial}
 \Exp[\|\epsilon_k\|^2_{\bB}\;|\;x_k,\bS_k]\leq\sigma_k^2,
\end{eqnarray}
where the upper bound $\sigma_k\in \R$ is a sequence of real numbers. 
\end{assumption}

\begin{assumption}{1}{b}
\label{Assumption3}
 \begin{equation}
 \label{Assumption3serial}
  \Exp[\|\epsilon_k\|^2_{\bB}\;|\;x_k,\bS_k]\leq\sigma_k^2 = q^2\|x_k-x_*\|^2_{\bB},
  \end{equation}
where the upper bound is a special sequence that depends on a non-negative inexactness parameter $q$ and the distance to the optimal value $\|x_k-x_*\|^2_{\bB}$.
\end{assumption}

\begin{assumption}{1}{c}
\label{Assumption4}
 \begin{equation}
 \label{Assumption4serial}
  \Exp[\|\epsilon_k\|^2_{\bB}\;|\;x_k,\bS_k]\leq\sigma_k^2 = 2 q^2 f_{\bS_k}(x_k),
  \end{equation}
where the upper bound is a special sequence that depends on a non-negative inexactness parameter $q$ and the value of the stochastic function $f_{\bS_k}$ computed at the iterate $x_k$.
\end{assumption}

Finally the next assumption is more algorithmic oriented. It holds in cases where the inexactness error $\epsilon_k$ in the update rule is chosen to be orthogonal with respect to the $\bB$-inner product to the vector  $\Pi_{\cL_{\bS_k},{\bB}} (x_k) - x_* = (\mI - \omega \mB^{-1}\mZ_k) (x_k-x_*)$. This statement may seem odd at this point but its usefulness will become more apparent in the next section where inexact algorithms with structured inexactness error will be analyzed. As it turns out, in the case of structured inexactness error (Algorithm~\ref{inexact_solver_algorithm}) this assumption is satisfied.
\begin{assumption}{2}{}
 \label{Assumption5}
 \begin{equation}
\Exp[\left\langle (\mI - \omega \mB^{-1}\mZ_k) (x_k-x_*), \epsilon_k \right \rangle_{\mB}]=0.
  \end{equation}
\end{assumption}

\subsection{Convergence Results}
In this section we present the analysis of the convergence rates of iBasic by assuming several combination of the previous presented assumptions. 

All convergence results are described only in terms of convergence of the iterates $x_k$, that is $\|x_k-x_*\|^2_{\bB}$, and not the objective function values $f(x_k)$. This is sufficient, because by $f(x)\leq\frac{\lambda_{\rm max}}{2}\|x-x_*\|^2_{\bB}$ (see Lemma \ref{boundLemma}) we can directly deduce a convergence rate for the function values. 

The exact Basic method (Algorithm~\ref{inexact_basic} with $\epsilon_k =0$), has been analyzed in \cite{ASDA} and it was shown to converge with $\Exp[\|x_{k}-x_*\|^2_{\mB}] \leq \rho^{k} \|x_{0}-x_*\|^2_{\mB}$ where $\rho=1-\omega(2-\omega) \lambda_{\min}^+$. Our analysis of iBasic is more general and includes the convergence of the exact Basic method as special case when we assume that the upper bound is $\sigma_k=0, \quad \forall k\geq0$. For brevity, in he convergence analysis results of this manuscript we also use $$\rho=1-\omega(2-\omega) \lambda_{\min}^+.$$

Let us start by presenting the convergence of iBasic when only Assumption~\ref{Assumption1} holds for the inexactness error.
\begin{thm}
\label{InexactSGDConstant}
Let assume exactness and let $\{x_k\}_{k=0}^\infty$ be the iterates produced by iBasic with $\omega\in(0,2)$.  Set $x_*= \Pi_{\cL,\mB}(x_0)$ and consider the error $\epsilon_k$ be such that it satisfies Assumption \ref{Assumption1}. Then,
\begin{eqnarray}
\label{Theorem1}
\Exp[\|x_{k}-x_*\|_{\mB}] \leq \rho^{k/2} \|x_{0}-x_*\|_{\mB} +  \sum_{i=0}^{k-1} \rho^{\frac{k-1-i}{2}}\sigma_i.
\end{eqnarray}
\end{thm}
\begin{proof}
See Appendix~\ref{Appendix1}.
\end{proof}
\begin{cor}
\label{FirstCorollary}
 In the special case that the upper bound $\sigma_k$ in Assumption \ref{Assumption1} is fixed, that is $\sigma_k=\sigma$ for all $k>0$ then inequality \eqref{Theorem1} of Theorem \ref{InexactSGDConstant} takes the following form:
\begin{eqnarray}
\Exp[\|x_{k}-x_*\|_{\mB}] \leq  \rho^{k/2} \|x_{0}-x_*\|_{\mB} +  \sigma \frac{\rho^{1/2}}{1-\rho}.
\end{eqnarray}
This means that we obtain a linear convergence rate up to a solution level that is proportional to the upper bound $\sigma$\footnote{Several similar more specific assumptions can be made for the upper bound $\sigma_k$. For example if the upper bound satisfies $\sigma_k=\sigma^k$ with $\sigma \in (0,1)$ for all $k>0$ then it can be shown that $C\in (0,1)$ exist such that inequality \eqref{Theorem1} of Theorem \ref{InexactSGDConstant} takes the form: $\Exp[\|x_{k}-x_*\|_{\mB}] \leq   O(C^k) $ (see \cite{so2017non, friedlander2012hybrid} for similar results). }.
\end{cor}
\begin{proof}
See Appendix~\ref{Appendix2}.
\end{proof}

Inspired from \cite{friedlander2012hybrid}, let us now analyze iBasic using the sequence of upper bounds that described in Assumption~\ref{Assumption3}. This construction of the upper bounds  allows us to obtain stronger and more controlled results. In particular using the upper bound of Assumption~\ref{Assumption3} the sequence of expected errors converge linearly to the exact $x_*$ (not in a potential neighborhood like the previous result). In addition Assumption~\ref{Assumption3} guarantees that the distance to the optimal solution reduces with the increasing of the number of iterations. However for this stronger convergence a bound for $\lambda_{\min}^+$ is required, a quantity that in many problems is unknown to the user or intractable to compute. Nevertheless, there are cases that this value has a close form expression and can be computed before hand without any further cost. See for example \cite{LoizouRichtarik,ConsensusSHB,loizou2018provably,hanzely2019privacy} where methods for solving the average consensus were presented and the value of $\lambda_{\min}^+$ corresponds to the algebraic connectivity of the network under study.
\begin{thm}
\label{ISGDwithq}
Assume exactness. Let $\{x_k\}_{k=0}^\infty$ be the iterates produced by iBasic with $\omega\in(0,2)$. Set $x_*= \Pi_{\cL, \mB}(x_0)$ and consider the inexactness error $\epsilon_k$ be such that it satisfies Assumption \ref{Assumption3},
with $0\leq q < 1-\sqrt{\rho}$. Then 
\begin{eqnarray}\label{Theorem2withq}
\Exp[\|x_{k}-x_*\|_{\mB}^2]\leq \left(\sqrt{\rho}+q\right)^{2k} \|x_0-x_*\|_{\mB}^2.
\end{eqnarray}
\end{thm}
\begin{proof}
See Appendix~\ref{Appendix3}.
\end{proof}

At Theorem~\ref{ISGDwithq}, to guarantee linear convergence the \emph{inexact parameter} $q$ should live in the interval $\left[0,1-\sqrt{\rho}\right)$. In particular, $q$ is the parameter that controls the level of inexactness of Algorithm \ref{inexact_basic}. Not surprisingly the fastest convergence rate is obtained when $q=0$; in such case the method becomes equivalent with its exact variant and the convergence rate simplifies to $\rho=1- \omega (2-\omega)\lambda_{\min}^+$.  Note also that similar to the exact case the optimal convergence rate is obtained for $\omega=1$ \cite{ASDA}.

Moreover, the upper bound $\sigma_k$ of Assumption~\ref{Assumption3} depends on two important quantities, the $\lambda_{\min}^+$ (through the upper bound of the inexactness parameter $q$) and the distance to the optimal solution $\|x_k-x_*\|^2_{\bB}$. Thus, it can have natural interpretation. In particular the inexactness error is allowed to be large either when the current iterate is far from the optimal solution ($\|x_k-x_*\|^2_{\bB}$ large) or when the problem is well conditioned and $\lambda_{\min}^+$ is large. In the opposite scenario, when we have ill conditioned problem or we are already close enough to the optimum $x_*$ we should be more careful and allow less errors to the updates of the method. 

In the next theorem we provide the complexity results of iBasic in the case that the Assumption~\ref{Assumption5} is satisfied combined with one of the previous assumptions.
\begin{thm}
\label{InexactSGDrandom}
Let assume exactness and let $\{x_k\}_{k=0}^\infty$ be the iterates produced by iBasic with $\omega\in(0,2)$.  Set $x_*= \Pi_{\cL,\mB}(x_0)$. Let also assume that the inexactness error $\epsilon_k$ be such that it satisfies Assumption~\ref{Assumption5}. Then:
\begin{enumerate}
\item[(i)] If Assumption \ref{Assumption2} holds: 
\begin{eqnarray}
\label{klasnaso}
\Exp[\|x_{k}-x_*\|_{\mB}^2] \leq  \rho^{k} \|x_{0}-x_*\|^2_{\mB} +  \sum_{i=0}^{k-1} \rho^{k-1-i}{\bar{\sigma}}^2_i,
\end{eqnarray}
where $\bar{\sigma}_i^2=\Exp[\sigma_i^2], \forall i \in [k-1].$
\item[(ii)] If Assumption \ref{Assumption3} holds with $q \in \left(0,\sqrt{\rho}\right)$: 
\begin{eqnarray}
\label{jaksxal}
\Exp[\|x_{k}-x_*\|_{\mB}^2] & \leq& (\rho + q^2)^k \|x_0-x_*\|_{\mB}^2. 
\end{eqnarray}
\item[(iii)] If Assumption \ref{Assumption4} holds with $q \in \left(0, \sqrt{\omega(2-\omega)}\right)$: 
\begin{equation}
\Exp[\|x_{k}-x_*\|_{\mB}^2]\leq(1- (\omega (2-\omega)-q^2 )\lambda_{\min}^+ )^k \|x_0-x_*\|_{\mB}^2=(\rho+q^2\lambda_{\min}^+ )^k \|x_0-x_*\|_{\mB}^2.
\end{equation}
\end{enumerate} 
\begin{proof}
See Appendix~\ref{Appendix4}.
\end{proof}
\end{thm}

\begin{rem}
In the case that Assumptions~\ref{Assumption1} and \ref{Assumption5} hold simultaneously, the convergence of iBasic is similar to \eqref{klasnaso} but in this case ${\bar{\sigma}}^2_i=\sigma_i^2,\, \forall i \in [k-1]$ (due to Assumption~\ref{Assumption1}, $\sigma_k \in \R$ is a sequence of real numbers). In addition, note that for $q \in (0, \min\{\sqrt{\rho}, 1-\sqrt{\rho}\})$ having Assumption~\ref{Assumption5} on top of Assumption~\ref{Assumption3} leads to improvement of the convergence rate. In particular, from Theorem~\ref{ISGDwithq}, iBasic converges with rate $(\sqrt{\rho}+q)^{2}= \rho+q^2+2\sqrt{\rho}q$ while having both assumptions this is simplified to the faster $\rho + q^2$ \eqref{jaksxal}. 
\end{rem}

\section{iBasic with Structured Inexactness Error}
\label{InexactSolvers}
Up to this point, the analysis of iBasic was focused in more general abstract cases where the inexactness error $\epsilon_k$ of the update rule satisfies several general assumptions.  In this section we are focusing on a more structured form of inexactness error and we provide convergence analysis in the case that a linearly convergent algorithm is used for the computation of the expensive key subproblem of the method. 

\subsection{Linear System in the Update Rule}
\label{sectionlinesyst}

As we already mentioned in Section~\ref{subproblems} the update rule of the exact Basic method (Algorithm~\ref{inexact_basic} with $\epsilon_k=0$) can be expressed as $ x_{k+1} = x_k+\omega \mB^{-1}\mA^\top \mS_{k} \lambda_k^*$, where $\lambda_k^*= (\mS_{k}^\top \mA \mB^{-1} \mA^\top \mS_{k})^{\dagger} \mS_{k}^\top(b-\mA x_k)$.

Using this expression the exact Basic method can be equivalently interpreted as the following two step procedure:
\begin{enumerate}
\item Find the least norm solution\footnote{We are precisely looking for the least norm solution of the linear system $\bM_k  \lambda =d_k$ because this solution can be written down in a compact way using the Moore-Penrose pseudoinverse.  This is equivalent with the expression that appears in our update: $\lambda_k^*=(\mS_{k}^\top \mA \mB^{-1} \mA^\top \mS_{k})^{\dagger} \mS_{k}^\top(b-\mA x_k)=\bM_k^{\dagger} d_k $. However it can be easily shown that the method will still converge with the same rate of convergence even if we choose any other solution of the linear system $\bM_k  \lambda =d_k$.} of $
\underbrace{\mS_{k}^\top \mA \mB^{-1} \mA^\top \mS_{k}}_{\bM_k} \lambda = \underbrace{\mS_{k}^\top(b-\mA x_k)}_{d_k}
$. That is find $ \lambda_k^*=\arg\min_{ \lambda \in \cQ_k} \| \lambda\|$ where $\cQ_k= \left\{ \lambda \in \R^q : \bM_k  \lambda =d_k\right\}$.
\item Compute the next iterate:  $x_{k+1}=x_k + \omega \mB^{-1} \mA ^\top \mS_{k}  \lambda_k^*.$
\end{enumerate}

In the case that the random matrix $\bS_k$ is large (this is the case that we are interested in), solving exactly the linear system $\bM_k  \lambda=d_k$ in each step can be prohibitively expensive. To reduce this cost we allow the inner linear system $\bM_k  \lambda=d_k$ to be solved inexactly using an iterative method. In particular we propose and analyze the following inexact algorithm: 

\begin{algorithm}[H]
  \caption{iBasic with structured inexactness error
    \label{inexact_solver_algorithm}}
  \begin{algorithmic}[1]
    \Require{Distribution $\cD$ from which we draw random matrices $\bS$, positive definite matrix $\bB\in\R^{n\times n}$, stepsize $\omega>0$.}
    \Ensure{$x_0\in\R^n$}
 \For{$k=1,2,\cdots$}
 \State Generate a fresh sample $\bS_k \sim {\cal D}$
 \State  Using an iterative method compute an approximation $\lambda_k^{\approx}$ of the least norm solution of the linear system:
 \begin{equation}
 \label{linearsysteminCode}
\underbrace{\mS_{k}^\top \mA \mB^{-1} \mA^\top \mS_{k}}_{\bM_k} \lambda = \underbrace{\mS_{k}^\top(b-\mA x_k)}_{d_k}.
\end{equation}
 \State Set $x_{k+1}=x_k + \omega \mB^{-1} \mA ^\top \mS_{k} \lambda_k^{\approx}$.
 \EndFor
 \end{algorithmic}
\end{algorithm}

For the computation of the inexact solution of the linear system \eqref{linearsysteminCode} any known iterative method for solving general linear systems can be used. In our analysis we focus on linearly convergent methods. For example based on the properties of the linear system \eqref{linearsysteminCode},  conjugate gradient (CG) or sketch and project method (SPM) can be used for the execution of step 3. In these cases, we name Algorithm~\ref{inexact_solver_algorithm},  \emph{InexactCG} and \emph{InexactSP } respectively.

It is known that the classical CG can solve linear systems with positive definite matrices. In our approach matrix $\bM_k $ is positive definite only when the original linear system $\bA x= b$ has full rank matrix $\bA$. On the other side SPM can solve any consistent linear system and as a result can solve the inner linear system $\bM_k  \lambda_k=d_k$ without any further assumption on the original linear system. In this case, one should be careful because the system has no unique solution. We are interested to find the least norm solution of $\bM_k  \lambda_k=d_k$ which means that the starting point of the sketch and project at the $k^{th}$ iteration should be always $\lambda_k^0=0$. Recall that any special case of the sketch and project method (Section~\ref{Special Cases}) solves the best approximation problem.

Let us now define $\lambda_k^r$ to be the approximate solution $\lambda_k^{\approx}$ of the $q \times q$ linear system \eqref{linearsysteminCode} obtained after $r$ steps of the linearly convergent iterative method.  Using this, the update rule of Algorithm~\ref{inexact_solver_algorithm}, takes the form:
\begin{equation}
\label{inesacdwdad}
x_{k+1}=x_k+ \omega \mB^{-1} \mA ^\top \mS_{k} \lambda_k^r.
\end{equation}

\begin{rem}
\label{corresINexacterror}
The update rule \eqref{inesacdwdad} of Algorithm~\ref{inexact_solver_algorithm} is equivalent to the update rule of iBasic (Algorithm~\ref{inexact_basic}) when the error $\epsilon_k$ is chosen to be,
\begin{equation}
\label{specEpsilon}
\epsilon_k=\omega \bB^{-1}\bA^\top \bS_k (\lambda_k^r-\lambda_k^*).
\end{equation}
This is precisely the connection between the abstract and more concrete/structured notion of inexactness that first presented in Table~\ref{KeySubproblems}.
\end{rem}

Let us now define a Lemma that is useful for the analysis of this section and it verifies that Algorithm \ref{inexact_solver_algorithm} with unit stepsize satisfies the general Assumption \ref{Assumption5} presented in Section~\ref{asssad}. 

\begin{lem}
\label{lemmaFigure}
Let us denote $x_k^*=\Pi_{\cL_{\bS_k},\bB}(x_k)$ the projection of $x_k$ onto $\cL_{\bS_k}$ in the $\mB$-norm and $x_*= \Pi_{\cL,\mB}(x_0)$. Let also assume that $\omega=1$ (unit stepsize). Then for the updates of Algorithm~\ref{inexact_solver_algorithm} it holds that:
\begin{equation}
\label{iandia}
 \left\langle x_k^*-x_*, \epsilon_k \right \rangle_{\mB}=\left\langle (\mI - \omega \mB^{-1}\mZ_k) (x_k-x_*), \epsilon_k \right \rangle_{\mB}=0, \quad \forall k \geq 0.
\end{equation}
\end{lem}
\begin{proof}
Note that $x_k^*-x_*=x_k-\nabla f_{\mS_k}(x_k)-x_* \in Null (\bS_k^\top \bA)$ . Moreover $\epsilon_k\overset{\eqref{specEpsilon}}=\bB^{-1}\bA^\top \bS_k (\lambda_k^r - \lambda_k^*) \in Range(\bB^{-1} \bA^\top \bS_k)$. From the knowledge that the null space of an arbitrary matrix is the orthogonal complement of the range space of its transpose we have that $Null (\bS_k^\top \bA)$ is orthogonal with respect to the $\bB$-inner product to $Range(\bB^{-1} \bA^\top \bS_k)$. This completes the proof (see Figure~\ref{SketchProject} for the graphical interpretation).
\end{proof}

\begin{figure}[t!]
  \centering
\includegraphics[width=8cm]{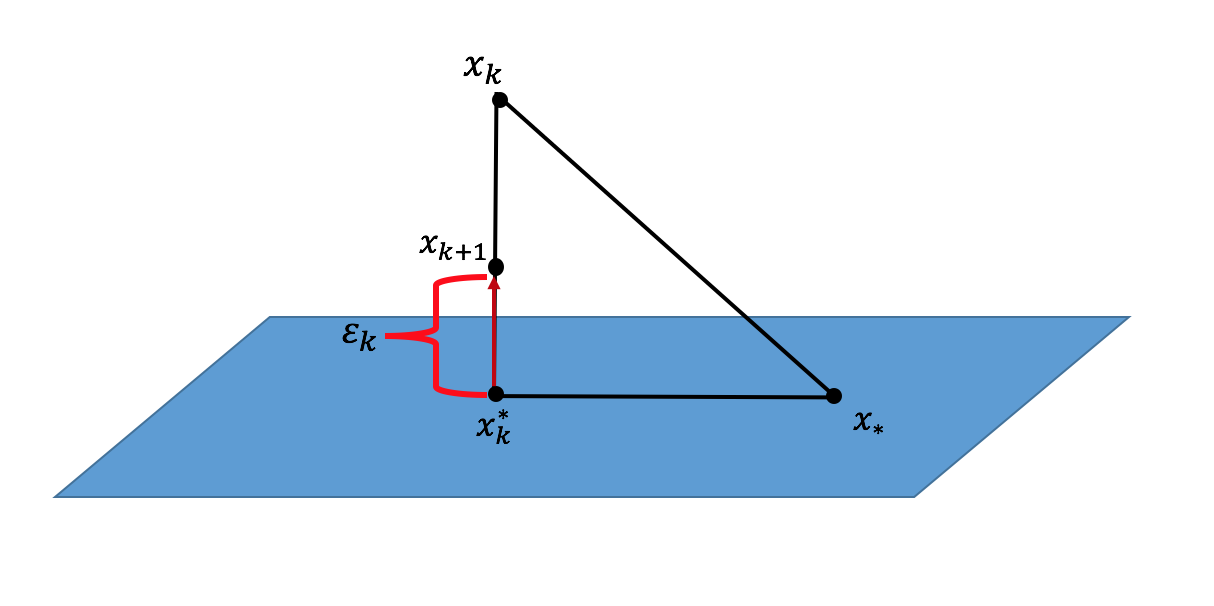}
\caption{\small Graphical interpretation of orthogonality (justifies equation \eqref{iandia}). It shows that the two vectors, $x_k^*-x_*$ and $\epsilon_k$, are orthogonal complements of each other with respect to the $\bB$-inner product. $x_{k+1}$ is the point that Algorithm \ref{inexact_solver_algorithm} computes in each step. The colored region represents the $Null (\bS_k^\top \bA)$. $x_k^*=\Pi_{\cL_{\bS_k}, \bB}(x_k)$, $x_*= \Pi_{\cL, \mB}(x_0)$ and $\epsilon_k$ is the inexactness error. }
\label{SketchProject}
\end{figure}

\subsection{Sketch and Project Interpretation}
\label{SFPinterpretation}
Let us now give a different interpretation of the inexact update rule of Algorithm~\ref{inexact_solver_algorithm} using the sketch and project approach. That will make us appreciate more the importance of the dual viewpoint and make clear the connection between the primal and dual methods.

Recall that in the special case of unit stepsize (see equation \eqref{spstepsize1}) the exact sketch and project method perform updates of the form:
\begin{equation}
\label{innerbestapprox}
x_{k+1}= \text{argmin}_{x\in \R^n} \tfrac{1}{2}\|x-x_k\|_\mB^2 \quad \text{subject to} \quad \bS_k^\top \mA x = \bS_k^\top b .
\end{equation}
That is, a {\em sketched} system $ \mS^\top \mA x =  \mS^\top b$ is first chosen and then a the next iterate is computed by making a projection of the current iterate $x_k$ onto this system. 

In general, execute a projection step is one of the most common task in numerical linear algebra/optimization literature. However in the large scale setting even this task can be prohibitively expensive and it can be difficult to execute inexactly.  For this reason we suggest to move to the dual space where the inexactness can be easily controlled. 

Observe that the update rule of equation \eqref{innerbestapprox} has the same structure as the best approximation problem \eqref{best approximation} where the linear system under study is the sketched system $\bS_k^\top \mA x = \bS_k^\top b$ and the starting point is the current iterate $ x_k$. Hence we can easily compute its dual:
\begin{equation}
\label{innerdualprobalme}
\max_{\lambda \in \R^q} D_k(\lambda) \eqdef (\bS_k^\top b - \bS_k^\top \bA x_k)^\top \lambda - \tfrac{1}{2}\|\bA^\top \bS_k \lambda\|^2_{\bB^{-1}}.
\end{equation}
where $\lambda \in \R^q$ is the dual variable. The $\lambda_k^*$ (possibly more than one) that solves the dual problem in each iteration $k$, is the one that satisfies $\nabla D_k (\lambda_k^*)=0$. By computing the derivative this is equivalent with finding the $\lambda$ that satisfies the linear system $\mS_{k}^\top \mA \mB^{-1} \mA^\top \mS_{k} \lambda = \mS_{k}^\top(b-\mA x_k)$. This is the same linear system we desire to solve inexactly in Algorithm \ref{inexact_solver_algorithm}. Thus, computing an inexact solution $\lambda_k^{\approx}$ of the linear system is equivalent with computing an inexact solution of the dual problem \eqref{innerdualprobalme}. Then by using the affine mapping \eqref{eq:dual-corresp} that connects the primal and the dual spaces we can also evaluate an inexact solution of the original primal problem \eqref{innerbestapprox}.

The following result relates the inexact levels of these quantities. In particular it shows that dual suboptimality of $\lambda_k$ in terms of dual function values is equal to the distance of the dual values $\lambda_k$ in the $\bM_k$-norm.
\begin{lem}
\label{lemmaD(y)}
Let us define $\lambda_k^* \in \R^q$ be the exact solution of the linear system $\mS_{k}^\top \mA \mB^{-1} \mA^\top \mS_{k} \lambda = \mS_{k}^\top(b-\mA x_k)$ or equivalently of dual problem \eqref{innerdualprobalme}. Let us also denote with $\lambda_k^{\approx} \in \R^q$ the inexact solution. Then:
$$D_k(\lambda_k^*)-D_k(\lambda_k^{\approx})= \frac{1}{2}\|\lambda_k^{\approx}-\lambda_k^*\|^2_{\mS_{k}^\top \mA \mB^{-1} \mA^\top \mS_{k}}.$$
\end{lem}

\begin{proof}
\begin{eqnarray*}
D_k(\lambda_k^*)-D_k(\lambda_k^{\approx}) & \overset {\eqref{innerdualprobalme}}{=} &  [\bS_k^\top b - \bS_k^\top \bA x_k]^\top [\lambda_k^*-\lambda_k^{\approx}]-\frac{1}{2} (\lambda_k^*)^\top \mS_{k}^\top \mA \mB^{-1} \mA^\top \mS_{k} \lambda_k^*\notag\\ 
&&+\frac{1}{2} (\lambda_k^{\approx})^\top \mS_{k}^\top \mA \mB^{-1} \mA^\top \mS_{k} \lambda_k^{\approx} \notag\\
& \overset {\eqref{eq:dual-corresp}}{=} & (\lambda_k^*)^\top \mS_{k}^\top \mA \mB^{-1} \mA^\top \mS_{k}  [\lambda_k^*-\lambda_k^{\approx}]-\frac{1}{2} (\lambda_k^*)^\top \mS_{k}^\top \mA \mB^{-1} \mA^\top \mS_{k} \lambda_k^*\notag\\ 
&&+ \frac{1}{2} (\lambda_k^{\approx})^\top \mS_{k}^\top \mA \mB^{-1} \mA^\top \mS_{k} \lambda_k^{\approx} \notag\\
& = & \frac{1}{2}(\lambda_k^{\approx}-\lambda_k^*)^\top \mS_{k}^\top \mA \mB^{-1} \mA^\top \mS_{k}(\lambda_k^{\approx}-\lambda_k^*) \notag\\
& = & \frac{1}{2}\|\lambda_k^{\approx}-\lambda_k^*\|^2_{\mS_{k}^\top \mA \mB^{-1} \mA^\top \mS_{k}} 
\end{eqnarray*}
where in the second equality we use equation \eqref{eq:dual-corresp} to connect the optimal solutions of \eqref{innerbestapprox} and \eqref{innerdualprobalme} and obtain $[\bS_k^\top b - \bS_k^\top \bA x_k]^\top=(\lambda_k^*)^\top \mS_{k}^\top \mA \mB^{-1} \mA^\top \mS_{k}.$
\end{proof}

\subsection{Complexity Results}
\label{MainTheory}
In this part we analyze the performance of Algorithm \ref{inexact_solver_algorithm} when a linearly convergent iterative method is used for solving inexactly the linear system \eqref{linearsysteminCode} in step 3 of Algorithm \ref{inexact_solver_algorithm} . We denote with $\lambda_k^r$ the approximate solution of the linear system after we run the iterative method for $r$ steps.

Before state the main convergence result let us present a lemma that summarize some observations that are true in our setting.
\begin{lem}
Let $\lambda_k^*=(\mS_{k}^\top \mA \mB^{-1} \mA^\top \mS_{k})^{\dagger} \mS_{k}^\top(b-\mA x_k)$  be the exact solution and $\lambda_k^r$ be approximate solution of the linear system \eqref{linearsysteminCode}. Then, $\|\lambda_k^*\|^2_{\bM_k}=2 f_{\bS_k}(x_k) $ and $\|\epsilon_k\|_{\mB}^2= \|\lambda_k^r -\lambda_k^*\|_{\bM_k}^2$.
\begin{proof}
\begin{eqnarray}
\label{asdasfarf}
\|\lambda_k^*\|^2_{\bM_k} &=& \|\bM_k^\dagger \bS_k^\top \bA (x_*-x_k) \|^2_{\bM_k} = (x_k-x_*)^\top \bA^\top \bS_k \underbrace{\bM_k^{\dagger} \bM_k \bM_k^{\dagger}}_{ \bM_k^{\dagger}} \mS_{k}^\top \mA  (x_k-x_*) \notag\\
 &\overset{\eqref{ZETA}}=& (x_k-x_*)^\top \bZ_k (x_k-x_*) \overset{\eqref{eq:f_s22}}= 2 f_{\bS_k}(x_k). 
\end{eqnarray}
Moreover,
\begin{eqnarray}
\label{noasdnk}
\|\epsilon_k\|_{\mB}^2 &\overset{Remark~\ref{corresINexacterror}}=& \|\bB^{-1}\bA^\top \bS_k (\lambda_k^r -\lambda_k^*)\|_{\mB}^2 
=  \|\lambda_k^r -\lambda_k^*\|_{\mS_{k}^\top \mA \mB^{-1} \mA ^\top \mS_{k}}^2
=   \|\lambda_k^r -\lambda_k^*\|_{\bM_k}^2.
\end{eqnarray}
\end{proof}
\end{lem}

\begin{thm}
\label{MainTheorem}
Let us assume that for the computation of the inexact solution of the linear system \eqref{linearsysteminCode} in step 3 of Algorithm~\ref{inexact_solver_algorithm}, a linearly convergent iterative method is chosen such that \footnote{In the case that deterministic iterative method is used, like CG, we have that $\|\lambda_k^r-\lambda_k^*\|_{\bM_k}^2 \leq \rho_{\bS_k}^r \|\lambda_k^0-\lambda_k^*\|^2_{\bM_k}$ which is also true in expectation}: 
\begin{equation}
\label{linaelr}
\Exp[\|\lambda_k^r-\lambda_k^*\|_{\bM_k}^2 \;|\;x_k,\bS_k] \leq \rho_{\bS_k}^r \|\lambda_k^0-\lambda_k^*\|^2_{\bM_k}, 
\end{equation}
where $\lambda_k^0=0$ for any $k>0$ and $\rho_{\bS_k} \in (0,1)$ for every choice of $\bS_k \sim \cD$.
Let exactness hold and let $\{x_k\}_{k=0}^\infty$ be the iterates produced by Algorithm~\ref{inexact_solver_algorithm} with unit stepsize ($\omega=1$).  Set $x_*= \Pi_{\cL,\mB}(x_0)$. 
Suppose further that there exists a scalar $\theta <1$ such that with
probability 1,  $\rho_{\bS_k} \leq \theta $. Then,  Algorithm \ref{inexact_solver_algorithm} converges linearly with:
$$\Exp[\|x_{k}-x_*\|_{\mB}^2] \leq \left[ 1- \left(1 -\theta^r  \right) \lambda_{\min}^+ \right] ^k \|x_0-x_*\|_{\mB}^2. $$
\end{thm}
\begin{proof}
Theorem~\ref{MainTheorem} can be interpreted as corollary of the general Theorem~\ref{InexactSGDrandom}(iii). Thus, it is sufficient to show that Algorithm \ref{inexact_solver_algorithm} satisfies the two Assumptions~\ref{Assumption4} and \ref{Assumption5}. Firstly, note that from Lemma~\ref{lemmaFigure}, Assumption \ref{Assumption5} is true. Moreover, 
\begin{eqnarray*}
\Exp[\|\epsilon_k\|_{\bM_k}^2 \;|\;x_k,\bS_k] &\overset{\eqref{noasdnk}}=&\Exp[\|\lambda_k^r-\lambda_k^*\|_{\bM_k}^2 \;|\;x_k,\bS_k] \overset{\eqref{linaelr}}\leq \rho_{\bS_k}^r \|\lambda_k^0-\lambda_k^*\|^2_{\bM_k} \notag\\
&\leq &\theta^r \|\lambda_k^0-\lambda_k^*\|^2_{\bM_k} \overset{\lambda_k^0=0}{=}  \theta^r \|\lambda_k^*\|^2_{\bM_k}\overset{\eqref{asdasfarf}}{=}  2 \theta^r f_{\bS_k}(x_k) 
\end{eqnarray*}
which means that Assumption~\ref{Assumption4} also holds with $q = \theta^{r/2} \in(0,1)$. This completes the proof.
\end{proof}

Having present the main result of this section let us now state some remarks that will help understand the convergence rate of the last Theorem.
\begin{rem}
From its definition $\theta^r \in (0,1)$ and as a result $\left(1 -\theta^r  \right) \lambda_{\min}^+ \leq \lambda_{\min}^+ $. This means that the method converges linearly but always with worst rate than its exact variant. 
\end{rem}

\begin{rem}
Let as assume that $\theta$ is fixed. Then as the number of iterations in step 3 of the algorithm ($r\rightarrow\infty$) increasing $(1-\theta^r) \rightarrow 1$ and as a result the method behaves similar to the exact case.
\end{rem}

\begin{rem}
The $\lambda_{\min}^+$ depends only on the random matrices $\bS \sim \cD$ and to the positive definite matrix $\bB$ and is independent to the iterative process used in step 3. The iterative process of step 3 controls only the parameter $\theta$ of the convergence rate.
\end{rem}

\begin{rem}
Let us assume that we run Algorithm \ref{inexact_solver_algorithm} two separate times for two different choices of the linearly convergence iterative method of step 3. Let also assume that the distribution $\cD$ of the random matrices and the positive definite matrix $\bB$ are the same for both instances and that for step 3 the iterative method run  for $r$ steps for both algorithms. Let assume that $\theta_1< \theta_2$ then we have that $\rho_1=1- \left(1 -\theta_1^r  \right) \lambda_{\min}^+< 1- \left(1 -\theta_2^r  \right) \lambda_{\min}^+=\rho_2$. This means in the case that $\theta$ is easily computable, we should always prefer the inexact method with smaller $\theta$.
\end{rem}

The convergence of Theorem~\ref{MainTheorem} is quite general and it holds for any linearly convergent methods that can inexactly solve  \eqref{linearsysteminCode}. However, in case that the iterative method is known we can have more concrete results. See below the more specified results for the cases of Conjugate gradient (CG) and Sketch and project method (SPM). 

\paragraph{Convergence of InexactCG:}
CG is deterministic iterative method for solving linear systems $\bA x=b$ with symmetric and positive definite matrix $\bA \in \R^{n\times n}$ in finite number of iterations. In particular, it can be shown that converges to the unique solution in at most $n$ steps. The worst case behavior of CG is given by \cite{wright1999numerical,golub2012matrix} \footnote{A sharper convergence rate of CG \cite{wright1999numerical} for solving $\bA x=b$ can be also used 
$$
\|x_k-x_*\|^2_{\bA}\leq \left( \frac{\lambda_{n-k}-\lambda_1}{\lambda_{n-k}+\lambda_1}\right)^{2} \|x_0-x_*\|^2_{\bA},
$$
where matrix $\bA\in \R^{n\times n}$ has $\lambda_1\leq \lambda_2 \leq \dots \leq \lambda_n$ eigenvalues.}:
\begin{equation}
\label{conjgrad}
\|x_k-x_*\|_{\bA}\leq \left( \frac{\sqrt{\kappa(\bA)}-1}{\sqrt{\kappa(\bA)}+1}\right)^{2k} \|x_0-x_*\|_{\bA},
\end{equation}
where $ x_k$ is the $k^{th}$ iteration of the method and $\kappa(\bA)$ the condition number of matrix $\bA$.

Having present the convergence of CG for general linear systems, let us now return back to our setting. We denote $\lambda_k^r \in \R^q$ to be the approximate solution of the inner linear system \eqref{linearsysteminCode} after $r$ conjugate gradient steps. Thus using \eqref{conjgrad} we know that
$\|\lambda_k^r-\lambda_k^*\|_{\bM_k}^2 \leq \rho_{\bS_k}^{4r} \|\lambda_k^0-\lambda_k^*\|^2_{\bM_k},$
where $\rho_{\bS_k}= \left( \frac{\sqrt{\kappa(\bM_k)}-1}{\sqrt{\kappa(\bM_k)}+1}\right)$. Now by making the same assumption as the general Theorem~\ref{MainTheorem} the InexactCG converges with
$\Exp[\|x_{k}-x_*\|_{\mB}^2] \leq \left[ 1- \left(1 -\theta_{CG}^r  \right) \lambda_{\min}^+ \right] ^k \|x_0-x_*\|_{\mB}^2, $
where $\theta_{CG} <1$ such that $\rho_{\bS_k}=\left( \frac{\sqrt{\kappa(\bM_k)}-1}{\sqrt{\kappa(\bM_k)}+1}\right)^{4} \leq \theta_{CG}$ with probability 1.

\paragraph{Convergence of InexactSP:}
In this setting we suggest to run the sketch and project method (SPM) for solving inexactly the linear system \eqref{linearsysteminCode}. This allow us to have no assumptions on the structure of the original system $\bA x=b$ and as a result we are able to solve more general problems compared to what problems InexactCG can solve\footnote{Recall that InexactCG requires the matrix $\bM_k$ to be positive definite (this is true when matrix $\bA$ is a full rank matrix)}. Like before, by making the same assumptions as in Theorem~\ref{MainTheorem} the more specific convergence $\Exp[\|x_{k}-x_*\|_{\mB}^2] \leq \left[ 1- \left(1 -\theta_{SP}^r  \right) \lambda_{\min}^+ \right] ^k \|x_0-x_*\|_{\mB}^2, $ for the InexactSP can be obtained. Now the quantity $\rho_{\bS_k}$ denotes the convergence rate of the exact Basic method\footnote{Recall that iBasic and its exact variant ($\epsilon_k=0$) can be expressed as sketch and project methods \eqref{anska}.} when this applied to solve linear system \eqref{linearsysteminCode} and $\theta_{SP} <1$ is a scalar such that $\rho_{\bS_k} \leq \theta_{SP}$ with probability 1.

\section{Inexact Dual Method}
\label{InexactDualMethods}
In the previous sections we focused on the analysis of inexact stochastic methods for solving the stochastic optimization problem \eqref{eq:stoch_reform} and the best approximation \eqref{best approximation}.  In this section we turn into the dual of the best approximation \eqref{dual} and we propose and analyze an inexact variant of the SDSA \eqref{SDSAalg}. We call the new method iSDSA and is formalized as Algorithm \ref{inexact_dual}. In the update rule $\epsilon^d_k$ indicates the dual inexactness error that appears in the $k^{th}$ iteration of iSDSA.

\begin{algorithm}[H]
  \caption{Inexact Stochastic Dual Subspace Ascent (iSDSA)
    \label{inexact_dual}}
  \begin{algorithmic}[1]
    \Require{Distribution $\cD$ from which we draw random matrices $\bS$, positive definite matrix $\bB\in\R^{n\times n}$, stepsize $\omega>0$.}
    \Ensure{$y_0=0 \in \R^m$, $x_0 \in \R^n$}
 \For{$k=1,2,\cdots$}
 \State Draw a fresh sample $\bS_k \sim \cD$
 \State  Set $y_{k+1}=y_k+\omega \mS_k \left(\mS_k^\top \bA \bB^{-1}\bA^\top \mS_k \right)^\dagger\bS_k^\top \left(b-\bA(x_0 + \bB^{-1}\bA^\top y_k) \right) + \epsilon^d_k$
 \EndFor
 \end{algorithmic}
\end{algorithm}

\subsection{Correspondence Between the Primal and Dual Methods}
With the sequence of the dual iterates $\{y_k\}_{k=0}^\infty$ produced by the iSDSA we can associate a sequence of primal iterates  $\{x_k\}_{k=0}^\infty$ using the affine mapping \eqref{eq:dual-corresp}.  In our first result we show that the random iterates produced by iBasic arise as an affine image of iSDSA under this affine mapping.

\begin{thm}(Correspondence between the primal and dual methods)
\label{LemmaCorrespondance}
Let $\{x_k\}_{k=0}^\infty$  be the iterates produced by iBasic (Algorithm~\ref{inexact_basic}). Let $y_0=0$, and $\{y_k\}_{k=0}^\infty$ the iterates of the iSDSA. Assume that the two methods use the same stepsize $\omega >0 $ and the same sequence of random matrices $\bS_k$. Assume also that $\epsilon_k=\bB^{-1}\bA^\top \epsilon^d_k$ where $\epsilon_k$ and $\epsilon^d_k$ are the inexactness errors appear in the update rules of iBasic and iSDSA respectively. Then $$x_k=\phi(y_k)=x_0+\bB^{-1}\bA^\top y_k.$$ for all $k\geq0$. That is, the primal iterates arise as affine images of the dual iterates.
\end{thm}
\begin{proof}
\begin{eqnarray*}
\phi(y_{k+1}) &\overset{\eqref{eq:dual-corresp}}{=}& x_0  + \bB^{-1}\bA^\top y_{k+1} \overset{ \eqref{lambdak},\text{Alg.}\ref{inexact_dual}}{=} x_0+\bB^{-1}\bA^\top\left[y_k+\omega \bS_k \lambda_k +\epsilon^d_k \right]\\
&\overset{\eqref{ZETA}, \eqref{lambdak} }{=}&  \underbrace{x_0+\bB^{-1}\bA^\top y_k}_{\phi(y_k)}+\omega \bB^{-1} \bZ_k \left(x_*-( \underbrace{x_0+\bB^{-1}\bA^\top y_k}_{\phi(y_k)}) \right)+\bB^{-1}\bA^\top \epsilon^d_k\\
&=& \phi(y_k)- \omega \bB^{-1}\bZ_k( \phi(y_k)-x_*) +\bB^{-1}\bA^\top \epsilon^d_k\\
\end{eqnarray*}
Thus by choosing the inexactness error of the primal method to be $\epsilon_k=\bB^{-1}\bA^\top \epsilon^d_k$ the sequence of vectors  $\{\phi(y_k)\}$ satisfies the same recursion as the sequence $\{x_k\}$ defined by iBasic.  It remains to check that the first element of both recursions coincide. Indeed, since $y_0=0$, we have $x_0 = \phi(0) =  \phi(y_0)$.
\end{proof}

\subsection{iSDSA with Structured Inexactness Error}
In this subsection we present Algorithm~\ref{inexact_solver_Dualalgorithm}. It can be seen as a special case of iSDSA but with a more structured inexactness error. 

\begin{algorithm}[H]
  \caption{iSDSA with structured inexactness error
    \label{inexact_solver_Dualalgorithm}}
  \begin{algorithmic}[1]
    \Require{Distribution $\cD$ from which we draw random matrices $\bS$, positive definite matrix $\bB\in\R^{n\times n}$, stepsize $\omega>0$.}
    \Ensure{$y_0=0\in\R^m$, $x_0 \in \R^n$}
 \For{$k=1,2,\cdots$}
 \State Generate a fresh sample $\bS_k \sim {\cal D}$
 \State Using an Iterative method compute an approximation $\lambda_k^{\approx}$ of the least norm solution of the linear system:
 \begin{equation}
\underbrace{\mS_{k}^\top \mA \mB^{-1} \mA^\top \mS_{k}}_{\bM_k} \lambda =  \underbrace{\mS_{k}^\top(b-\bA(x_0 + \bB^{-1}\bA^\top y_k)}_{d_k}
\end{equation}
 \State Set $y_{k+1}=y_k + \omega \mS_{k} \lambda_k^{\approx}$
 \EndFor
 \end{algorithmic}
\end{algorithm}

Similar to their primal variants, it can be easily checked that Algorithm~\ref{inexact_solver_Dualalgorithm} is a special case of the iSDSA ( Algorithm~\ref{inexact_dual}) when the dual inexactness error is chosen to be $\epsilon_k^d= \bS_k(\lambda_k^r-\lambda_k^*)$. Note that, using the observation of Remark~\ref{corresINexacterror} that $\epsilon_k=\omega \bB^{-1}\bA^\top \bS_k (\lambda_k^r-\lambda_k^*)$ and the above expression of $\epsilon_k^d$ we can easily verify that the expression
$\epsilon_k=\bB^{-1}\bA^\top \epsilon^d_k$ holds.
This is precisely the connection between the primal and dual inexactness errors that have already been used in the proof of Theorem~\ref{LemmaCorrespondance}.

\subsection{Convergence of Dual Function Values}
We are now ready to state a linear convergence result describing the behavior of the inexact dual method in terms of the function values $D(y_k)$. The following result is focused on the convergence of iSDSA by making similar assumption to Assumption~\ref{Assumption3}. Similar convergence results can be obtained using any other assumption of Section~\ref{asssad}. The convergence of Algorithm~\ref{inexact_solver_Dualalgorithm}, can be also easily derived using similar arguments with the one presented in Section~\ref{InexactSolvers} and the convergence guarantees of Theorem~\ref{MainTheorem}.

\begin{thm}(Convergence of dual objective).
Assume exactness. Let $y_0=0$ and let $\{y_k\}_{k=0}^\infty$ to be the dual iterates of iSDSA (Algorithm~\ref{inexact_dual}) with $\omega \in (0,2)$.  Set $x_*= \Pi_{\cL,\mB}(x_0)$ and let $y_*$ be any dual optimal solution. Consider the inexactness error $\epsilon_k^d$ be such that it satisfies 
 $\Exp[\|\bB^{-1}\bA^\top\epsilon_k^d\|^2_{\bB}\;|\;y_k,\bS_k]\leq\sigma_k^2= q^2 2 \left[D(y_*)-D(y_k)\right]$
where $0\leq q < 1-\sqrt{\rho}$. Then 
\begin{eqnarray}
\Exp[D(y_*)-D(y_k)] \leq \left(\sqrt{\rho}+q\right)^{2k} \left[ D(y_*)-D(y_0)\right].
\end{eqnarray}
\end{thm}
\begin{proof}
The proof follows by applying Theorem~\ref{ISGDwithq} together with Theorem~\ref{LemmaCorrespondance} and the identity $\tfrac{1}{2}\|x_k-x_*\|^2_\mB = D(y_*) - D(y_k)$ \eqref{identity}.
\end{proof}

Note that in the case that $q=0$, iSDSA simplifies to its exact variant SDSA and the convergence rate coincide with the one presented in \cite{loizou2017momentum, gower2015stochastic}. Following similar arguments to those in \cite{gower2015stochastic}, the same rate can be proved for the duality gap $\Exp[P(x_k)-D(y_k)]$.

\section{Numerical Evaluation}
\label{experiments}
In this section we perform preliminary numerical tests for studying the computational behavior of iBasic with structured inexactness error when is used to solve the best approximation problem \eqref{best approximation} or equivalently the stochastic optimization problem \eqref{eq:stoch_reform}\footnote{Note that from Section~\ref{InexactDualMethods} and the correspondence between the primal and dual methods, iSDSA will have similar behavior when is applied to the dual problem \eqref{dual}.}. As we have already mentioned, iBasic can be interpreted as sketch-and-project method, and as a result a comprehensive array of well-known algorithms can be recovered as special cases by varying the main parameters of the methods (Section \ref{Special Cases}).  In particular, in our experiments we focus on the evaluation of two popular special cases, the inexact Randomized Block Kaczmarz (iRBK) (equation \eqref{iRBK}) and inexact randomized block coordinate descent method (iRBCD) (equation \eqref{iRBCD})
We implement Algorithm~\ref{inexact_solver_algorithm} presented in Section~\ref{InexactSolvers} using CG \footnote{Recall that in order to use CG, the matrix $\bM_k$ that appears in linear system \eqref{linearsysteminCode} should be positive definite. This is true in the case that the matrix $\bA$ of the original system has full column rank matrix. Note however that the analysis of Section~\ref{InexactSolvers} holds for any consistent linear system $\bA x=b$ and without making any further assumption on its structure or the linearly convergence methods.} to inexactly solve the linear system of the update rule (equation \eqref{linearsysteminCode}). Recall that in this case we named the method InexactCG.

The convergence analysis of previous sections is quite general and holds for several combinations of the two main parameters of the method, the positive definite matrix $\bB$ and the distribution $\cD$ of the random matrices $\bS$. For obtaining iRBK as special case we have to choose $\bB=\bI \in \R^{n \times n}$ (Identity matrix) and for the iRBCD the given matrix $\bA $ should be positive definite and choose $\bB=\bA$. For both methods the distribution $\cD$ should be over random matrices $\bS=\bI_{:C}$ where $\bI_{:C}$ is the column concatenation of the $m \times m$ identity matrix indexed by a random subset $C$ of $[m]$. In our experiments we choose to have one specific distribution over these matrices. In particular, we assume that the random matrix in each iteration is chosen uniformly at random to be $\bS=\bI_{:d}$ with the subset $d$ of $[m]$ to have fixed pre-specified cardinality.

The code for all experiments is written in the Julia 0.6.3 programming language and run on a Mac laptop computer (OS X El Capitan), 2.7 GHz Intel Core i5 with 8 GB of RAM.

To coincide with the theoretical convergence results of Algorithm~\ref{inexact_solver_algorithm} the relaxation parameter (stepsize) of the methods study in our experiments is chosen to be $\omega=1$ (no relaxation). In all implementations, we use $x_0=0 \in \R^n$ as an initial point and in comparing the methods with their inexact variants we use the relative error measure $\|x_k-x_*\|^2_\bB / \|x_0-x_*\|^2_\bB \overset{x_0=0}{=}\|x_k-x_*\|^2_\bB / \|x_*\|^2_\bB $. We run each method (exact and inexact) until the relative error is below $10^{-5}$. For the horizontal axis we use either the number of iterations or the wall-clock time measured using the tic-toc Julia function. In the exact variants, the linear system \eqref{linearsysteminCode} in Algorithm~\ref{inexact_solver_algorithm} needs to be solved exactly. In our experiments we follow the implementation of \cite{gower2015randomized} for both exact RBCD and exact RBK where the built-in direct solver (sometimes referred to as "backslash") is used. 

\paragraph{Experimental setup:}
For the construction of consistent linear systems $\bA x=b$ we use the following setup:
\begin{itemize}
\item \textbf{For iRBK:}
Let matrix $\bA \in \R^{m \times n}$ being given (it can be either synthetic or real data). Then a vector $z \in \R^n$ is chosen to be i.i.d $\mathcal{N}(0,1)$ and the right hand side of the linear system is set to $b=\bA z$. With this way the consistency of the linear system with matrix $\bA$ and right hand side $b$ is ensured. 
\item \textbf{For iRBCD:}  A Gaussian matrix $\bP \in \R^{m \times n}$ is generated and then matrix $\bA = \bP^\top \bP \in \R^{n \times n}$ is used in the linear system (with this way matrix $\bA$ is positive definite with probability 1). The vector $z\in \R^n$ is chosen to be i.i.d $\mathcal{N}(0,1)$ and again to ensure consistency of the linear system, the right hand side is set to $b=\bA z$.
\end{itemize}

\subsection{Importance of Large Block Size}
Many recent works have shown that using larger block sizes can be very beneficial for the performance of randomized iterative algorithms \cite{gower2015randomized, richtarik2014iteration, RBK, LoizouRichtarik}. In Figure~\ref{BlockSizeRBKRBCD} we numerically verify this statement. We show that both RBK and RBCD (no inexact updates) outperform in number of iterations and wall clock time their serial variants where only one coordinate is chosen (block of size $d=1$) per iteration. This justify the necessity of choosing methods with large block sizes. Recall that this is precisely the class of algorithms that could have an expensive subproblem in their update rule which is required to be solved exactly and as a result can benefit the most from the introduction of inexactness.

\begin{figure}[t!]
\centering
\begin{subfigure}{.24\textwidth}
  \centering
  \includegraphics[width=1\linewidth]{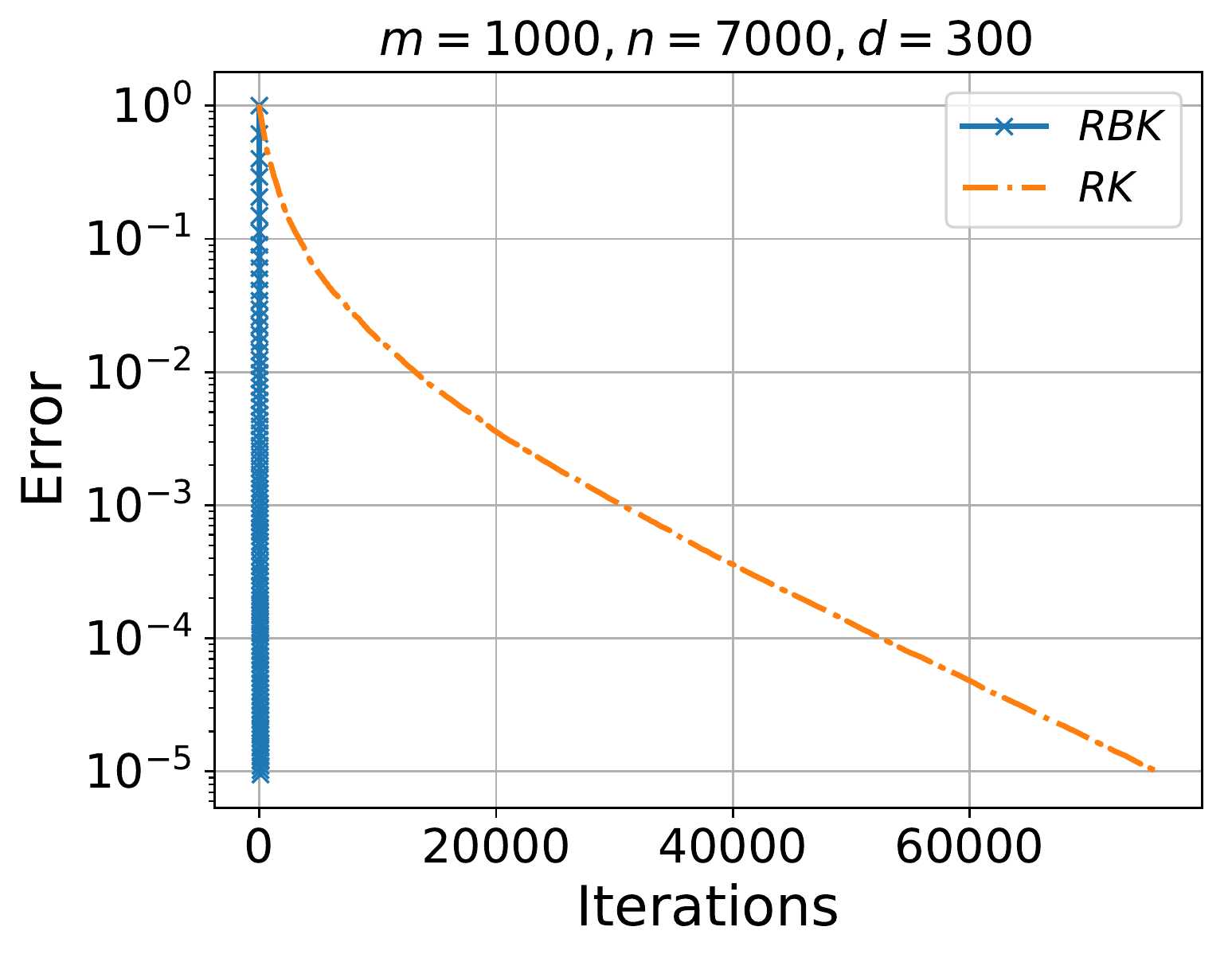}
\end{subfigure}%
\begin{subfigure}{.24\textwidth}
  \centering
  \includegraphics[width=1\linewidth]{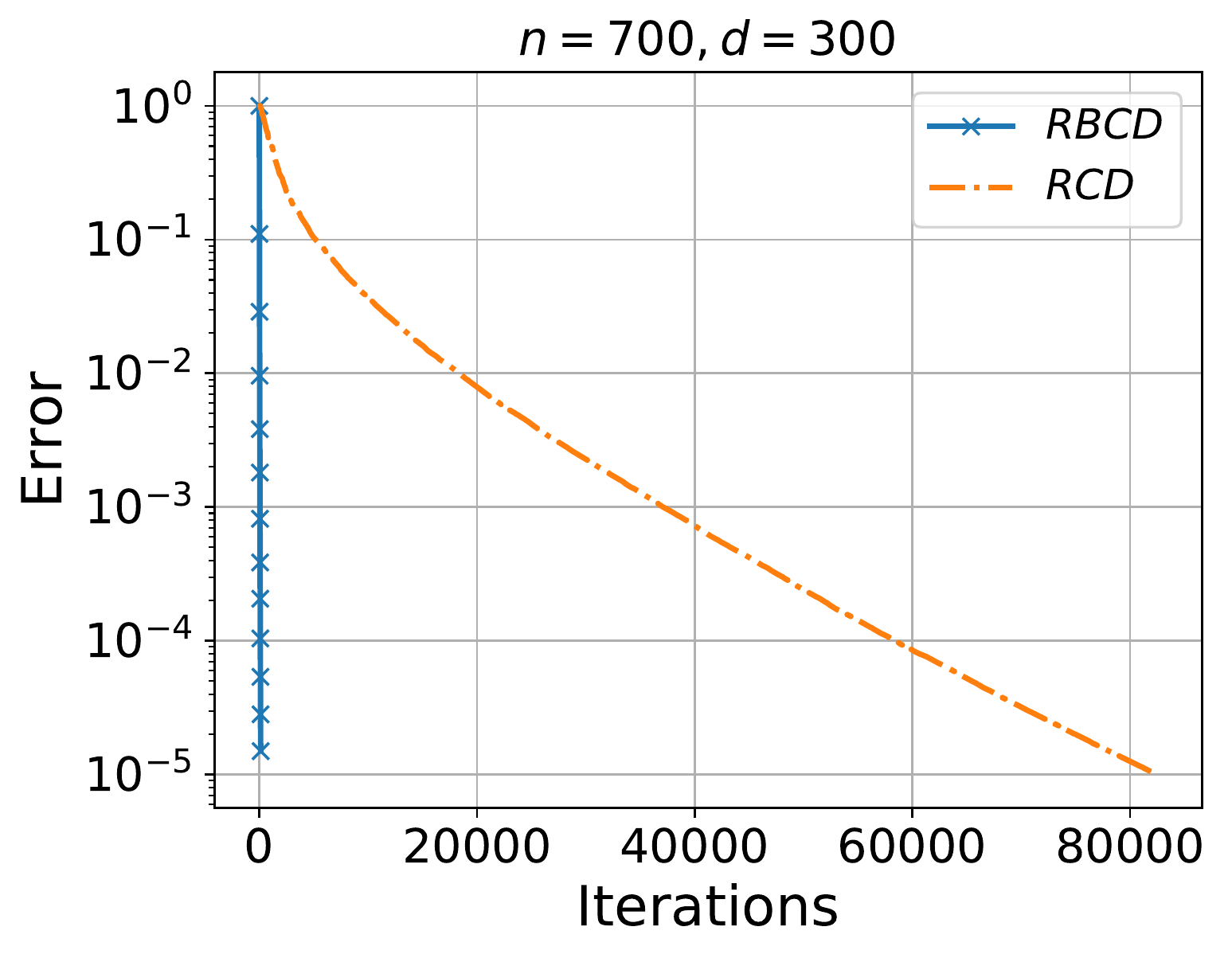}
\end{subfigure}\\
\begin{subfigure}{.24\textwidth}
  \centering
  \includegraphics[width=\linewidth]{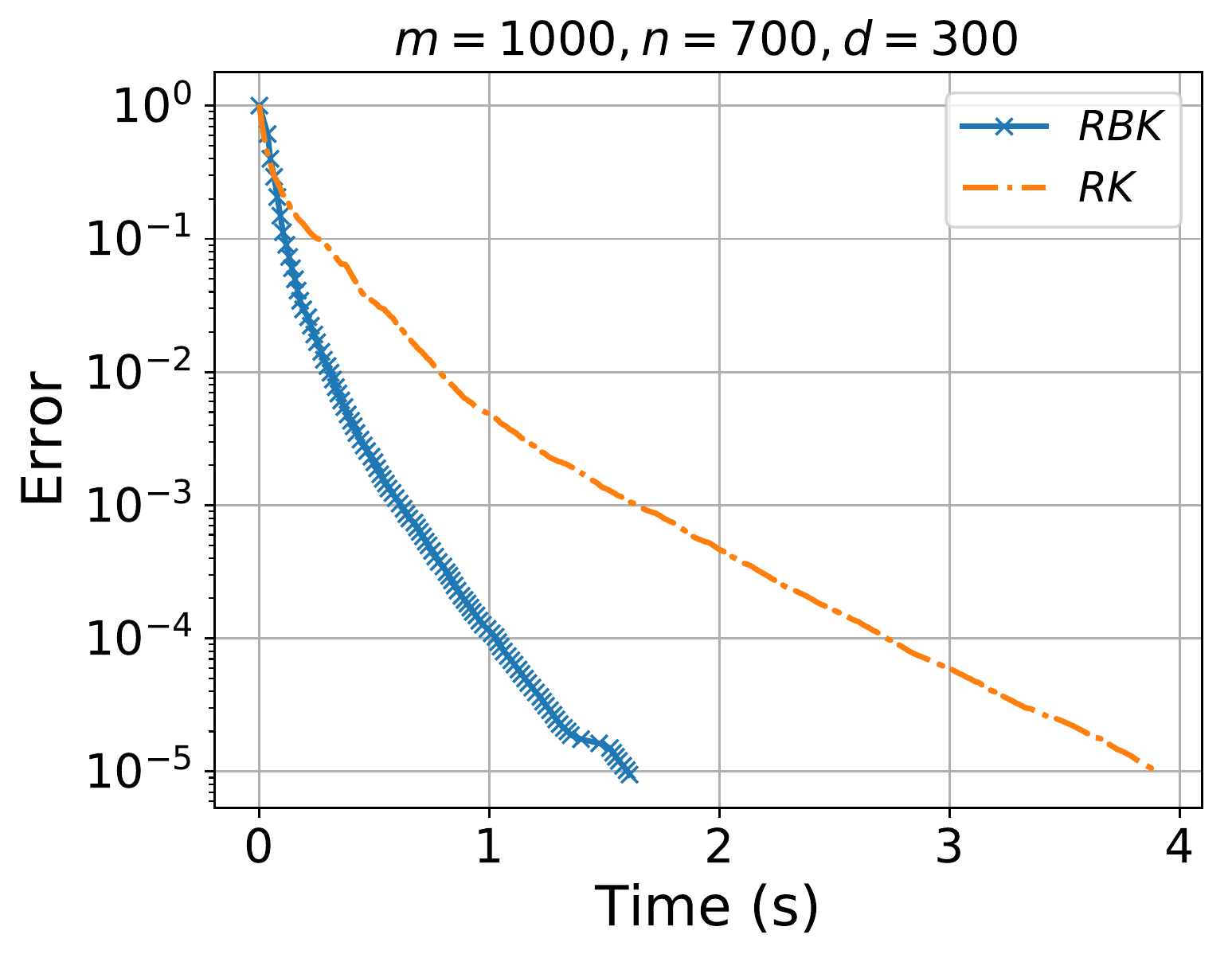}
  \caption{RK vs RBK}
\end{subfigure}
\begin{subfigure}{.24\textwidth}
  \centering
  \includegraphics[width=1\linewidth]{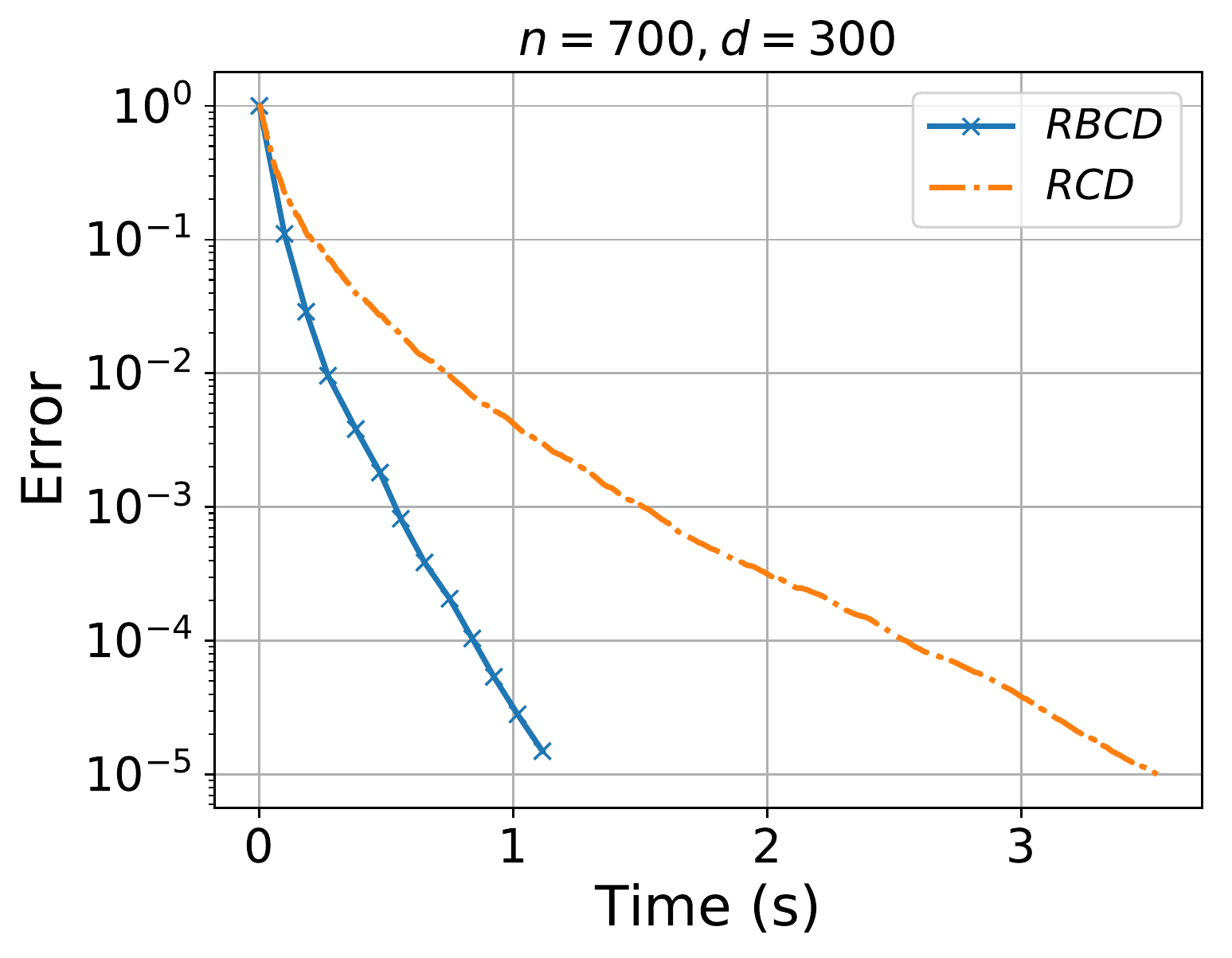}
  \caption{RCD vs RBCD }
\end{subfigure}\\
\caption{\small Comparison of the performance of the exact RBK and RBCD with their non-block variants RK and RCD. For the Kaczmarz methods (first column)  $\bA \in \R^{1000,700}$ is a Gaussian matrix and for the Coordinate descent methods (second column)  $\bA = \bP^\top \bP \in \R^{700 \times 700}$ where $\bP \in \R^{1000 \times 700}$ is Gaussian matrix.  To guarantee consistency $b=\bA z$ where $z$ is also Gaussian vector. The block size that chosen for the block variants is $d=300$.}
\label{BlockSizeRBKRBCD}
\end{figure}

\subsection{Inexactness and Block Size (iRBCD)}
In this experiment, we first construct a positive definite linear system following the previously described procedure for iRBCD. We first generate a Gaussian matrix $\bP \in \R^{10000 \times 7000}$ and then the positive definite matrix $\bA = \bP^\top \bP \in \R^{7000 \times 7000}$ is used to define a consistent liner system. We run iRBCD in this specific linear system and compare its performance with its exact variance for several block sizes $d$ (numbers of column of matrix $\bS$). For evaluating the inexact solution of the linear system in the update rule we run CG for either 2, 5 or 10 iterations. In Figure~\ref{iRBCDfigure}, we plot the evolution of the relative error in terms of both the number of iterations and the wall-clock time.

We observe that for any block size the inexact methods are always faster in terms of wall clock time than their exact variants even if they require (as is expected) equal or larger number of iterations. Moreover it is obvious that the performance of the inexact method becomes much better than the exact variant as the size $d$ increases and as a results the sub-problem that needs to be solved in each step becomes more expensive. It is worth to highlight that for the chosen systems, the exact RBCD behaves better in terms of wall clock time as the size of block increases (this coincides with the findings of the previous experiment).

\begin{figure}[t!]
\centering
\begin{subfigure}{.24\textwidth}
  \centering
  \includegraphics[width=1\linewidth]{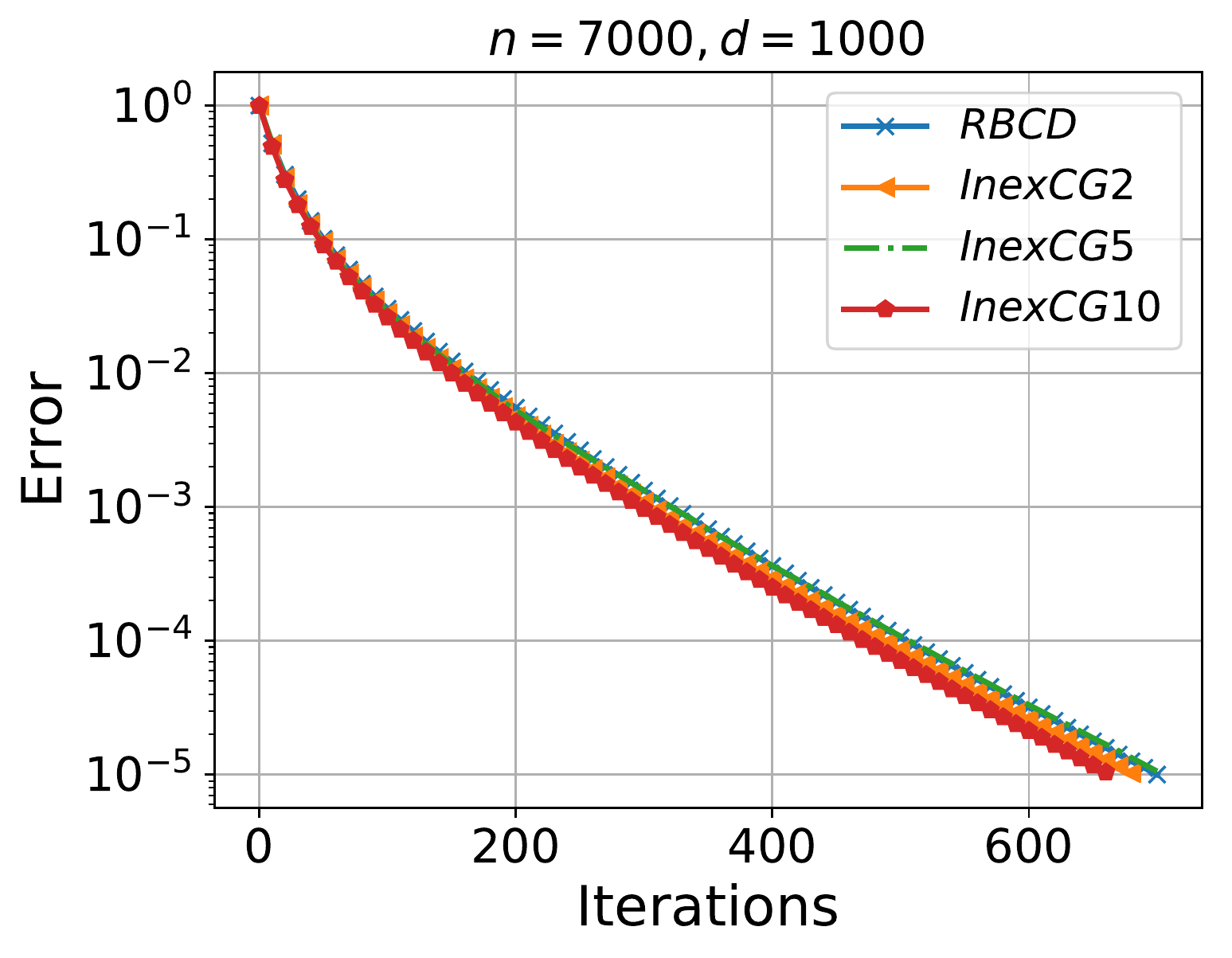}
\end{subfigure}%
\begin{subfigure}{.24\textwidth}
  \centering
  \includegraphics[width=1\linewidth]{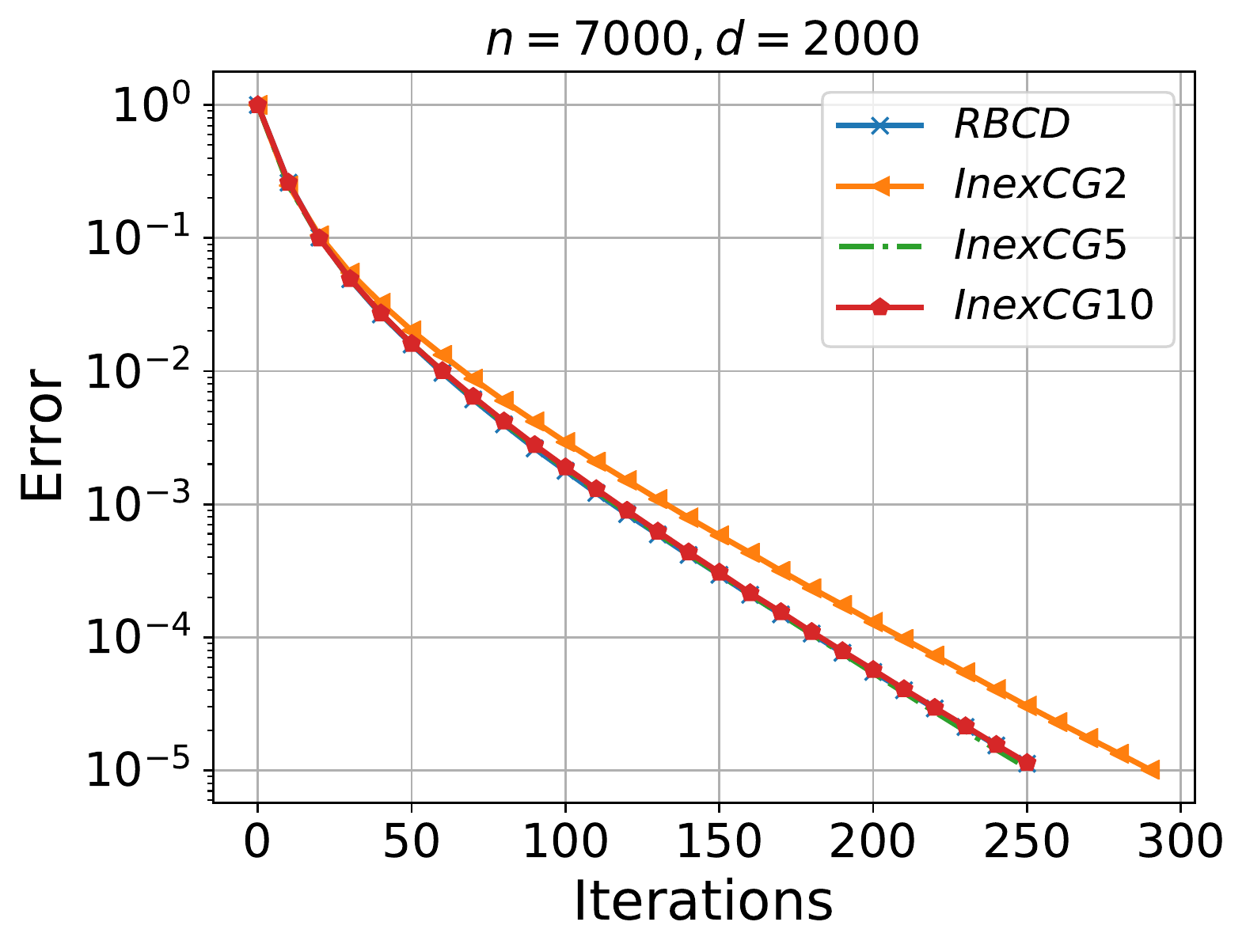}
\end{subfigure}%
\begin{subfigure}{.24\textwidth}
  \centering
  \includegraphics[width=1\linewidth]{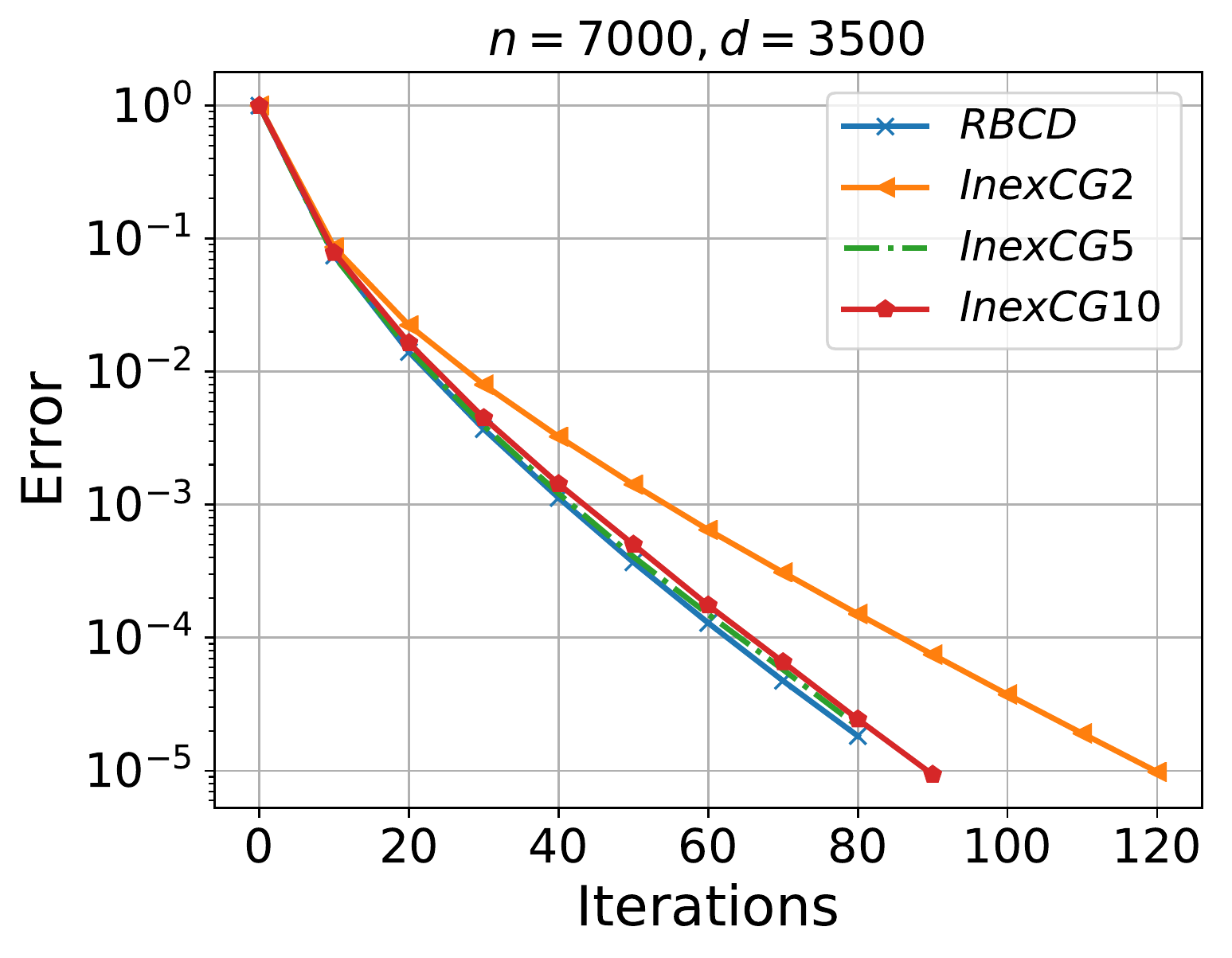}
\end{subfigure}%
\begin{subfigure}{.24\textwidth}
  \centering
  \includegraphics[width=1\linewidth]{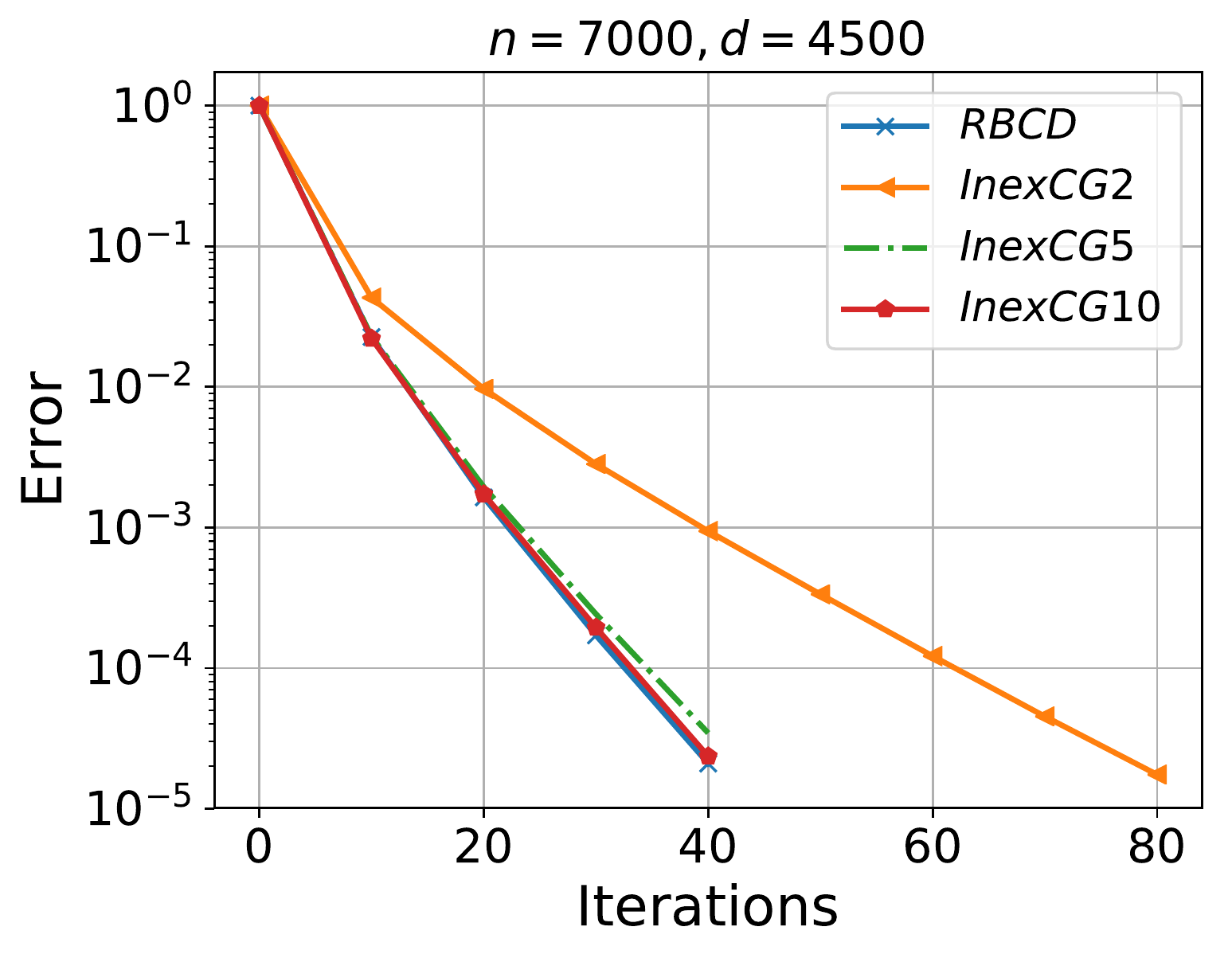}
\end{subfigure}\\
\begin{subfigure}{.24\textwidth}
  \centering
  \includegraphics[width=1\linewidth]{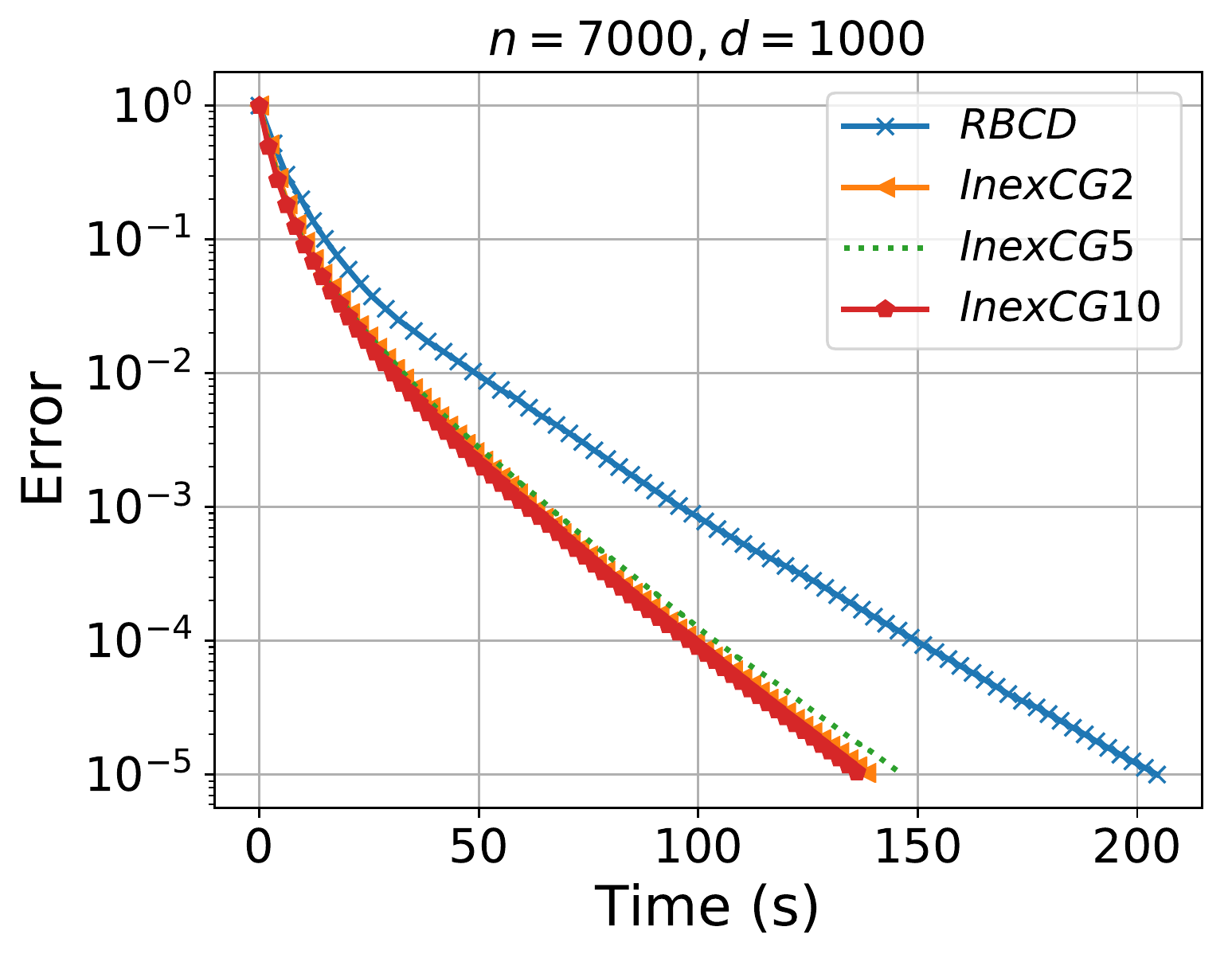}
  \caption{d=1000}
\end{subfigure}
\begin{subfigure}{.24\textwidth}
  \centering
  \includegraphics[width=1\linewidth]{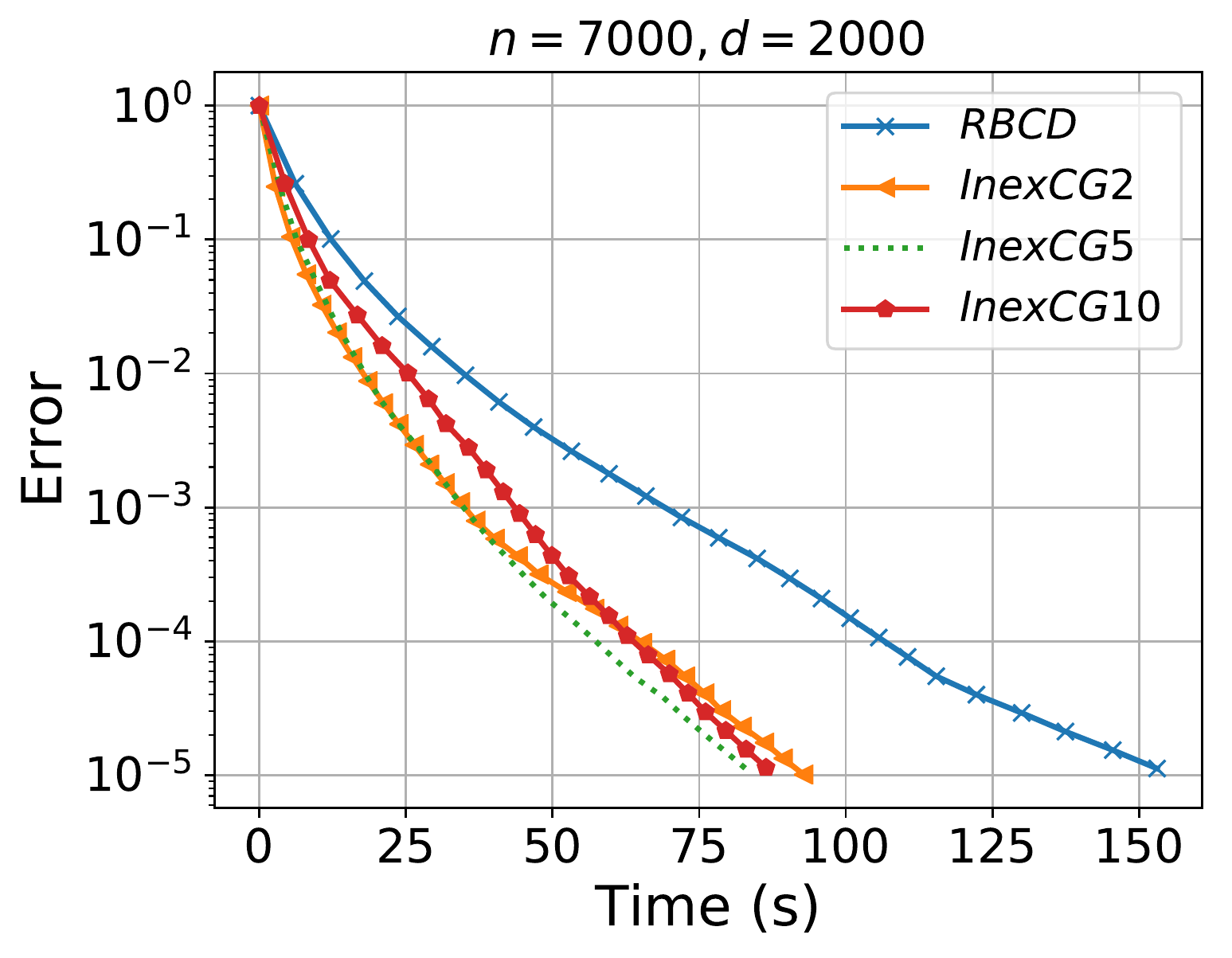}
  \caption{d=2000}
\end{subfigure}
\begin{subfigure}{.24\textwidth}
  \centering
  \includegraphics[width=1\linewidth]{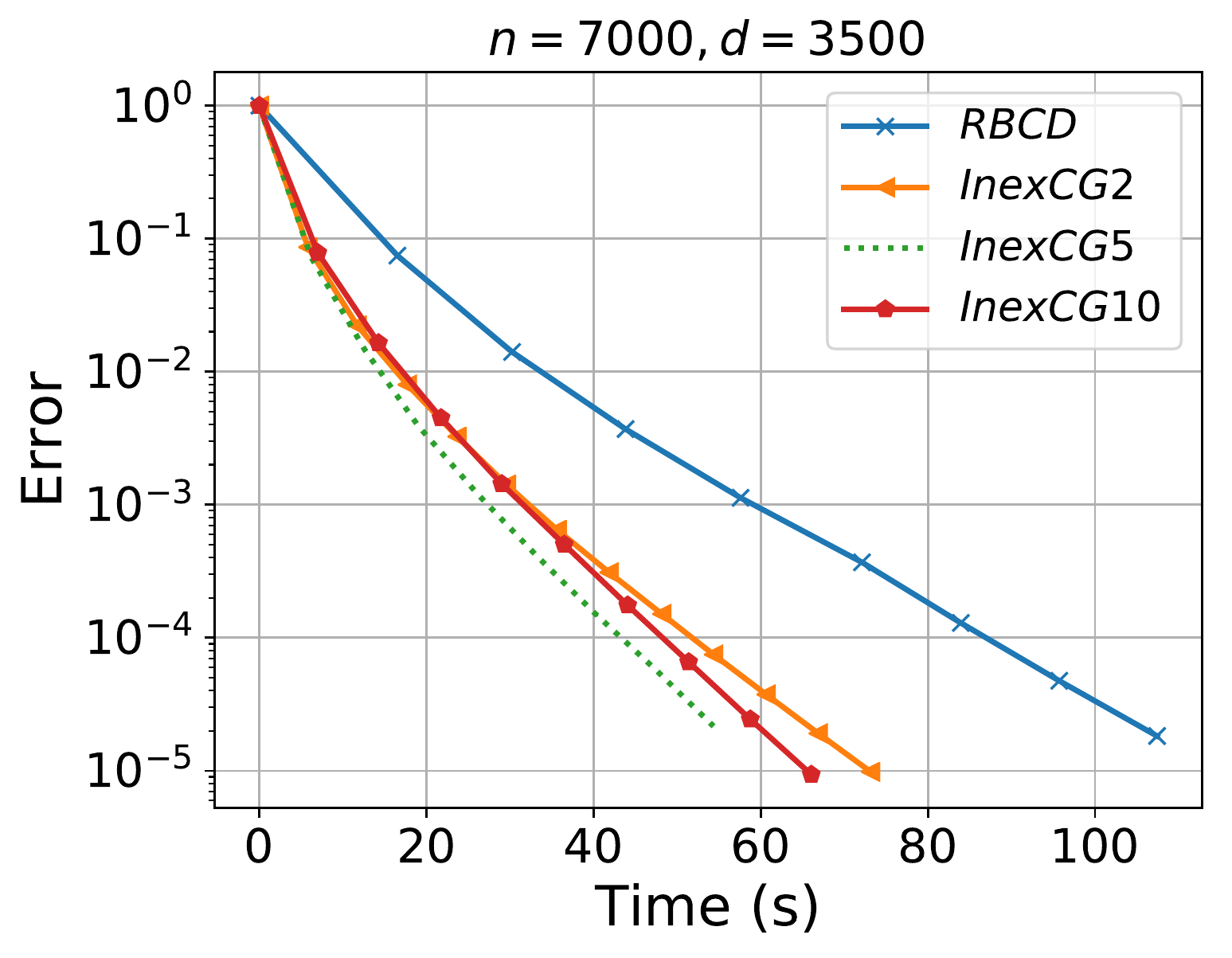}
  \caption{d=3500}
\end{subfigure}
\begin{subfigure}{.24\textwidth}
  \centering
  \includegraphics[width=1\linewidth]{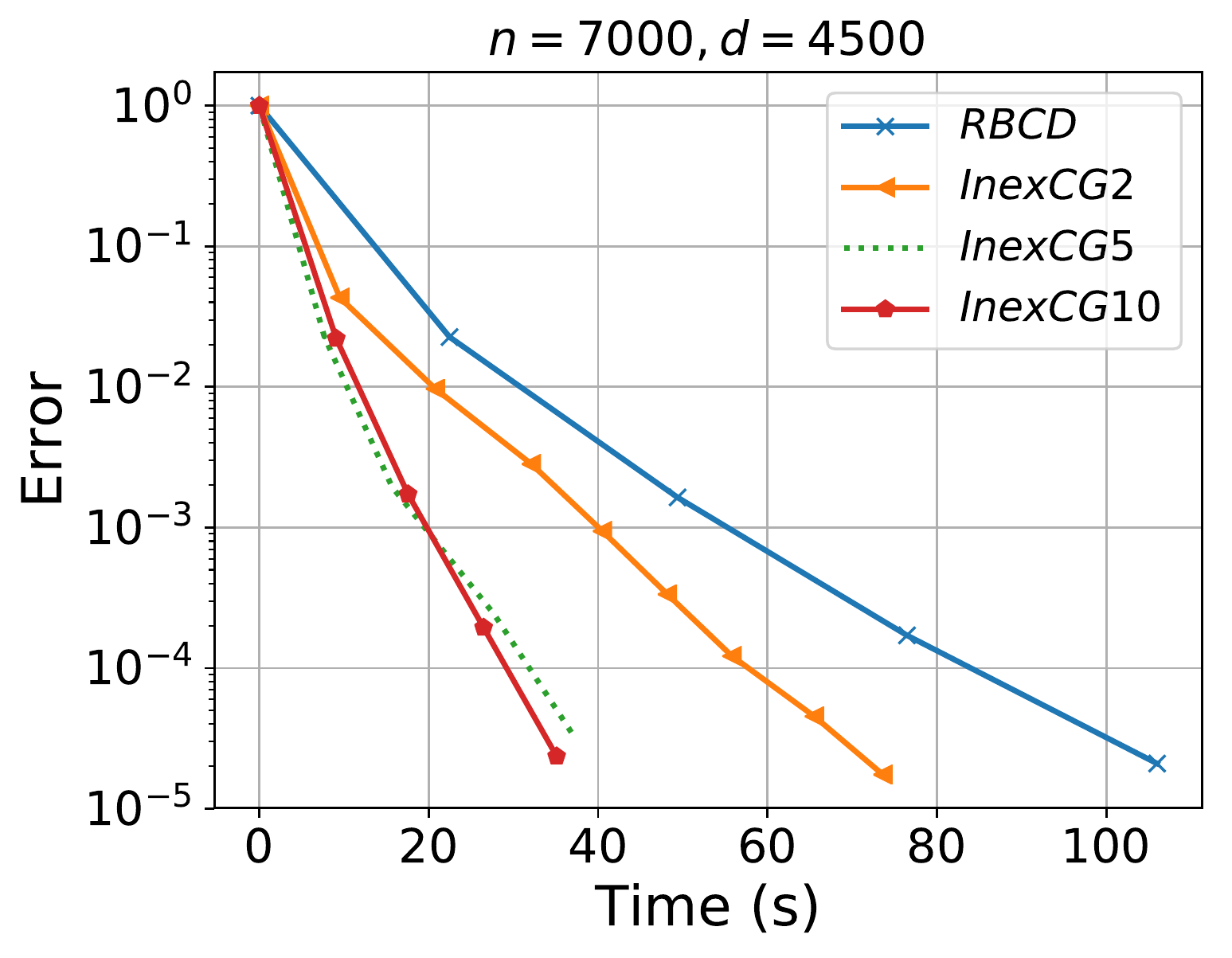}
  \caption{d=4500}
\end{subfigure}\\
\caption{\small Performance of iRBCD (InexactCG) and exact RBCD for solving a consistent linear systems with $\bA = \bP^\top \bP \in \R^{7000 \times 7000}$, where $\bP \in \R^{10000 \times 7000}$ is a Gaussian matrix. The right hand side for the system is chosen to be $b=\bA z$ where $z$ is also a Gaussian vector. Several block sizes are used: $d=1000, 2000, 3500, 4500.$ The graphs in the first (second) row plot the iterations (time) against relative error $\|x_k-x_*\|^2_\bA / \|x_*\|^2_\bA $.}
\label{iRBCDfigure}
\end{figure}

\subsection{Evaluation of iRBK}
In the last experiment we evaluate the performance of iRBK in both synthetic and real datasets. For computing the inexact solution of the linear system in the update rule we run CG for pre-specified number of iterations that can vary depending the datasets.  In particular, we compare iRBK and RBK on synthetic linear systems generated with the Julia Gaussian matrix functions ``randn(m,n)" and ``sprandn(m,n,r)" (input $r$ of sprandn function indicates the density of the matrix). For the real datasets, we test the performance of iRBK and RBK using real matrices from the library of support vector machine problems LIBSVM \cite{chang2011libsvm}. Each dataset of the LIBSVM consists of a matrix $\bA \in \R^{m \times n}$ ($m$ features and $n$ characteristics) and a vector of labels $b \in \R^m$. In our experiments we choose to use only the matrices of the datasets and ignore the label vectors \footnote{Note that the real matrices of the Splice and Madelon datasets are full rank matrices.}. As before, to ensure consistency of the linear system, we choose a Gaussian vector $z\in\R^n$ and the right hand side of the linear system is set to $b = \mA z$ (for both the synthetic and the real matrices). By observing Figure~\ref{RKinsideRBKreal} it is clear that for all problems under study the performance of iRBK in terms of wall clock time is much better than its exact variant RBK.

\begin{figure}[t!]
\centering
\begin{subfigure}{.24\textwidth}
  \centering
  \includegraphics[width=1\linewidth]{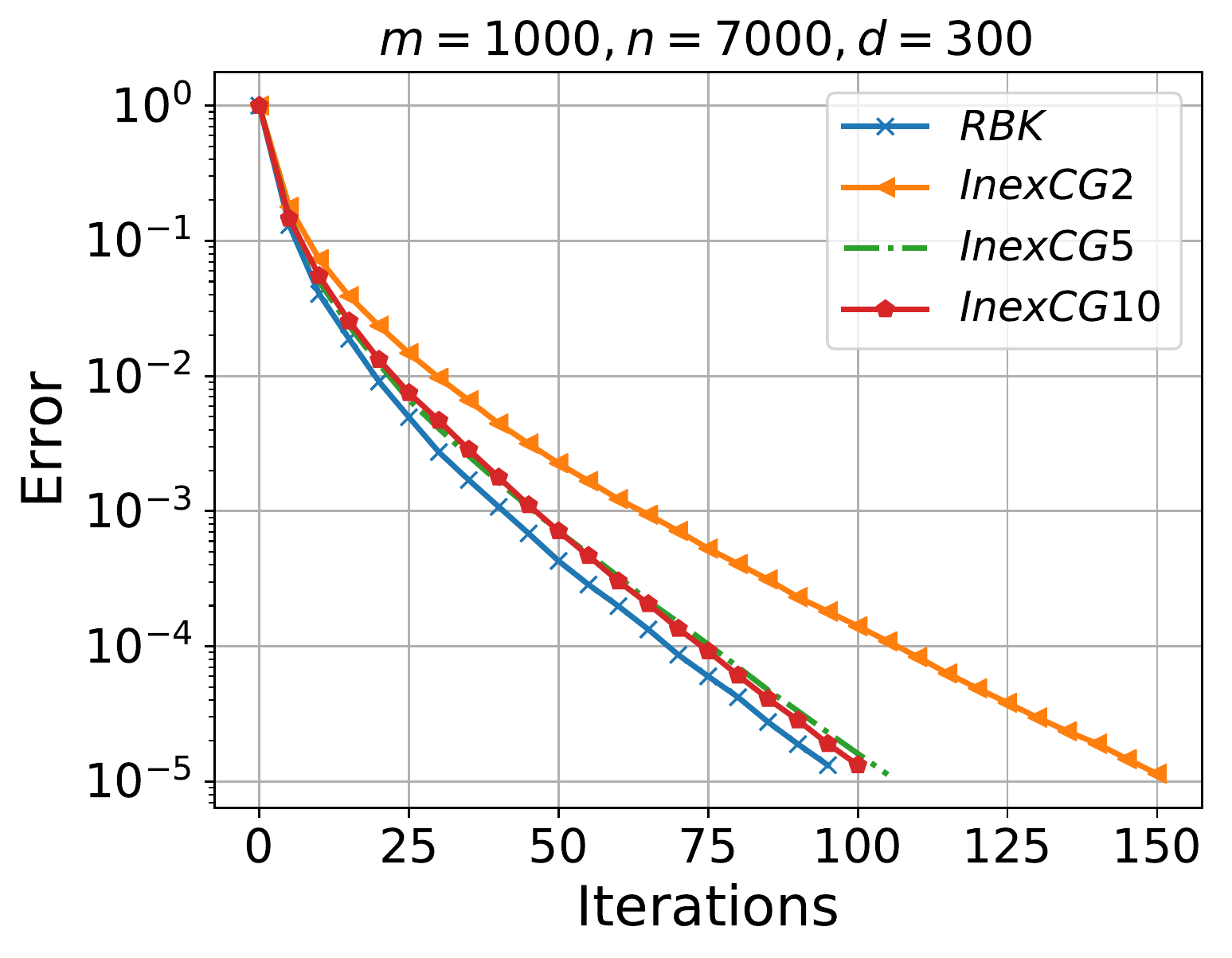}
\end{subfigure}%
\begin{subfigure}{.24\textwidth}
  \centering
  \includegraphics[width=1\linewidth]{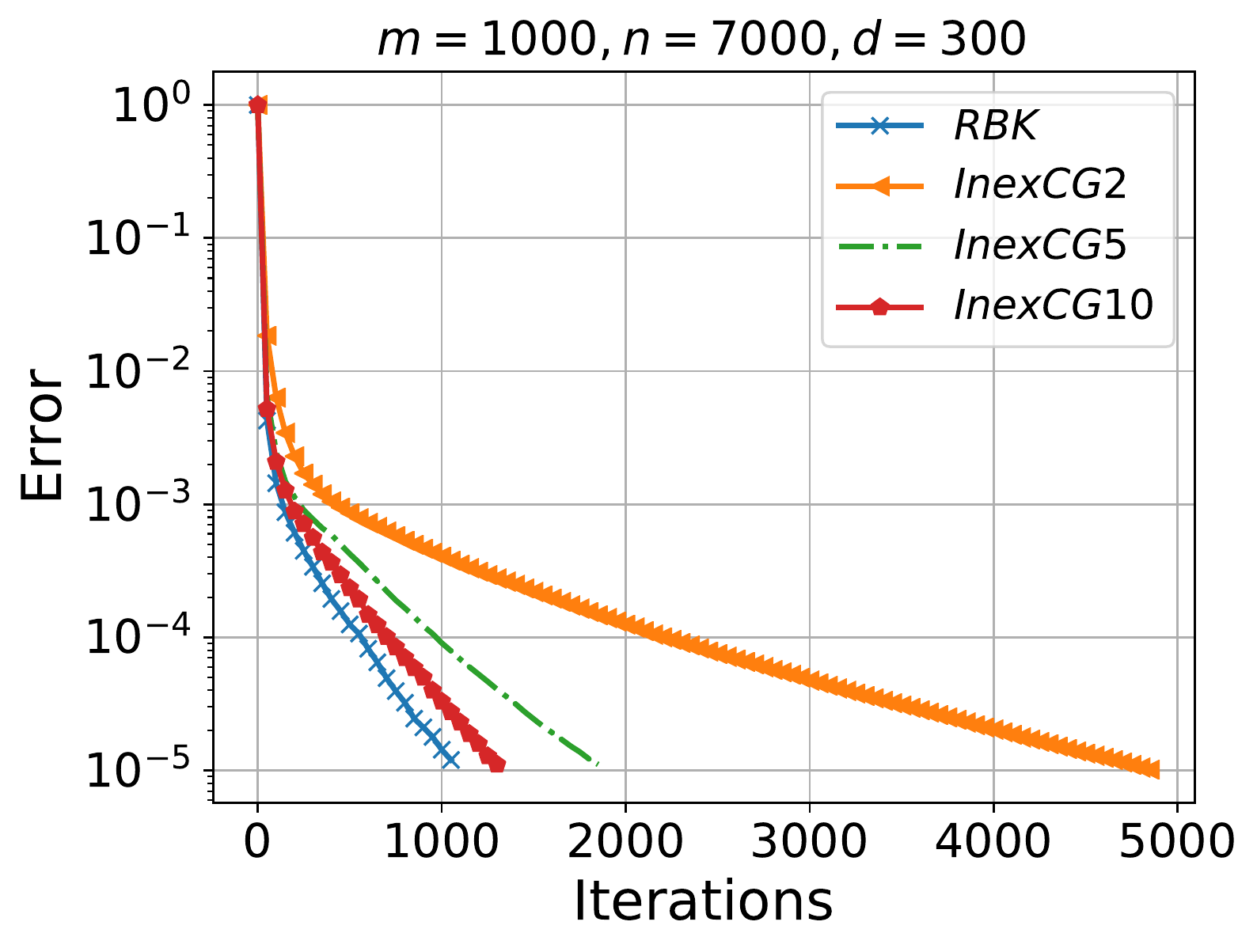}
\end{subfigure}
\begin{subfigure}{.24\textwidth}
  \centering
  \includegraphics[width=\linewidth]{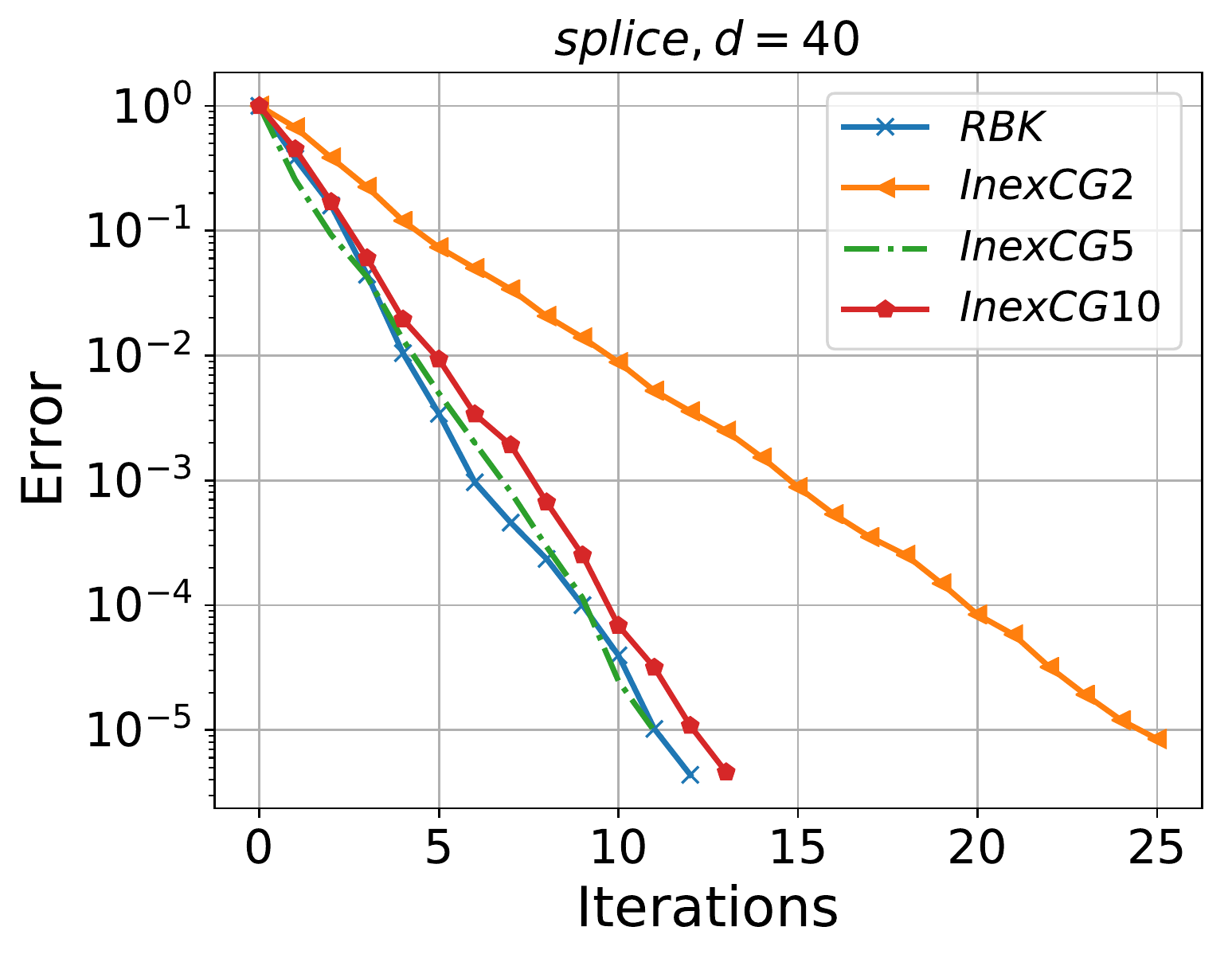}
\end{subfigure}
\begin{subfigure}{.24\textwidth}
  \centering
  \includegraphics[width=1\linewidth]{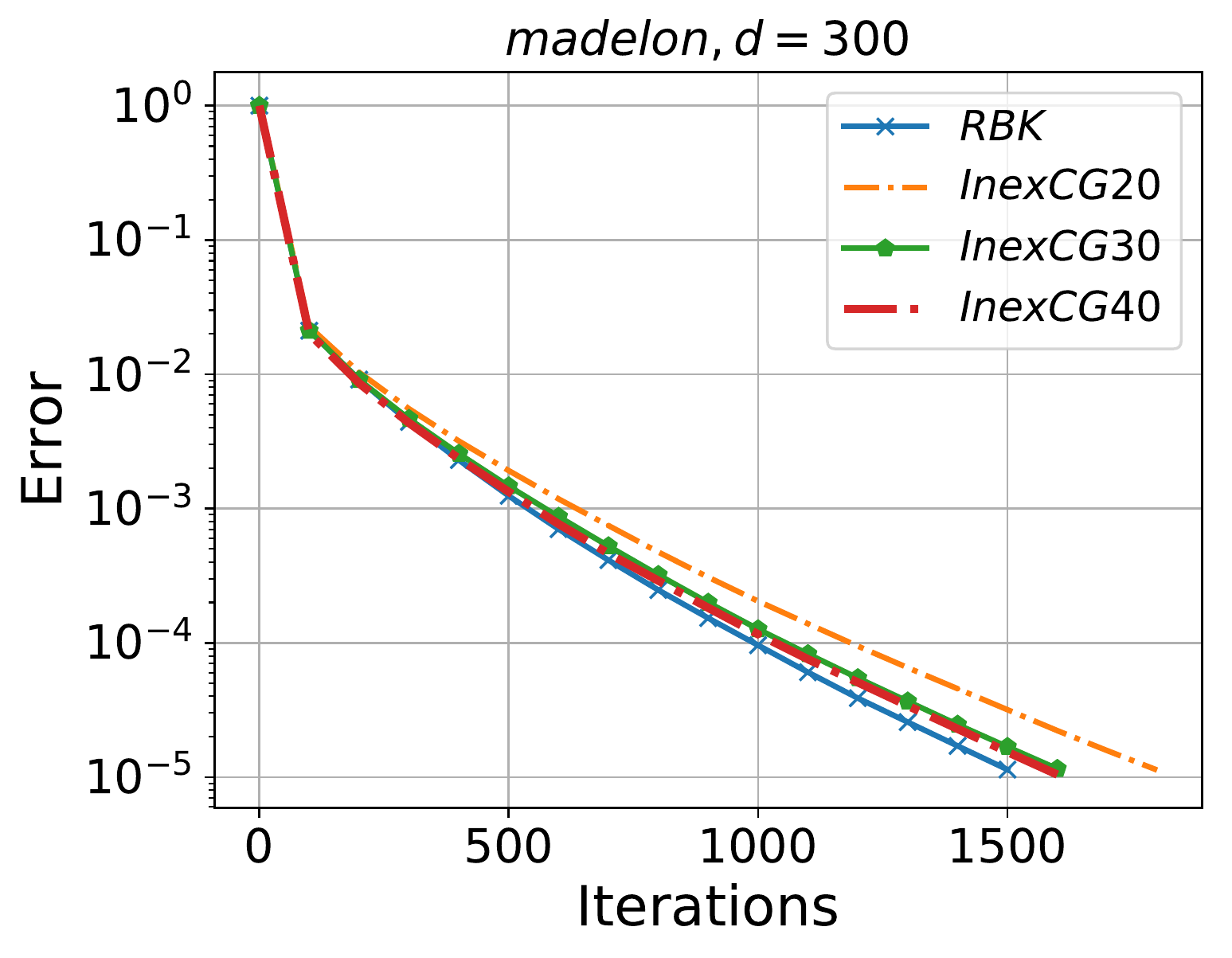}
\end{subfigure}\\
\begin{subfigure}{.24\textwidth}
  \centering
  \includegraphics[width=1\linewidth]{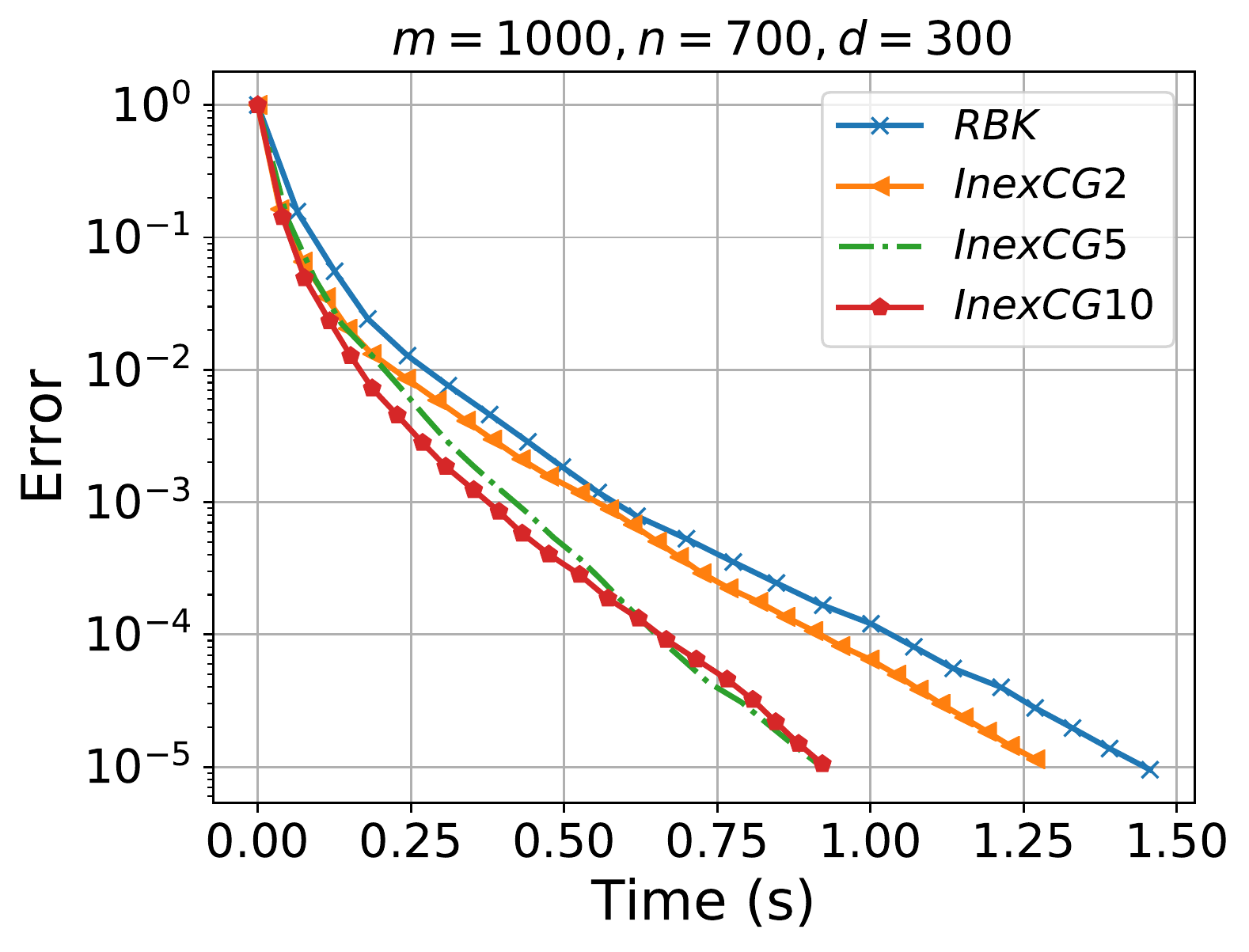}
  \caption{randn(m,n)}
\end{subfigure}%
\begin{subfigure}{.24\textwidth}
  \centering
  \includegraphics[width=1\linewidth]{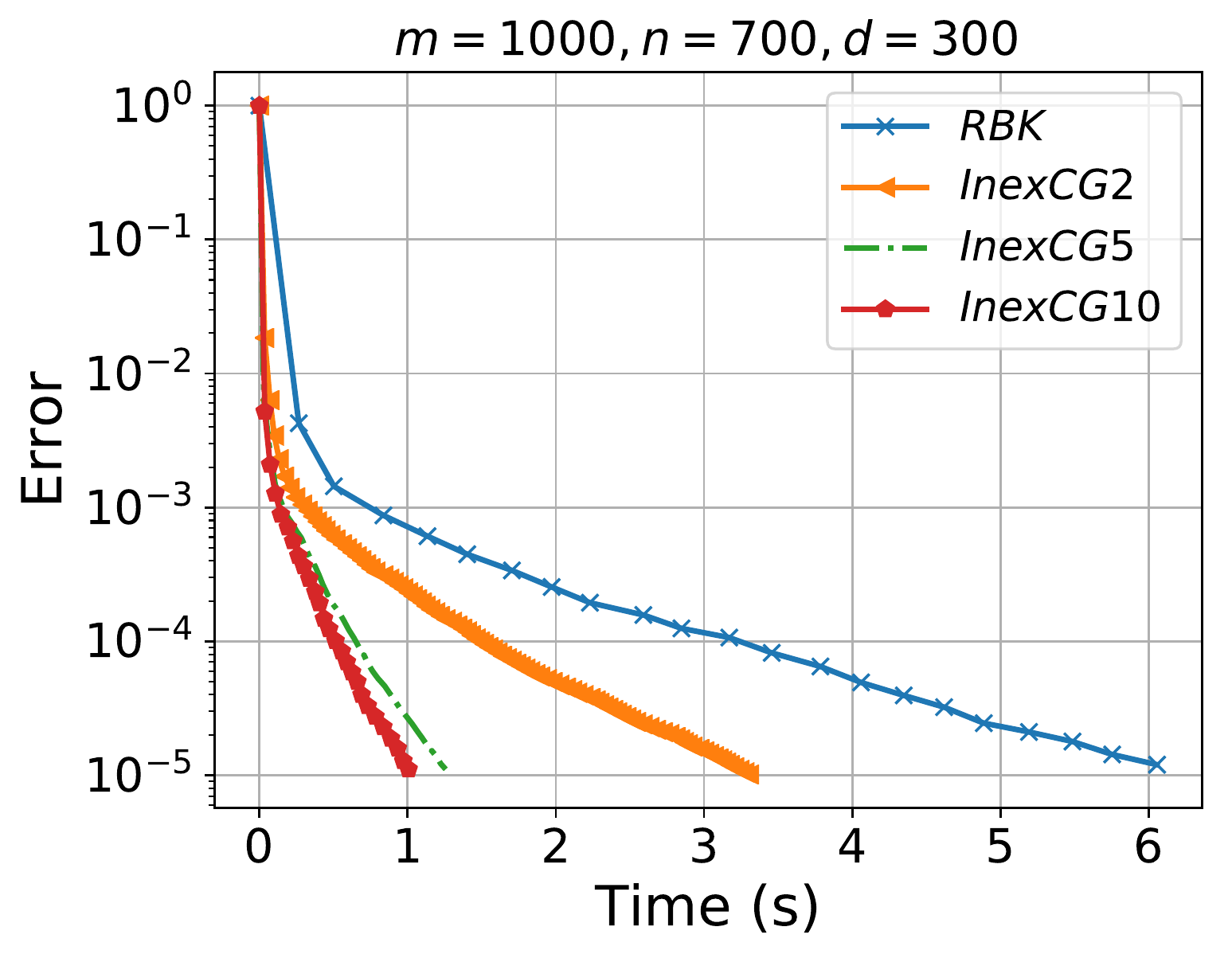}
  \caption{sprandn(m,n,0.01)}
\end{subfigure}
\begin{subfigure}{.24\textwidth}
  \centering
  \includegraphics[width=\linewidth]{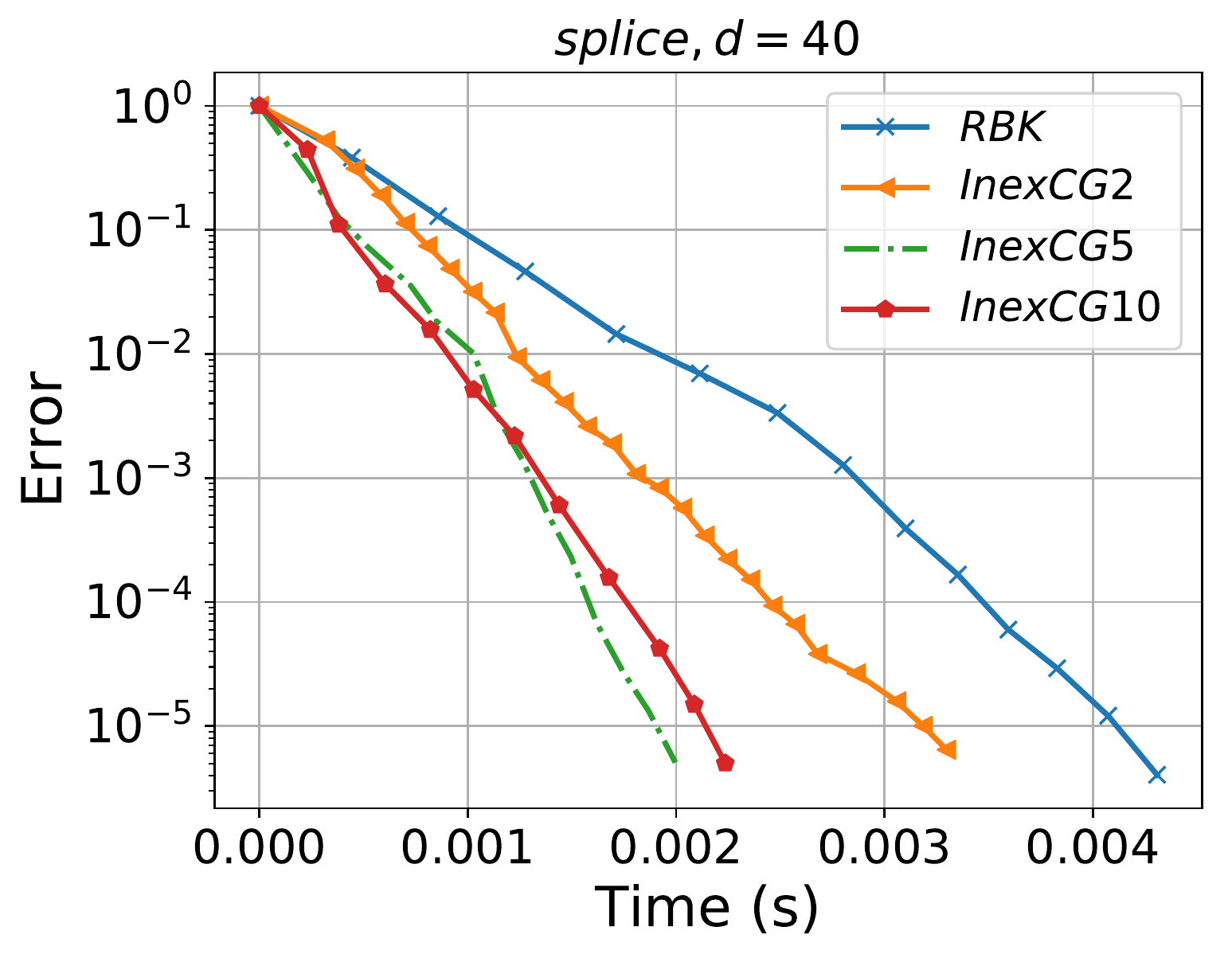}
  \caption{splice}
\end{subfigure}
\begin{subfigure}{.24\textwidth}
  \centering
  \includegraphics[width=1\linewidth]{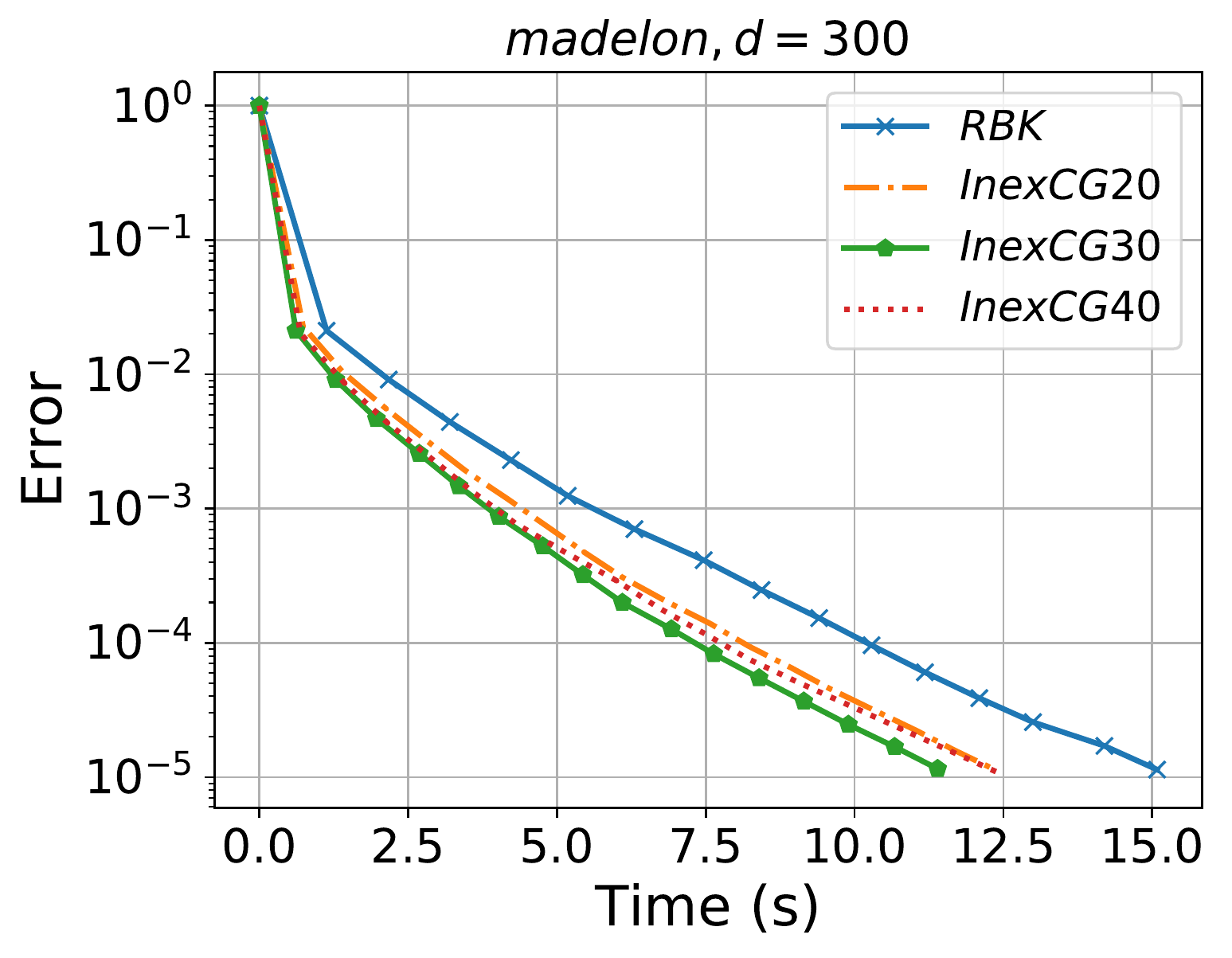}
  \caption{madelon}
\end{subfigure}\\
\caption{\small The performance of iRBK (InexactCG) and RBK on synthetic and real datasets. Synthetic matrices: (a) randn(m,n) with (m,n)=(1000,700), (b) sprandn(m,n,0.01) with (m,n)=(1000,700). Real Matrices from LIBSVM \cite{chang2011libsvm} :  (c) splice : (m,n)=(1000,60), (d) madelon: (m,n)=(2000,500). The graphs in the first (second) row plot the iterations (time) against relative error $\|x_k-x_*\|^2 / \|x_*\|^2$. The quantity $d$ in the title of each plot indicates the size of the block size for both iRBK and RBK.}
\label{RKinsideRBKreal}
\end{figure}

\section{Conclusion}
\label{conlcusion}

In this work we propose and analyze inexact variants of several stochastic algorithms for solving quadratic optimization problems and linear systems. We provide linear convergence rate under several assumptions on the inexactness error. The proposed methods require more iterations than their exact variants to achieve the same accuracy. However, as we show through our numerical evaluations, the inexact algorithms require significantly less time to converge.

With the continuously increasing size of datasets, inexactness should  definitely be a tool that practitioners should use in their implementations even in the case of stochastic methods that have much cheaper-to-compute iteration complexity than their deterministic variants.  Recently, accelerated and parallel stochastic optimization methods \cite{loizou2017momentum, ASDA, tu2017breaking} have been proposed for solving linear systems. We speculate that the addition of inexactness to these update rules will lead to methods faster in practice. We also believe that our approach and complexity results can be extended to the more general case of minimization of convex and non-convex functions in the stochastic setting.  Finally, sketch-and-project algorithms have been used for solving the average consensus problem \cite{LoizouRichtarik, hanzely2017privacy}  popular in distributed optimization literature. Our results could also be useful in this area and lead to the development of novel randomized gossip algorithms that use inexactness in their update rule. 

\section{Acknowledgements}
The first author would like to acknowledge Robert Mansel Gower, Georgios Loizou  and Rachael Tappenden for useful discussions.

\bibliographystyle{plain}
{\footnotesize\bibliography{InexactMethods}}
\appendix

\section{Technical Preliminaries}
\begin{lem}(Lemma 4.2 \cite{ASDA}: Quadratic bounds)
\label{boundLemma}
For all $x\in\R^n$ and $x_*\in \cL$ the following hold:
$\lambda_{\rm min}^+f(x)\leq\frac{1}{2}\|\nabla f(x)\|_{\bB}^2\leq\lambda_{\rm max}f(x)$
and $ f(x)\leq\frac{\lambda_{\rm max}}{2}\|x-x_*\|^2_{\bB}$.
Furthermore, if exactness is satisfied and we let  $x_*= \Pi_{\cL,\mB}(x_0)$ we have
\begin{eqnarray}\label{LowerBoundOfFunctions}
\frac{\lambda_{\rm min}^+}{2}\|x-x_*\|^2_{\bB}\le f(x). 
\end{eqnarray}
\end{lem}

\begin{lem} \cite{ASDA}
\label{stochastic_x}
Let $x_*\in \cL$ and $\{x_k\}_{k\ge 0}$ be the random iterates produced by the exact Basic method (Algorithm~\ref{inexact_basic} with $\epsilon_k=0$) with an arbitrary stepsize $\omega \in \R$. Then:
\begin{eqnarray}\label{x_k_omega}  
\|x_{k+1}-x_*\|_{\mB}^2 = \|(\mI-\omega \mB^{-1}\mZ_k)(x_k-x_*)\|_{\mB}^2=\|x_k-x_*\|_{\mB}^2 - 2 \omega(2-\omega)f_{\mS_k}(x).
\end{eqnarray}
By taking expectation condition on $x_k$ (that is, the expectation is with respect to $\bS_k$) and assuming $\omega \in (0,2)$ we can further obtain:
\begin{eqnarray}\label{exp_x_k_omega}
\E{\|x_{k+1}-x_*\|_{\mB}^2 \;|\; x_k}= \|x_k-x_*\|_{\mB}^2  - 2 \omega(2-\omega)f(x_k) \overset{\eqref{LowerBoundOfFunctions}}{\leq} \left[1 - \omega(2-\omega)\lambda_{\rm min}^+ \right] \|x_k-x_*\|^2_{\bB}.
\end{eqnarray}
\end{lem}

 \begin{rem}\label{variance_ineq}
Let $x$ and $y$ be random vectors and let $\sigma$ positive constant. If we assume $\Exp[\|x\|_{\mB}^2\;|\;y] \leq \sigma^2$ then by using the variance inequality (check Table~\ref{inequalities}) we obtain $\Exp[\|x\|_{\mB}\;|\;y] \leq \sigma$.
In our setting if we assume $\Exp[\|\epsilon_k\|_{\mB}^2\;|\;x_k,\bS_k] \leq \sigma_k^2$ where $\epsilon_k$ is the inexactness error and $x_k$ is the current iterate then by the variance inequality it holds that $\Exp[\|\epsilon_k\|_{\mB}\;|\;x_k,\bS_k] \leq \sigma_k$.
\end{rem}

\section{Proofs of Main Results}
In our convergence analysis we use several popular inequalities. Look Table~\ref{inequalities} in Appendix~\ref{some Tables} for the abbreviations and the relevant formulas.

A key step in the proofs of the theorems is to use the tower property of the expectation. We use it in the form
\begin{eqnarray}\label{eq:tower3}
\Exp[\Exp[\Exp[ X  \;|\; x_k,  \mS_k] \;|\; x_k]] = \Exp[X],
\end{eqnarray}
where $X$ is some random variable. In all proofs we perform the three expectations in order, from the innermost to the outermost. Similar to the main part of the paper we use $\rho=1 - \omega(2-\omega)\lambda_{\rm min}^+$.

\subsection{Proof of Theorem~\ref{InexactSGDConstant}}
\label{Appendix1}
\begin{proof}
First we decompose:
\begin{eqnarray}\label{main_equality}
\|x_{k+1}-x_*\|_{\mB}^2 &=& \| (\mI - \omega \mB^{-1}\mZ_k) (x_k-x_*) +  \epsilon_k\|_{\mB}^2\notag\\
&=& \|(\mI - \omega \mB^{-1}\mZ_k) (x_k-x_*)\|_{\mB}^2+\|\epsilon_k\|_{\mB}^2+2 \left\langle (\mI - \omega \mB^{-1}\mZ_k) (x_k-x_*), \epsilon_k \right \rangle.
\end{eqnarray}
Applying the innermost expectation of \eqref{eq:tower3} to \eqref{main_equality}, we get:
\begin{eqnarray}\label{mainequation}
\Exp[\|x_{k+1}-x_*\|_{\mB}^2\;|\;x_k, \bS_k] &=&\underbrace{\Exp[ \|(\mI - \omega \mB^{-1}\mZ_k) (x_k-x_*)\|_{\mB}^2\;|\;x_k,\bS_k]}_{T1} + \underbrace{\Exp[ \|\epsilon_k\|_{\mB}^2\;|\;x_k,\bS_k]}_{T2}\notag\\
&&+ 2\underbrace{\Exp[\left\langle (\mI - \omega \mB^{-1}\mZ_k) (x_k-x_*), \epsilon_k \right \rangle_{\mB}\;|\;x_k,\bS_k]}_{T3}.
\end{eqnarray}
We now analyze the three expression T1,T2,T3 separately. 

Note that an upper bound for the expression T2 can be directly obtained from the assumption 
\begin{equation}
\label{ForT2}
T2= \Exp[\|\epsilon_k\|^2_{\bB}\;|\;x_k,\bS_k]\leq\sigma_k^2.
 \end{equation}
The first expression can be written as: 
\begin{eqnarray}\label{ForT1}
T1= \Exp[ \|(\mI - \omega \mB^{-1}\mZ_k) (x_k-x_*)\|_{\mB}^2\;|\;x_k,\bS_k]& = &\|(\mI - \omega \mB^{-1}\mZ_k) (x_k-x_*)\|_{\mB}^2\notag\\
&\overset{\eqref{x_k_omega}}=& \|x_k-x_*\|_{\mB}^2 - 2\omega (2-\omega)f_{\bS_k}(x_k).
\end{eqnarray}
For expression T3: 
\begin{eqnarray}\label{ForT3}
\Exp[\left\langle (\mI - \omega \mB^{-1}\mZ_k) (x_k-x_*), \epsilon_k \right \rangle_{\mB}\;|\;x_k,\bS_k] &=&\left\langle (\mI - \omega \mB^{-1}\mZ_k) (x_k-x_*), \Exp[\epsilon_k\;|\;x_k,\bS_k] \right \rangle_{\mB}\notag\\
& \overset{{\rm \;C.S.}}{\leq}&  \|(\mI - \omega \mB^{-1}\mZ_k) (x_k-x_*)\|_{\bB} \|\Exp[\epsilon_k\;|\;x_k,\bS_k]\|_{\bB}\notag\\
&\overset{{\rm Cond. Jensen}}{\leq}&  \|(\mI - \omega \mB^{-1}\mZ_k) (x_k-x_*)\|_{\bB} \Exp[\|\epsilon_k \|_{\bB}\;|\;x_k,\bS_k]\notag\\
&\overset{{\rm \;Remark\;}\ref{variance_ineq}\;{\rm and\;} \eqref{Assumption1serial}}\leq& \|(\mI - \omega \mB^{-1}\mZ_k) (x_k-x_*)\|_{\bB} \sigma_k.
\end{eqnarray}
By substituting the bounds \eqref{ForT2}, \eqref{ForT1}, and \eqref{ForT3} into \eqref{mainequation} we obtain:
\begin{eqnarray}\label{afterdba}
\Exp[\|x_{k+1}-x_*\|_{\mB}^2\;|\;x_k, \bS_k] &\leq&  \|x_k-x_*\|_{\mB}^2 - 2\omega (2-\omega)f_{\bS_k}(x_k)+ \sigma_k^2 \notag\\
&&+2 \|(\mI - \omega \mB^{-1}\mZ_k) (x_k-x_*)\|_{\bB} \sigma_k.
\end{eqnarray}
We now take the middle expectation~(see \eqref{eq:tower3}) and apply it to inequality \eqref{afterdba}:
\begin{eqnarray}\label{4}
\Exp[\Exp[\|x_{k+1}-x_*\|_{\mB}^2\;|\;x_k, \bS_k]\;|\;x_k] &\leq&  \|x_k-x_*\|_{\mB}^2 - 2\omega (2-\omega)f(x_k)+\sigma_k^2 \notag\\
&&+2 \Exp[\|(\mI - \omega \mB^{-1}\mZ_k) (x_k-x_*)\|_{\bB}\;|\;x_k]\sigma_k.
\end{eqnarray}
Now let us find a bound on the quantity $\Exp\left[\|(\mI - \omega \mB^{-1}\mZ_k) (x_k-x_*)\|_{\bB}\;|\;x_k\right]$.
Note that from \eqref{exp_x_k_omega} and \eqref{x_k_omega} we have that $\Exp\left[\|(\mI - \omega \mB^{-1}\mZ_k) (x_k-x_*)\|^2_{\bB}\;|\;x_k\right] \leq \rho \|x_k-x_*\|^2_{\bB}$. By using Remark~\ref{variance_ineq} in the last inequality we obtain:
\begin{eqnarray}\label{5} 
\Exp\left[\|(\mI - \omega \mB^{-1}\mZ_k) (x_k-x_*)\|_{\bB}\;|\;x_k\right] & = &\sqrt{\rho}\|x_k-x_*\|_{\mB}.
\end{eqnarray} 

By substituting \eqref{5} in \eqref{4}:
\begin{eqnarray}
\label{eq19}
\Exp[\Exp[\|x_{k+1}-x_*\|_{\mB}^2\;|\;x_k, \bS_k]\;|\;x_k] &\leq &  \|x_k-x_*\|_{\mB}^2 - 2\omega (2-\omega)f(x_k)+ \sigma_k^2 \notag\\
&&+2\sigma_k \sqrt{\rho}\|x_k-x_*\|_{\mB} \notag\\
&\overset{\eqref{exp_x_k_omega}}\leq& \rho \|x_k-x_*\|_{\mB}^2 + \sigma_k^2 + 2\sigma_k \sqrt{\rho}\|x_k-x_*\|_{\mB}
\end{eqnarray}
We take the final expectation (outermost expectation in the tower rule \eqref{eq:tower3}) on the above expression to find:
\begin{eqnarray}
\label{asjcakd}
\Exp[\|x_{k+1}-x_*\|_{\mB}^2]&=&\Exp[\Exp[\Exp[\|x_{k+1}-x_*\|_{\mB}^2\;|\;x_k, \bS_k]\;|\;x_k]]\notag\\
& \leq& \rho \Exp[\|x_k-x_*\|_{\mB}^2] + \sigma_k^2 + 2 \sigma_k \sqrt{\rho} \,\ \Exp[\|x_k-x_*\|_{\mB}] \notag\\
&\overset{V.I}\leq& \rho \Exp[\|x_k-x_*\|_{\mB}^2] + \sigma_k^2 + 2 \sigma_k \sqrt{\rho} \sqrt{\Exp[\|x_k-x_*\|_{\mB}^2]} 
\end{eqnarray}
Using $r_k = \E{\|x_{k}-x_*\|_{\mB}^2}$ equation \eqref{asjcakd} takes the form: 
\begin{eqnarray*}
r_{k+1}\leq & \rho r_k  +  \sigma_k^2 + 2\sigma_k \sqrt{\rho} \sqrt{r_k} = \left( \sqrt{\rho r_k}+\sigma_k \right)^2
 \end{eqnarray*}
 If we further substitute $p_k=\sqrt{r_k}$ and $\ell=\sqrt{\rho}$ the recurrence simplifies to:
 \begin{eqnarray*}
p_{k+1}\leq & \ell p_k +\sigma_k 
 \end{eqnarray*}
By unrolling the final inequality:
\begin{eqnarray*}
p_k \leq \ell^k r_0 + (\ell^0\sigma_{k-1}+\ell\sigma_{k-2}+\cdots + \ell^{k-1}\sigma_0) = \ell^k p_0 +  \sum_{i=0}^{k-1} \ell^{k-1-i}\sigma_i.
\end{eqnarray*}
Hence,  $$\sqrt{\Exp[\|x_{k}-x_*\|_{\mB}^2]} \leq  \rho^{k/2} \|x_{0}-x_*\|_{\mB} +  \sum_{i=0}^{k-1} \rho^{\frac{k-1-i}{2}}\sigma_i.$$
The result is obtained by using V.I in the last expression.
\end{proof}

\subsection{Proof of Corollary~\ref{FirstCorollary}}
\label{Appendix2}
By denoting $r_k=\Exp[\|x_{k}-x_*\|_{\mB}]$ in \eqref{Theorem1} we obtain:
\[r_k \leq \rho^{k/2} r_0 + \rho^{1/2} \sigma \sum_{i=0}^{k-1} \rho^{k-1-i} =\rho^{k/2} r_0 + \rho^{1/2} \sigma \sum_{i=0}^{k-1} \rho^{i}=\rho^{k/2} r_0 + \rho^{1/2} \sigma  \frac{1-\rho^k}{1-\rho}.\]
Since $1-\rho^k\le 1$ the result is obtained. 

\subsection{Proof of Theorem~\ref{ISGDwithq}}
\label{Appendix3}
In order to prove Theorem~\ref{ISGDwithq} we need to follow similar steps to the proof of Theorem~\ref{InexactSGDConstant}. The main  differences of the two proofs appear at the points that we need to upper bound the norm of the inexactness error ($\|\epsilon_k\|^2$). In particular instead of using the general sequence $\sigma_k^2 \in \R$ we utilize the bound $q^2\|x_k-x_*\|^2_{\bB}$ from Assumption~\ref{Assumption3}. Thus, it is sufficient to focus at the parts of the proof that these bound is used.

Similar to the proof of Theorem~\ref{InexactSGDConstant} we first decompose to obtain the equation \eqref{mainequation}. There, the expression T1 can be upper bounded from \eqref{ForT1} but now using the Assumption~\ref{Assumption3} the expression T2  and T3 can be upper bounded as follows:
\begin{equation}
\label{For2T2}
T2= \Exp[\|\epsilon_k\|^2_{\bB}\;|\;x_k,\bS_k]\leq q^2\|x_k-x_*\|^2_{\bB}.
 \end{equation}
\begin{eqnarray}
\label{For2T3}
T3=\Exp[\left\langle (\mI - \omega \mB^{-1}\mZ_k) (x_k-x_*), \epsilon_k \right \rangle_{\mB}\;|\;x_k,\bS_k] \overset{{\rm \;Remark\;}\ref{variance_ineq}\;{\rm and\;} \eqref{ForT3}}\leq \|(\mI - \omega \mB^{-1}\mZ_k) (x_k-x_*)\|_{\bB} q\|x_k-x_*\|
\end{eqnarray}

As a result by substituting the bounds \eqref{ForT1}, \eqref{For2T2}, and \eqref{For2T3} into \eqref{mainequation} we obtain:
\begin{eqnarray}
\Exp[\|x_{k+1}-x_*\|_{\mB}^2\;|\;x_k, \bS_k] &\overset{\eqref{mainequation}}\leq&  \|x_k-x_*\|_{\mB}^2 - 2\omega (2-\omega)f_{\bS_k}(x_k)+ q^2\|x_k-x_*\|^2_{\bB} \notag\\
&&+2 \|(\mI - \omega \mB^{-1}\mZ_k) (x_k-x_*)\|_{\bB} q\|x_k-x_*\|_{\bB}.
\end{eqnarray}

By following the same steps to the proof of Theorem~\ref{InexactSGDConstant} the equation \eqref{eq19} takes the form:
\begin{eqnarray}
\Exp[\Exp[\|x_{k+1}-x_*\|_{\mB}^2\;|\;x_k, \bS_k]\;|\;x_k] 
&\leq& \rho \|x_k-x_*\|_{\mB}^2 + q^2\|x_k-x_*\|^2_{\bB} + 2q\|x_k-x_*\|_{\bB} \sqrt{\rho}\|x_k-x_*\|_{\mB} \notag\\
&=& \left(\rho+2q \sqrt{\rho}+q^2\right)\|x_k-x_*\|_{\mB}^2.\notag\\
& = &\left(\sqrt{\rho}+q\right)^2 \|x_k-x_*\|_{\mB}^2
\end{eqnarray}
We take the final expectation (outermost expectation in the tower rule \eqref{eq:tower3}) on the above expression to find:
\begin{eqnarray}
\Exp[\|x_{k+1}-x_*\|_{\mB}^2]&=&\Exp[\Exp[\Exp[\|x_{k+1}-x_*\|_{\mB}^2\;|\;x_k, \bS_k]\;|\;x_k]]\notag\\
& \leq& \left(\sqrt{\rho}+q\right)^2 \Exp[\|x_k-x_*\|_{\mB}^2].
\end{eqnarray}
The final result follows by unrolling the recurrence. 

\subsection{Proof of Theorem \ref{InexactSGDrandom}}
\label{Appendix4}
\begin{proof}
Similar to the previous two proofs by decomposing the update rule and using the innermost expectation of \eqref{eq:tower3} we obtain equation \eqref{mainequation}.  An upper bound of expression T1 is again given by inequality \eqref{ForT1}. For the expression T2 depending the assumption that we have on the norm of the inexactness error different upper bounds can be used. In particular, 
\begin{enumerate}
\item[(i)] If Assumption \ref{Assumption2} holds then:  
 $ T2=\Exp[\|\epsilon_k\|^2_{\bB}\;|\;x_k,\bS_k]\leq\sigma_k^2.$
\item[(ii)] If Assumption \ref{Assumption3} holds then: $T2=\Exp[\|\epsilon_k\|^2_{\bB}\;|\;x_k,\bS_k]\leq\sigma_k^2 = q^2\|x_k-x_*\|^2_{\bB}.$
\item[(iii)] If Assumption \ref{Assumption4} holds then:  
$ T2=\Exp[\|\epsilon_k\|^2_{\bB}\;|\;x_k,\bS_k]\leq\sigma_k^2 = 2 q^2 f_{S_k}(x_k).$
\end{enumerate}

The main difference from the previous proofs, is that due to the Assumption \ref{Assumption5} and tower property \eqref{eq:tower3} the expression T3 will eventually be equal to zero. More specifically, we have that:
$$\Exp[\Exp[\Exp[ \left\langle (\mI - \omega \mB^{-1}\mZ_k) (x_k-x_*), \epsilon_k \right \rangle_{\mB}  \;|\; x_k,  \mS_k] \;|\; x_k]] = \Exp[\left\langle (\mI - \omega \mB^{-1}\mZ_k) (x_k-x_*), \epsilon_k \right \rangle_{\mB}]=T3=0,$$

Thus, in this case equation \eqref{afterdba} takes the form:
\begin{eqnarray}\label{mainequationansdjkad}
\Exp[\|x_{k+1}-x_*\|_{\mB}^2\;|\;x_k, \bS_k] &\leq & \|x_k-x_*\|_{\mB}^2 - 2\omega (2-\omega)f_{\bS_k}(x_k) + \sigma_k^2.
\end{eqnarray}
Using the above expression depending the assumption that we have we obtain the following results:
\begin{enumerate}
\item[(i)] 
By taking the middle expectation~(see \eqref{eq:tower3}) and apply it to the above inequality:
\begin{eqnarray}
\Exp[\Exp[\|x_{k+1}-x_*\|_{\mB}^2\;|\;x_k, \bS_k]\;|\;x_k] &\leq&  \|x_k-x_*\|_{\mB}^2 - 2\omega (2-\omega)f(x_k)+\Exp[\sigma_k^2 \;|\;x_k] \notag\\
&\overset{\eqref{exp_x_k_omega}}\leq&  \rho \|x_k-x_*\|_{\mB}^2 +\Exp[\sigma_k^2 \;|\;x_k] 
\end{eqnarray}

We take the final expectation (outermost expectation in the tower rule \eqref{eq:tower3}) on the above expression to find:
\begin{eqnarray}
\Exp[\|x_{k+1}-x_*\|_{\mB}^2]&=&\Exp[\Exp[\Exp[\|x_{k+1}-x_*\|_{\mB}^2\;|\;x_k, \bS_k]\;|\;x_k]]\notag\\
& \leq& \rho\Exp[\|x_k-x_*\|_{\mB}^2] + \Exp[\Exp[\sigma_k^2 \;|\;x_k]] \notag\\
& =& \rho \Exp[\|x_k-x_*\|_{\mB}^2] + \Exp[\sigma_k^2] \notag\\
& =& \rho \Exp[\|x_k-x_*\|_{\mB}^2] + {\bar{\sigma}}_k^2 
\end{eqnarray}
Using $r_k = \E{\|x_{k}-x_*\|_{\mB}^2}$ the last inequality takes the form
$r_{k+1}\leq \rho r_k  +   {\bar{\sigma}}_k^2$. By unrolling the last expression:
$r_k \leq \rho^k r_0 + (\rho^0 {\bar{\sigma}}_{k-1}^2+\rho {\bar{\sigma}}^2_{k-2}+\cdots + \rho^{k-1} {\bar{\sigma}}^2_0) =\rho^k r_0 +  \sum_{i=0}^{k-1} \rho^{k-1-i} {\bar{\sigma}}^2_i.$
Hence,  $$\Exp[\|x_{k}-x_*\|_{\mB}^2] \leq  \rho^{k} \|x_{0}-x_*\|^2_{\mB} +  \sum_{i=0}^{k-1} \rho^{k-1-i}{\bar{\sigma}}^2_i.$$
\item[(ii)]
For the case (ii) inequality \eqref{mainequationansdjkad} takes the form:
\begin{eqnarray}
\Exp[\|x_{k+1}-x_*\|_{\mB}^2\;|\;x_k, \bS_k] &\leq & \|x_k-x_*\|_{\mB}^2 - 2\omega (2-\omega)f_{\bS_k}(x_k) + q^2\|x_k-x_*\|^2_{\bB},
\end{eqnarray}
and by taking the middle expectation~(see \eqref{eq:tower3}) we obtain:
\begin{eqnarray}
\Exp[\Exp[\|x_{k+1}-x_*\|_{\mB}^2\;|\;x_k, \bS_k]\;|\;x_k]  &\leq & \|x_k-x_*\|_{\mB}^2 - 2\omega (2-\omega)f(x_k) + q^2\|x_k-x_*\|^2_{\bB}\notag\\
& \overset{\eqref{exp_x_k_omega}}\leq & \rho \|x_k-x_*\|_{\mB}^2 + q^2\|x_k-x_*\|^2_{\bB} \notag\\
& =& (\rho + q^2) \|x_k-x_*\|_{\mB}^2. 
\end{eqnarray}
By taking the final expectation of the tower rule \eqref{eq:tower3} and apply it to the above inequality:
\begin{eqnarray}
\Exp[\|x_{k+1}-x_*\|_{\mB}^2] & \leq& (\rho + q^2) \Exp[\|x_k-x_*\|_{\mB}^2]. 
\end{eqnarray}
and the result is obtain by unrolling the last expression.
\item[(iii)] 
For the case (iii) inequality \eqref{mainequationansdjkad} takes the form:
\begin{eqnarray}
\Exp[\|x_{k+1}-x_*\|_{\mB}^2\;|\;x_k, \bS_k] &\leq & \|x_k-x_*\|_{\mB}^2 - 2(\omega (2-\omega)-q^2 ) f_{\bS_k}(x_k),
\end{eqnarray}
and by taking the middle expectation~(see \eqref{eq:tower3}) we obtain:
\begin{eqnarray}
\Exp[\Exp[\|x_{k+1}-x_*\|_{\mB}^2\;|\;x_k, \bS_k]\;|\;x_k]  &\leq &\|x_k-x_*\|_{\mB}^2 - 2(\omega (2-\omega)-q^2 ) f(x_k) \notag\\
& \overset{\eqref{LowerBoundOfFunctions}}\leq & \|x_k-x_*\|_{\mB}^2 - (\omega (2-\omega)-q^2 )\lambda_{\min}^+ \|x_k-x_*\|_{\mB}^2 \notag\\
& =& (1- (\omega (2-\omega)-q^2 )\lambda_{\min}^+ ) \|x_k-x_*\|_{\mB}^2.
\end{eqnarray}
By taking the final expectation of the tower rule \eqref{eq:tower3} to the above inequality:
\begin{eqnarray}
\Exp[\|x_{k+1}-x_*\|_{\mB}^2] & \leq& (1- (\omega (2-\omega)-q^2 )\lambda_{\min}^+ ) \Exp[\|x_k-x_*\|_{\mB}^2].
\end{eqnarray}
and the result is obtain by unrolling the last expression.
\end{enumerate}
\end{proof}
\newpage
\section{Useful Inequalities and Frequently Used Notation}
\label{some Tables}

 \begin{table}[H]
\begin{center}
 \begin{tabular}{ |p{5cm}||p{3cm}|p{4cm}|p{3cm}|  }
 \hline
 \multicolumn{4}{|c|}{Useful inequalities} \\
 \hline
 Inequalities (Full names) & Abbreviations & Formula  & Assumptions\\
 \hline
 \hline
Jensen Inequality& \textit{Jensen } & $f[\Exp(x)]\leq \Exp[f(x)]$ &  f is convex\\
 \hline
 Conditioned Jensen &   \textit{cond. Jensen }  & $f(\Exp[x \;|\; s]) \leq \Exp[f(x) \;|\; s]$ &  f is convex\\
  \hline
  Cauchy-Swartz (B-norm)  & \textit{ C.S } &  $| \langle a,b \rangle_{\bB}| \leq \|a\|_{\bB} \|b\|_{\bB}$ & $a,b \in \R^n$ \\
   \hline
 Variance Inequality &  \textit{ V.I } & $(\Exp[X])^2\leq \Exp[X^2]$ & $X$ random variable\\
  \hline
\end{tabular}
\end{center}
 \caption{Popular inequalities with abbreviations and formulas.}
\label{inequalities}
\end{table}

\begin{table}[!h]
\begin{center}
\begin{tabular}{|c|l|c|}
 \hline
 \multicolumn{2}{|c|}{{\bf The Basics}}\\
 \hline
$\mA, b$    & $m \times n$ matrix and $m\times 1$ vector defining the system $\mA x =b$\\
$\cL$ &  $\{x\;:\; \mA x = b\}$ (solution set of the linear system) \\
$\mB$    & $n \times n$ symmetric positive definite matrix \\
$\langle x, y \rangle_{\mB}$ & $x^\top \mB y$ ($\mB$-inner product) \\
$\|x\|_{\mB}$ & $\sqrt{\langle x, x \rangle_{\mB}}$ ($\mB$-norm)  \\
$\mM^{\dagger}$ & Moore-Penrose pseudoinverse of matrix $\mM$  \\
$\mS$    & a random real matrix with $m$ rows  \\
$\cD$    & distribution from which matrix $\mS$ is drawn ($\mS\sim \cD$) \\ 
$\mH$ & $\mS (\mS^\top \mA \mB^{-1} \mA^\top \mS)^{\dagger} \mS^\top$ \\
$\mZ$ & $\mA^\top \mH \mA$\\
$\range{\mM}$ & range space of matrix $\mM$  \\
$\kernel{\mM}$ & null space of matrix $\mM$  \\
$\Prob(\cdot)$ & probability of an event\\
$\Exp[\cdot]$ & expectation\\
 \hline
 \multicolumn{2}{|c|}{{\bf Projections}}\\
  \hline
$\Pi_{\cL,\mB}(x)$ & projection of $x$ onto $\cL$ in the $\mB$-norm\\
$\mB^{-1}\mZ$ & projection matrix, in the $\mB$-norm, onto $\range{\mB^{-1}\mA^\top \mS}$ \\
 \hline
 \multicolumn{2}{|c|}{{\bf Optimization} }\\
 \hline
$\cX$ &  set of minimizers of $f$  \\ 
  $x_*$ & a point in $\cL$\\
$f_{\mS}$, $\nabla f_{\mS}$, $\nabla^2 f_{\mS}$ & stochastic function, its gradient and Hessian \\
$\cL_{\mS}$ & $\{x\;:\; \mS^\top \mA x = \mS^\top b\}$ (set of minimizers of $f_{\mS}$)  \\
$f$ & $\Exp[f_{\mS}]$\\
$\nabla f$ & gradient of $f$ with respect to  the $\mB$-inner product \\
$\nabla^2 f$ & $\mB^{-1}\Exp[\mZ]$ (Hessian of $f$ in the $\mB$-inner product) \\
 \hline
 \multicolumn{2}{|c|}{{\bf Eigenvalues} }\\
 \hline
 $\mW$ & $\mB^{-1/2}\Exp[\mZ]\mB^{-1/2}$ (psd matrix with the same spectrum as $\nabla^2 f$)\\
$\lambda_1,\dots,\lambda_n$ & eigenvalues of $\mW$\\
$\lambda_{\max}, \lambda_{\min}^+$ & largest and smallest nonzero eigenvalues of $\mW$\\
 \hline
 \multicolumn{2}{|c|}{{\bf Algorithms} }\\
 \hline
$\omega$ & relaxation parameter / stepsize  \\
$\epsilon_k$ & Inexactness error \\
$q$ & Inexactness parameter\\
$\rho$ & $1-\omega(2-\omega) \lambda_{\min}^+$\\
\hline
\end{tabular}
\end{center}
\caption{Frequently used notation.}
\label{tbl:notation}
\end{table}

\end{document}